\documentclass[11pt,a4paper,leqno]{amsart}
\usepackage{version}
\usepackage[utf8]{inputenc}
\usepackage{amsmath,amssymb,amsfonts,amsthm,yfonts}
\usepackage{latexsym,epsfig,graphicx,subfigure}
\usepackage{calrsfs}
\usepackage{esint}
\usepackage{hyperref}
\usepackage[parfill]{parskip}	
\usepackage{mathtools}
\usepackage{bbm}
\usepackage{enumitem}

\newcommand*\mathinhead[2]{\texorpdfstring{$#1$}{#2}}
\newcommand*{\dd}{\mathop{}\!\mathrm{d}}

\newcommand*{\ox}{\overline{x}}
\newcommand*{\oy}{\overline{y}}
\newcommand*{\oz}{\overline{z}}

\hypersetup{
	colorlinks=true,
	linkcolor=magenta,
	filecolor=black,      
	urlcolor=blue,
	citecolor=cyan,
}

\numberwithin{equation}{section}

\newtheoremstyle{note}
{8pt}
{1pt}
{\itshape}
{}
{\bfseries}
{.}
{.5em}
{}
\theoremstyle{note}
\newtheorem*{thm*}{Theorem}
\newtheorem{thm}{Theorem}[section]
\newtheorem{lem}[thm]{Lemma}

\newtheoremstyle{note2}
{8pt}
{1pt}
{}
{}
{\bfseries}
{.}
{.5em}
{}
\theoremstyle{note2}
\newtheorem{defn}{Definition}[section]

\newtheoremstyle{note3}
{8pt}
{1pt}
{}
{}
{\itshape}
{.}
{.5em}
{}
\theoremstyle{note3}
\newtheorem{rem}{Remark}[section]

\DeclareMathAlphabet{\pazocal}{OMS}{zplm}{m}{n}

\begin{document}
	\title[On fractional parabolic BMO and \mathinhead{\text{Lip}_{\alpha}}{} caloric capacities]{On fractional parabolic BMO and \mathinhead{\text{Lip}_{\alpha}}{} caloric capacities}
	\author{Joan Hernández, Joan Mateu, and Laura Prat}\thanks{Joan Hernández, Joan Mateu and Laura Prat have been supported by PID2020-114167GB-I00 (Ministerio de Ciencia e Innovación, Spain). In addition, Joan Mateu and Laura Prat have been partially supported by 2021SGR-00071 (Departament de Recerca i Universitats, Catalonia).}
	
	\begin{abstract}
		
		In the present paper we characterize the removable sets for solutions of the fractional heat equation satisfying some parabolic BMO or $\text{Lip}_\alpha$ normalization conditions. We do this by introducing associated fractional caloric capacities, that we show to be comparable to a certain parabolic Hausdorff content.
		
		
	\end{abstract}
	
	\maketitle
	
	\section{Introduction}
	\label{sec1}
	
	In this paper we characterize removable sets for solutions of the fractional heat equation under certain parabolic BMO or Lip$_\alpha$ normalization conditions. Our main motivation stems from the results obained in  \cite{MaPrTo} and \cite{MaPr}. 
	The study conducted by Mateu, Prat and Tolsa in \cite{MaPrTo} explores removable singularities for regular $(1,1/2)-$Lipschitz solutions of the classical heat equation,  associated with the operator $$\Theta:=(-\Delta_x)+\partial_t,\;\mbox{ where }\;(x,t)\in\mathbb{R}^n\times \mathbb{R}.$$ Here $(-\Delta_x)$ is the usual Laplacian, computed with respect to the spatial variables. In \cite{MaPr},  the authors extend the study to the fractional heat equation, defined via the $s$-heat operator $$\Theta^s:=(-\Delta_x)^s+\partial_t,\; s\in(0,1].$$ For $s=1$, we recover the classical heat equation, while for $s<1$,  the operator $(-\Delta_x)^s$, commonly referred to as the $s$-fractional Laplacian or $s$-Laplacian, requires an alternative definition. It is typically introduced through its Fourier transform:
	\begin{equation*}
		\widehat{(-\Delta_x)^{s}}f(\xi,t)=|\xi|^{2s}\widehat{f}(\xi,t),
	\end{equation*}
	or via the singular integral representation
	\begin{align*}
		(-\Delta_x)^{s} f(x,t)& =c_{n,s} \;\text{p.v.}\int_{\mathbb{R}^n}\frac{f(x,t)-f(y,t)}{|x-y|^{n+2s}}\dd y  \\
		&= c_{n,s}' \int_{\mathbb{R}^n}\frac{f(x+y,t)-2f(x,t)+f(x-y,t)}{|y|^{n+2s}}\dd y.
	\end{align*}
	These representations are equivalent and highlight that $(-\Delta_x)^s$ is no longer a local operator and that as $s\to 1$,  one recovers the expression of $(-\Delta_x)$. The reader may consult \cite[\textsection{3}]{DPV} or \cite{St} for details on the properties of $(-\Delta_x)^{s}$.
    
	To study removable sets in this context, we introduce the $s$-\textit{parabolic distance} between two points $\ox:=(x,t),\,\oy:=(y,\tau)$ in $\mathbb{R}^{n+1}$, defined as
	\begin{equation*}
		|\ox-\oy|_{p_s}=\text{dist}_{p_s}(\ox,\oy):=\max\big\{ |x-y|,|t-\tau|^{\frac{1}{2s}} \big\}, \qquad \text{for }\, 0<s\leq 1.
	\end{equation*}
	This leads naturally to the notions of $s$-parabolic cubes and $s$-parabolic balls. We convey that $B(\ox,r)$ will be the $s$-\textit{parabolic ball} centered at $\ox$ with radius $r$, where the spatial coordinates are contained in a Euclidean ball $B_1$ of radius $r$, while the temporal coordinate lies in a real interval $I$ of length $(2r)^{2s}$. On the other hand,  an $s$-\textit{parabolic cube}  $Q$ of side length $\ell$ is a set of the form
	\begin{equation*}
		I_1\times\cdots\times I_n\times I_{n+1},
	\end{equation*}
	where $I_1,\ldots, I_n$ are intervals of length $\ell$, while $I_{n+1}$ is another interval of length $\ell^{2s}$. We write $\ell(Q)=\ell$. 
	\textbf{}\newline
		Let us recall that a function $f$ is said to be $(1,1/2)$-Lipschitz regular if, as precised in \cite{MaPrTo}, it is such that
\begin{equation}
	\label{eq1.1}
	\|\nabla_x f\|_{L^{\infty}(\mathbb{R}^{n+1})}<\infty, \hspace{1cm} \|\partial_t^{1/2} f\|_{\ast, p_1}<\infty.
\end{equation}
Here, the norm $\|\cdot\|_{\ast, p_1}$ stands for the usual BMO norm of $\mathbb{R}^{n+1}$ but computed with respect to $1$-\textit{parabolic} cubes. As shown by Hofmann and Lewis \cite[Lemma 1]{Ho}, \cite[Theorem 7.4]{HoL}, such functions satisfy
\begin{equation*}
	\|f\|_{\text{Lip}_{1/2},t}:=\sup_{\substack{x\in \mathbb{R}^n\\ t,u\in \mathbb{R}, t\neq u}} \frac{|f(x,t)-f(x,u)|}{|t-u|^{1/2}}\lesssim \|\nabla_x f\|_{L^\infty(\mathbb{R}^{n+1})}+\|\partial_t^{1/2}f\|_{\ast,p_1}.
\end{equation*}
Thus a $(1,1/2)-$Lipschitz function is Lipschitz in the spatial variables and 1/2-Lipschitz in time. This explains the term  $(1,1/2)-$Lipschitz caloric capacity introduced in \cite{MaPrTo}, defined for a compact set $E\subset\mathbb{R}^{n+1}$ as
$$\Gamma_{\Theta}(E)=\sup\{|\langle \Theta f,1\rangle|\textbf{}\},$$
the supremum taken over all $(1,1/2)-$Lipschitz regular functions $f$ satisfying the heat equation on $\mathbb{R}^{n+1}\setminus E$ and with the norms in \eqref{eq1.1} smaller or equal than one.
A key result in \cite{MaPrTo} establishes the equivalence between the removability of the compact set $E$ for $(1,1/2)-$Lipschitz  solutions to the heat equation and  the fact that $\Gamma_{\Theta}(E)$ vanishes. 

In this paper, we aim at characterizing different variants of the previous Lipschitz caloric capacity, replacing the previous estimates with parabolic BMO or a $\text{Lip}_{\alpha}$ conditions for $\nabla_x f$ and $\partial_t^{1/2}f$. More generally, we analyze removable sets for solutions of the $s$-fractional heat equation with $s$-parabolic gradient  $(\nabla_x,\partial_t^{\frac{1}{2s}})$ satisfying either an $s$-parabolic $\text{BMO}$ or $\text{Lip}_{\alpha}$ condition. The reader who is not familiar with the notion of removability may conceive removable sets as those which ``do not matter'' when solving the $\Theta^s$-equation, $0<s\leq 1$. This has to be understood in the sense that any solution defined on their complement that satisfies the above $(1,\frac{1}{2s})$-gradient estimates, can be extended to verify the $\Theta^s$-equation throughout the entire domain, including the set itself.

Our main result characterizes removability in terms of two different capacities: one requiring solutions to the $\Theta^s$-equation satisfy $s$-parabolic BMO estimates, and another one requiring solutions satisfy $s$-parabolic $\text{Lip}_\alpha$ bounds.  These capacities, denoted by $\Gamma_{\Theta^s,\ast}$ and $\Gamma_{\Theta^s,\alpha}$ respectively, are related to certain $s$-parabolic Hausdorff content $ \pazocal{H}^{m}_{\infty,p_s}$, which is defined as in the Euclidean case (see \cite{Mat}, for instance), just replacing the Euclidean distance by the parabolic distance introduced above. Our main result reads as follows:
	\begin{thm*}
		Let $s\in(1/2,1]$, $\alpha\in(0,1)$ and $E\subset \mathbb{R}^{n+1}$ compact set. Then,
		\begin{align*}
			\Gamma_{\Theta^s,\ast}(E) &\approx_{n,s} \pazocal{H}^{n+1}_{\infty,p_s}(E),\\ \text{if $\alpha <2s-1$,} \quad \Gamma_{\Theta^s,\alpha}(E) &\approx_{n,s,\alpha} \pazocal{H}^{n+1+\alpha}_{\infty,p_s}(E).
		\end{align*}
		Moreover, the nullity of these capacities is equivalent to the removability of the corresponding compact set for solutions satisfying $(1,\frac{1}{2s})$-gradient estimates in either $s$-parabolic  $\text{\normalfont{BMO}}$ or $\text{\normalfont{Lip}}_{\alpha}$, assuming $\alpha<2s-1$ in the latter case. 
	\end{thm*}
	
	We further study the same type of question for a generalization of the capacities presented by Mateu and Prat in \cite[\textsection 4 \& \textsection 7]{MaPr}. That is, we will ask for a characterization of removable sets for solutions to the $\Theta^s$-equation satisfying conditions of the form
	\begin{equation*}
		\|(-\Delta)^{\sigma} f\|  < \infty, \hspace{0.75cm} \|\partial_t^{\sigma/s}f\| <\infty, \qquad \text{$s\in (0,1]$ and $\sigma\in[0,s)$}.
	\end{equation*}
	Here the symbols $\|\cdot\|$ can refer both to $s$-parabolic $\text{BMO}$ norms or both to $s$-parabolic $\text{Lip}_{\alpha}$ seminorms. We prove the following result:
	\begin{thm*}
		For any $s\in (0,1]$, $\sigma\in[0,s)$, $\alpha\in(0,1)$ and $E\subset \mathbb{R}^{n+1}$ compact set,
		\begin{align*}
			\gamma_{\Theta^s,\ast}^{\sigma}(E) &\approx_{n,s,\sigma} \pazocal{H}^{n+2\sigma}_{\infty,p_s}(E),\\
			\text{if $\alpha <2s-2\sigma$,} \quad \gamma_{\Theta^s,\alpha}^{\sigma}(E) &\approx_{n,s,\sigma,\alpha} \pazocal{H}^{n+2\sigma+\alpha}_{\infty,p_s}(E).
		\end{align*}
		The nullity of these capacities is equivalent to the removability of the corresponding compact set for solutions satisfying $(\sigma, \sigma/s)$-Laplacian estimates in either $s$-parabolic $\text{\normalfont{BMO}}$ or $\text{\normalfont{Lip}}_{\alpha}$, assuming $\alpha<2s-2\sigma$ in the latter case.
	\end{thm*}
	The previous study has been motivated by the one carried out for the BMO variant of analytic capacity by Kaufman \cite{K} and Verdera \cite{Ve} (for a brief overview the reader may consult \cite[\textsection 13.5.1]{AsIM}); and that for the $\text{Lip}_\alpha$ variant of the same capacity in the direction presented by Mel'nikov \cite{Me} or O'Farrell \cite{O}. We remark that the results presented here also generalize those of \cite[\textsection 5 \& \textsection 6]{He}.
    
	A brief overview of the paper for the reader: sections \textsection \ref{sec2.1} and \textsection \ref{sec2.2} focus on kernel estimates and growth estimates for the so-called \textit{admissible} functions. Moreover, in \textsection \ref{sec2.3} we deduce some important properties regarding potentials defined against positive Borel measures with some growth properties. Finally, in \textsection \ref{sec2.4}, we define all the different capacities and characterize them in terms of certain $s$-parabolic Hausdorff contents.
    
	\textit{About the notation}: Constants appearing in the sequel may depend on the dimension of the ambient space and the  parameter $s$, and their value may change at different occurrences. They will be frequently denoted by the letters $c$ or $C$. The notation $A\lesssim B$ means that there exists such a constant, say $C$, so that $A\leq CB$. Moreover, $A\approx B$ is equivalent to $A\lesssim B \lesssim A$. Also, $A \simeq B$ will mean $A= CB$. If the reader finds expressions of the form $\lesssim_{\beta}$ or $\approx_\beta$, for example, it will mean that the implicit constants depend on $n,s$ and $\beta$.
    
	Since Laplacian operators (fractional or not) will frequently appear in our discussion and will be always taken with respect to spatial variables, we will write:
	\begin{equation*}
		(-\Delta)^s:=(-\Delta_x)^s, \qquad s\in(0,1], \quad \text{and we convey $(-\Delta)^0:=\text{Id}$}.
	\end{equation*}
	We will also write $\|\cdot\|_\infty:=\|\cdot\|_{L^\infty(\mathbb{R}^{n+1})}$. Finally, we stress that an important parameter which will play a fundamental role in \textsection \ref{sec2.1} is $2\zeta:=\min\{1,2s\}$.
    
	\section{Basic notation and kernel estimates}
	\label{sec2.1}
    
	We begin by noticing that the $s$-\textit{parabolic distance} between $\ox:=(x,t),\,\oy:=(y,\tau)$ in $\mathbb{R}^{n+1}$, defined in the introduction as 
	\begin{equation*}
		|\ox-\oy|_{p_s}=\text{dist}_{p_s}(\ox,\oy):=\max\big\{ |x-y|,|t-\tau|^{\frac{1}{2s}} \big\}, \qquad \text{for }\, 0<s\leq 1,
	\end{equation*}
	is, in fact, equivalent to
	\begin{equation*}
		\text{dist}_{p_s}(\ox,\oy) \approx \big(|x-y|^2+|t-\tau|^{1/s}\big)^{1/2}.
	\end{equation*}
	The $s$\textit{-parabolic dilation} of factor $\lambda>0$, written $\delta_\lambda$, is given by 
	\begin{equation*}
		\delta_\lambda(x,t)=\big(\lambda x, \lambda^{2s} t\big
		).
	\end{equation*}
	To ease notation, since we will always work with $s$-parabolic distances, we will write $\lambda Q$ to denote $\delta_{\lambda}(Q)$, the $s$-parabolic cube concentric with $Q$ of side length $\lambda \ell(Q)$.
    
	As the reader may suspect, the notion of $s$\textit{-parabolic BMO space}, $\text{BMO}_{p_s}$, refers to the space of usual BMO functions (strictly, equivalence classes of functions where constants are identified as 0) obtained by replacing Euclidean cubes by $s$-parabolic ones. 
	Similarly, a function $f:\mathbb{R}^{n+1}\to\mathbb{R}$ is said to be $s$\textit{-parabolic} $\textit{\text{\textit{Lip}}}_\alpha$ for some $0<\alpha<1$, shortly $\text{\textit{Lip}}_{\alpha,p_s}$, if
	\begin{equation*} \|f\|_{\text{Lip}_\alpha,p_s}:=\sup_{\overline{x},\overline{y}\in\mathbb{R}^{n+1}}\frac{|f(\overline{x})-f(\overline{y})|}{|\overline{x}-\overline{y}|_{p_s}^{\alpha}}\lesssim 1.
	\end{equation*} 
	For each $s\in(0,1]$, the fundamental solution $P_s(x,t)$ to the $\Theta^s$-equation, i.e. that associated with the operator
	\begin{equation*}
		\Theta^s:=(-\Delta)^s+\partial_t,
	\end{equation*} 
	is the inverse spatial Fourier transform of $e^{-4\pi^2 t|\xi|^{2s}}$ for $t>0$, and it equals $0$ if $t\leq 0$. For the special case $s=1$, we retrieve the classical \textit{heat kernel}, given by:
	\begin{equation*}
		W(\ox):=P_1(\ox)=c_nt^{-\frac{n}{2}}\phi_{n,1}(|x|t^{-\frac{1}{2}}), \qquad \text{if }\, t>0,
	\end{equation*}
	where $\phi_{n,1}(\rho):=e^{-\rho^2/4}$, independent of $n$. Although the expression of $P_s$ is not explicit in general, Blumenthal and Getoor \cite[Theorem 2.1]{BG} established that for $s<1$,
	\begin{equation}
		\label{eq2.1}
		P_s(\ox)=c_{n,s}\,t^{-\frac{n}{2s}}\phi_{n,s}\big( |x|t^{-\frac{1}{2s}} \big)\chi_{t>0},
	\end{equation}
	Here, $\phi_{n,s}$ is a smooth function, radially decreasing and satisfying, for $0<s<1$,
	\begin{equation}
		\label{eq2.2}
		\phi_{n,s}(\rho)\approx \big(1+\rho^2\big)^{-(n+2s)/2},
	\end{equation}
	being an exact equality if $s=1/2$ \cite{Va}. Therefore,
	\begin{equation*}
		P_s(\ox)\approx \frac{t}{|\ox|_{p_s}^{n+2s}}\chi_{t>0}.
	\end{equation*}
	The function $\phi_{n,s}$ is tightly related to the Fourier transform of $e^{-4\pi^2|\xi|^{2s}}$. Indeed, taking the spatial Fourier transform in both sides of identity \eqref{eq2.1}, we get
	\begin{equation*}
		e^{-4\pi^2 t|\xi|^{2s}}=c_{n,s}\,t^{-\frac{n}{2s}}\big[\phi_{n,s}\big(|\,\cdot \,|\,t^{-\frac{1}{2s}}\big)\big]^\wedge(\xi).
	\end{equation*}
	Recall that for $\lambda>0$, the dilation $f_{\lambda} := f(\lambda x)$ satisfies $\widehat{f_\lambda}(\xi)=\lambda^{-n}\widehat{f}(\lambda^{-1}\xi)$. Then,
	\begin{equation*}
		e^{-4\pi^2 t|\xi|^{2s}} = c_{n,s}\,\widehat{\phi_{n,s}\big(|\,\cdot\,|\big)}\big(\xi t^{\frac{1}{2s}}\big), \hspace{0.5cm} \text{that implies} \hspace{0.5cm} e^{-4\pi^2 |\xi|^{2s}}\simeq\widehat{\phi_{n,s}\big(|\,\cdot\,|\big)}(\xi).
	\end{equation*}
	The above relations will allow us to obtain explicit bounds for the derivatives of $\phi_{n,s}$. 
	Let us present our first lemma. Although it can be deduced straightforwardly from \cite[Theorem 1.1]{GrT},  we shall give a detailed proof for the sake of clarity and completeness.
	\begin{lem}
		\label{lem2.1}
		Let $s\in (0,1]$ and $\beta\in (0,1)$. We define the following function in $\mathbb{R}^n$:
		\begin{align*}
			\psi_{n,s}^{(\beta)}(x):=(-\Delta)^\beta\phi_{n,s}\big( |x| \big).
		\end{align*}
		Then,
		\begin{enumerate}
			\item[1.] $\phi'_{n,s}(\rho)\simeq -\rho \, \phi_{n+2,s}(\rho)$.
			\item[2.] $|\psi_{n,s}^{(\beta)}(x)|\lesssim_{\beta} \big( 1+|x|^2 \big)^{-(n+2\beta)/2}$.
			\item[3.] $\nabla \psi_{n,s}^{(\beta)}(x) \simeq - x\,\psi^{(\beta)}_{n+2,s}(x)$.
		\end{enumerate}
	\end{lem}
	\begin{proof}
		We begin by proving \textit{1} for $s<1$ (the case $s=1$ is trivial). To do so, we will use the explicit integral representation for the inverse Fourier transform of a radial function in \cite[\textsection B.5]{Gr} or \cite[\textsection{IV.I}]{StW}. Applying it to the Fourier transform $e^{-4\pi^2|\xi|^{2s}}$ we get
		\begin{equation*}
			\phi_{n,s}(|z|) = 2\pi|z|^{1-n/2}\int_0^\infty e^{-4\pi^2r^{2s}}r^{n/2}J_{n/2-1}(2\pi r|z|)\dd r, \quad \text{for any $z\in\mathbb{R}^n\setminus{\{0\}}$},
		\end{equation*}
		where $J_k$ is the classical Bessel function of order $k$ \cite[\textsection 9]{AS}. Since we are interested in the derivatives of $\phi_{n,s}$ as a radial real variable function, let us rewrite the previous expression in terms of $\rho\in(0,\infty)$ so that it reads as
		\begin{equation}
			\label{eq2.3}
			\phi_{n,s}(\rho) = 2\pi\rho^{1-n/2}\int_0^\infty e^{-4\pi^2r^{2s}}r^{n/2}J_{n/2-1}(2\pi r\rho) \dd r.
		\end{equation}
		Therefore, to estimate the derivatives of $\phi_{n,s}$ we need to determine first if we can differentiate under the integral sign. To that end, we use the following recurrence relation for classical Bessel functions \cite[\textsection{9.1.27}]{AS},
		\begin{align*}
			J_{k}'(x)= \frac{k}{x}J_{k}(x)-J_{k+1}(x).
		\end{align*}
		This recurrence formula together with \eqref{eq2.3} remain valid for the case $k=-1/2$, conveying that $J_{-1/2}(x)=\sqrt{\frac{2}{\pi x}}\cos{x}$. In our case these imply
		\begin{align*}
			\partial_\rho J_{n/2-1}(2\pi r\rho)&=\bigg(\frac{n}{2}-1\bigg)\rho^{-1}J_{n/2-1}(2\pi r\rho)-2\pi r\,J_{n/2}(2\pi r\rho).
		\end{align*}
		If we differentiated under the integral sign in \eqref{eq2.3}, we would get integrands of the form
		\begin{align*}
			&e^{-r^{2s}}r^{n/2}J_{n/2-1}(2\pi r\rho), \hspace{0.75cm} e^{-r^{2s}}r^{n/2+1}J_{n/2}(2\pi r\rho).
		\end{align*}
		Notice that both are bounded by integrable functions in the domain of integration, locally for each $\rho>0$ (by the boundedness of the functions $J_k$ for $n>1$, and by that of $\cos{x}$ if $n=1$). Hence, we can indeed differentiate under the integral sign to compute $\phi'_{n,s}$, obtaining the desired result:
		\begin{align*}
			\phi_{n,s}'(\rho)&= 2\pi \bigg[ \bigg( 1-\frac{n}{2} \bigg)\rho^{-n/2}\int_0^\infty e^{-4\pi^2r^{2s}}r^{n/2}J_{n/2-1}(2\pi r\rho) \dd r\\
			&\hspace{3cm}+\rho^{1-n/2}\partial_\rho\bigg( \int_0^\infty e^{-4\pi^2r^{2s}}r^{n/2}J_{n/2-1}(2\pi r\rho)\dd r \bigg)\bigg]\\
			&= 2\pi \bigg[ \bigg( 1-\frac{n}{2} \bigg)\rho^{-n/2}\int_0^\infty e^{-4\pi^2r^{2s}}r^{n/2}J_{n/2-1}(r\rho) \dd r\\
			&\hspace{1.3cm}\rho^{1-n/2}\bigg(\frac{n}{2}-1\bigg)\rho^{-1}\bigg( \int_0^\infty e^{-4\pi^2r^{2s}}r^{n/2}J_{n/2-1}(2\pi r\rho)\dd r \bigg)\\
			&\hspace{1.15cm}-2\pi\rho\,\rho^{1-(n+2)/2} \int_{0}^\infty e^{-4\pi^2r^{2s}}r^{(n+2)/2}J_{(n+2)/2-1}(r\rho)\dd r\bigg] = -2\pi\rho\,\phi_{n+2,s}(\rho).
		\end{align*}
		Next we prove statement \textit{2}. Observe that for $s\in(0,1]$ and $\beta\in(0,1)$, we have $\widehat{\psi_{n,s}^{(\beta)}}(\xi)=|\xi|^{2\beta}e^{-4\pi^2|\xi|^{2s}}$, 
		which is an integrable function, and thus $\psi_{n,s}^{(\beta)}$ is bounded (in fact, since the product of $\widehat{\psi_{n,s}^{(\beta)}}$ by any polynomial is also integrable, we infer that $\psi_{n,s}^{(\beta)}$ is smooth). By the integral representation formula for inverse Fourier transforms of radial functions,
		\begin{equation}
			\label{eq2.4}
			\psi_{n,s}^{(\beta)}(x) = 2\pi|x|^{1-n/2}\int_0^\infty e^{-4\pi^2r^{2s}}r^{n/2+2\beta}J_{n/2-1}(2\pi r|x|)\dd r, \qquad x\in\mathbb{R}^n\setminus\{0\}.
		\end{equation}
		Now, we apply \cite[Lemma 1]{PruTa} to deduce the desired decaying property $|\psi_{n,s}^{(\beta)}(x)|=O\big(|x|^{-n-2\beta}\big)$, for $|x|$ large. Hence, since $\psi_{n,s}^{(\beta)}$ is bounded, we deduce the desired bound $|\psi_{n,s}^{(\beta)}(x)|\lesssim_\beta \big( 1+|x|^2 \big)^{-(n+2\beta)/2}$.
        
		We are left to control the norm of $\nabla\psi_{n,s}^{(\beta)}$, provided the latter is well-defined. We claim that this is the case, since we can differentiate under the integral sign in \eqref{eq2.4}. Indeed, by the recurrence relation satisfied by the derivatives of $J_k$ we get
		\begin{align*}
			|\nabla_xJ_{n/2-1}(r|x|)|=\bigg\rvert \bigg(\frac{n}{2}-1\bigg)\frac{1}{|x|}J_{n/2-1}(2\pi r|x|)-2\pi rJ_{n/2}(2\pi r|x|)\bigg\rvert.
		\end{align*}
		So the resulting integrands to study are terms of the form
		\begin{align*}
			&e^{-4\pi^2r^{2s}}r^{n/2+2\beta}|J_{n/2-1}(2\pi r|x|)|, \hspace{0.75cm} e^{-4\pi^2r^{2s}}r^{n/2+2\beta+1}|J_{n/2}(2\pi r|x|)|,
		\end{align*}
		both bounded by the integrable functions $C_1e^{-r^{2s}}r^{n/2+2\beta}$ and $C_2e^{-r^{2s}}r^{n/2+2\beta+1}$ for some constants $C_1,C_2$ depending on $n,s$ and $\beta$, and locally for each $x\in\mathbb{R}^n$ with $|x|>0$. Hence, we can differentiate under the integral sign in \eqref{eq2.4} and obtain
		\begin{align*}
			\nabla\psi_{n,s}^{(\beta)}(x)&= 2\pi\bigg[\bigg( 1-\frac{n}{2} \bigg)\frac{x}{|x|^{n/2+1}}\int_0^\infty e^{-4\pi^2r^{2s}}r^{n/2+2\beta}J_{n/2-1}(2\pi r|x|)\dd r\\
			&\hspace{3cm}+\bigg(\frac{n}{2}-1\bigg)\frac{x}{|x|^{n/2+1}}\int_0^\infty e^{-4\pi^2r^{2s}}r^{n/2+2\beta}J_{n/2-1}(2\pi r|x|)\dd r\\
			&\hspace{3cm} -2\pi\frac{x}{|x|^{n/2}}\int_0^\infty e^{-4\pi^2r^{2s}}r^{(n+2)/2+2\beta}J_{(n+2)/2-1}(2\pi r|x|)\dd r\bigg]\\
			&= -2\pi x\,\psi_{n+2,s}^{(\beta)}(x).
		\end{align*}
	\end{proof}
	Using the above lemma together with \eqref{eq2.2} we can estimate the derivatives of $\phi_{n,s}$ and $\psi_{n,s}^{(\beta)}$. In particular, the following relations hold:
	\begin{align}
		\label{eq2.5}
		\text{If } s<1\text{,}\quad \phi_{n,s}'(\rho) &\approx \frac{-\rho}{(1+\rho^2)^{(n+2s+2)/2}}, \qquad \phi_{n,s}''(\rho) \approx \frac{-1+(2\pi-1)\rho^2}{(1+\rho^2)^{(n+2s+4)/2}},\\
		\label{eq2.6}
		|\nabla \psi_{n,s}^{(\beta)}(x)|&\lesssim_{\beta} \frac{|x|}{(1+|x|^2)^{(n+2\beta+2)/2}}. 
	\end{align}
	
	\subsection{Estimates for \mathinhead{\nabla_xP_s}{} and \mathinhead{\Delta^\beta P_s}{}}
	
	We shall now present some growth estimates for the kernels $P_s$. Our first result provides bounds for $\nabla_xP_s$, $s\in(0,1)$. These estimates are analogous to those of \cite[Lemma 5.4]{MaPrTo} which cover the case $s=1$. In the forthcoming results, the parameter $2\zeta:=\min\{1,2s\}$ will play an important role. 
	\begin{thm}
		\label{C-Z_thm2}
		The following estimates hold for any $\ox\neq 0$ and $s\in(0,1)$:
		\begin{align*}
			|\nabla_xP_{s}(\ox)|&\lesssim \frac{|xt|}{|\ox|_{p_s}^{n+2s+2}}, \qquad |\Delta P_{s}(\ox)|\lesssim \frac{|t|}{|\ox|_{p_s}^{n+2s+2}}, \qquad |\partial_t\nabla_xP_{s}(\ox)| \lesssim  \frac{|x|}{|\ox|_{p_s}^{n+2s+2}}.
		\end{align*}
		The last bound is only valid for points with $t\neq 0$. Also, if $\ox'$ is such that $|\ox-\ox'|_{p_s}\leq |\ox|_{p_s}/2$,
		\begin{align*}
			|\nabla_x P_{s}(\ox)-\nabla_xP_{s}(\ox')|\lesssim \frac{|\ox-\ox'|_{p_s}^{2\zeta}}{|\ox|_{p_s}^{n+1+2\zeta}}.
		\end{align*}
	\end{thm}
	\begin{proof}
        To simplify the arguments below, we specify the dependence of $P_s$ with respect to $n$. Let us write $P_{s,n+1}$ to refer to the fundamental solution to the $\Theta^s$-equation in $\mathbb{R}^{n+1}$ and use the following abuse of notation: given $\ox=(x_1,\ldots,x_n,t)\in\mathbb{R}^{n+1}$, write
    	\begin{align*}
    		P_{s,n+3}(\ox)&:=P_{s,n+3}(x_1,\ldots,x_n,0,0,t),\\
    		P_{s,n+5}(\ox)&:=P_{s,n+5}(x_1,\ldots,x_n,0,0,0,0,t).
    	\end{align*}
		This way, we directly apply relations \eqref{eq2.1} and \eqref{eq2.5} to obtain for each $t>0$,
		\begin{align*}
			|\nabla_xP_s(\ox)| &\simeq t^{-\frac{n+1}{2s}}|\phi_{n,s}'(|x|t^{-\frac{1}{2s}})| \simeq |xP_{s,n+3}(\ox)| \approx \frac{|xt|}{|\ox|_{p_s}^{n+2s+2}}.
		\end{align*}
		The bounds for $\Delta P_{s,n+1}$ and $\partial_t\nabla_xP_{s,n+1}$ can be obtained from the previous result and \eqref{eq2.1}. Indeed,
        \begin{align*}
            |\Delta P_{s,n+1}(\ox)|&\simeq P_{s,n+3}(\ox)+|x|^2P_{s,n+5}(\ox) \lesssim \frac{|t|}{|\ox|_{p_s}^{n+2s+2}},\\
            |\partial_t\nabla_xP_{s,n+1}(\ox)|&\lesssim  \frac{|x|}{t} \Big( P_{s,n+3}(\ox)+|x|^2P_{s,n+5}(\ox) \Big) \lesssim  \frac{|x|}{|\ox|_{p_s}^{n+2s+2}}.
        \end{align*}
        For the final estimate, we recover the notation $P_s:=P_{s,n+1}$. Let $\ox'=(x',t')\in\mathbb{R}^{n+1}$ with $|\ox-\ox'|_{p_s}\leq |\ox|_{p_s}/2$ and use the definition of $\text{dist}_{p_s}$ to obtain
		\begin{equation}
			\label{new_eq2.7}
			|\ox|_{p_s}\leq 2 |\ox'|_{p_s} \quad \text{and} \quad |x'|\geq |x|-\frac{|\ox|_{p_s}}{2}.
		\end{equation}
		Put $\widehat{x}=(x',t)$ and write
		\begin{equation*}
			|\nabla_xP_s(\ox)-\nabla_xP_s(\ox')|\leq |\nabla_xP_s(\ox)-\nabla_xP_s(\widehat{x})|+|\nabla_xP_s(\widehat{x})-\nabla_xP_s(\ox')|.
		\end{equation*}
		We observe that the first term in the above inequality satisfies the desired bound,
		\begin{equation*}
			|x-x'|\sup_{\xi\in[x,x']}|\Delta P_s(\xi,t)|\lesssim \frac{|x-x'|}{|\ox|_{p_s}^{n+2}}\leq \frac{|\ox-\ox'|_{p_s}^{2\zeta}}{|\ox|_{p_s}^{n+1+2\zeta}}\bigg( \frac{|\ox-\ox'|_{p_s}}{|\ox|_{p_s}} \bigg)^{1-2\zeta}\leq \frac{|\ox-\ox'|_{p_s}^{2\zeta}}{|\ox|_{p_s}^{n+1+2\zeta}}.
		\end{equation*}
		Regarding the second term, assume without loss of generality $t>t'$. If $t'>0$, use $|\ox'|_{p_s}\geq |\ox|_{p_s}/2$ so that we also have
		\begin{align*}
			|t-t'|\sup_{\tau\in[t,t']}|\partial_t\nabla_xP_s(x',\tau)|\lesssim \frac{|t-t'|}{|\ox|_{p_s}^{n+2s+1}}\leq \frac{|\ox-\ox'|_{p_s}^{2\zeta}}{|\ox|_{p_s}^{n+1+2\zeta}}\Bigg(  \frac{|\ox-\ox'|_{p_s}}{|\ox|_{p_s}}\Bigg)^{2s-2\zeta} \lesssim \frac{|\ox-\ox'|_{p_s}^{2\zeta}}{|\ox|_{p_s}^{n+1+2\zeta}},
		\end{align*}
		If $t<0$ then $|\nabla_xP_s(\widehat{x})-\nabla_xP_s(\ox')|=0$, and the estimate becomes trivial. Then, we are left to study the case $t>0$ and $t'<0$. These two conditions imply that the $p_s$-ball
		\begin{equation*}
			B(\ox):=\bigg\{\oy\in \mathbb{R}^{n+1} \,: \, |\ox-\oy|_{p_s}\leq \frac{|\ox|_{p_s}}{2}\bigg\}\ni \ox'
		\end{equation*}
		intersects the hyperplane $\{t=0\}$. Since the radius of $B(\ox)$ also depends on $\ox$, the previous property imposes
		the following condition over $\ox$,
		\begin{equation*}
			t^{1/s}\leq \frac{x_1^2+\cdots +x_n^2}{3}, \quad \text{that is} \quad t^{\frac{1}{2s}}\leq \frac{|x|}{\sqrt{3}},
		\end{equation*}
		which is attained if the point $(x,0)$ belongs to $\partial B(\ox)$. Therefore $|\ox|_{p_s}:=\max\big\{|x|,t^{\frac{1}{2s}}\big\} = |x|$, so by \eqref{new_eq2.7} we get $|x'|\geq |x|/2$, and this in turn implies
		\begin{equation}
			\label{eq2.8}
			\frac{|\ox|_{p_s}}{2}\leq |\ox'|_{p_s}\leq |\ox-\ox'|_{p_s}+|\ox|_{p_s}\leq \frac{3|x|}{2}\leq 3 |x'|.
		\end{equation}
		Using this last inequality we can finally conclude:
		\begin{align*}
			|\nabla_xP_s(\widehat{x})-\nabla_xP_s(\ox')|&=|\nabla_xP_s(x',t)-\nabla_xP_s(x',0)|\lesssim |t|\sup_{\tau\in(0,t]}|\partial_t\nabla_xP_s(x',\tau)|\\
			&\lesssim \frac{|t|}{|x'|^{n+2s+1}} \lesssim \frac{|t|}{|\ox|_{p_s}^{n+2s+1}}\leq \frac{|t-t'|}{|\ox|_{p_s}^{n+2s+1}}\lesssim \frac{|\ox-\ox'|_{p_s}^{2\zeta}}{|\ox|_{p_s}^{n+1+2\zeta}}.
		\end{align*}
	\end{proof}
	
	\begin{thm}
		\label{lem2.3}
		Let $s\in(0,1]$ and $\beta,\gamma\in[0,1)$. Then, for any $\ox\neq 0$ we have,
		\begin{align*}
			\textit{1}\hspace{0.03cm}.&\;\;|(-\Delta)^{\beta}P_s(\ox)|\lesssim_{\beta} \frac{1}{|\ox|_{p_{s}}^{n+2\beta}},\\  \textit{2}\hspace{0.03cm}.&\;\;|(-\Delta)^{\gamma} (-\Delta)^{\beta}P_s(\ox)|\lesssim_{\beta,\gamma} \; \frac{1}{|\ox|_{p_s}^{n+2\beta+2\gamma}},  \hspace{0.5cm}  \textit{3}\hspace{0.03cm}.\;\;|\nabla_x (-\Delta)^{\beta}P_s(\ox)|\lesssim_\beta \frac{|x|}{|\ox|_{p_s}^{n+2\beta+2}}.
		\end{align*}
		Moreover, for any $\ox\neq(x,0)$,
		\begin{align*}
			\textit{4}\hspace{0.03cm}.\;\; |\partial_t (-\Delta)^{\beta}P_s(\ox)|\lesssim_{\beta} \frac{1}{|\ox|_{p_s}^{n+2\beta+2s}}.
		\end{align*}
		Finally, if $\ox'\in\mathbb{R}^{n+1}$ is such that $|\ox-\ox'|_{p_s}\leq |\ox|_{p_s}/2$,
		\begin{align*}
			\textit{5}\hspace{0.03cm}.\;\; |(-\Delta)^{\beta}P_s(\ox)-(-\Delta)^{\beta}P_s(\ox')|\lesssim_{\beta} \frac{|\ox-\ox'|_{p_s}^{2\zeta}}{|\ox|_{p_s}^{n+2\beta+2\zeta}}.
		\end{align*}
	\end{thm}
	
	\begin{proof}
		We shall also assume $\beta>0$, since the case $\beta=0$ is already covered in \cite[Lemma 2.2]{MaPr}. For the sake of notation, in this proof we will write $\phi:=\phi_{n,s}$ and $\psi:=\psi_{n,s}^{(\beta)}$, and we also set $K_\beta:=(-\Delta)^{\beta}P_s$. Let us begin by applying the integral representation of $K_{\beta}$ together with relation \eqref{eq2.1} to obtain for $t>0$,
		\begin{align*}
			K_{\beta}(x,t):=(-\Delta)^{\beta}P_s(x,t)&\simeq_{\beta} \text{p.v.} \int_{\mathbb{R}^n}\frac{P_s(x,t)-P_s(y,t)}{|x-y|^{n+2\beta}}\dd y\\
			&=t^{-\frac{n}{2s}}\, \text{p.v.} \int_{\mathbb{R}^n}\frac{\phi\big( |x|t^{-\frac{1}{2s}} \big)-\phi\big( |y|t^{-\frac{1}{2s}} \big)}{|x-y|^{n+2\beta}}\dd y\\
			&=t^{-\frac{n+2\beta}{2s}}\, \text{p.v.} \int_{\mathbb{R}^n}\frac{\phi\big( |x|t^{-\frac{1}{2s}} \big)-\phi(|z|)}{|xt^{-\frac{1}{2s}}-z|^{n+2\beta}}\dd z\\
			&=t^{-\frac{n+2\beta}{2s}}(-\Delta)^{\beta}\phi\big( |x|t^{-\frac{1}{2s}} \big)=t^{-\frac{n+2\beta}{2s}}\psi\big( xt^{-\frac{1}{2s}} \big)
		\end{align*}
		Using the estimate proved in Lemma \ref{lem2.1} for $\psi$ we deduce the desired bound:
		\begin{equation*}
			\big\rvert K_{\beta}(x,t) \big\rvert\lesssim_\beta \frac{t^{-\frac{n+2\beta}{2s}}}{\big( 1+|x|^2t^{-1/s} \big)^{(n+2\beta)/2}}=\frac{1}{\big( t^{1/s}+|x|^2 \big)^{(n+2\beta)/2}}\approx \frac{1}{|\ox|_{p_s}^{n+2\beta}}.
		\end{equation*}
		We shall continue by studying estimate \textit{2} in a similar way. Indeed,
		\begin{align*}
			(-\Delta)^{\gamma}K_\beta(x,t)&\simeq_{\gamma} \text{p.v.}\int_{\mathbb{R}^n}\frac{K_\beta(x,t)-K_\beta(y,t)}{|x-y|^{n+2\gamma}}\dd y\\
			&\simeq_{\beta}t^{-\frac{n+2\beta}{2s}}\,\text{p.v.}\int_{\mathbb{R}^n}\frac{\psi\big( xt^{-\frac{1}{2s}} \big)-\psi\big( yt^{-\frac{1}{2s}} \big)}{|x-y|^{n+2\gamma}}\dd y\\
			&=t^{-\frac{n+2\beta+2\gamma}{2s}}\, \text{p.v.}\int_{\mathbb{R}^n}\frac{\psi\big( xt^{-\frac{1}{2s}} \big)-\psi(z)}{|xt^{-\frac{1}{2s}}-z|^{n+2\gamma}}\dd z=t^{-\frac{n+2\beta+2\gamma}{2s}}(-\Delta)^{\gamma}\psi\big( xt^{-\frac{1}{2s}} \big).
		\end{align*}
		Set $\Psi:=(-\Delta)^{\gamma}\psi(\,\cdot\,)$ and notice that
		\begin{equation*}
			\widehat{\Psi}(\xi)=|\xi|^{2\gamma}|\xi|^{2\beta}e^{-4\pi^2|\xi|^{2s}}=|\xi|^{2\beta+2\gamma}e^{-4\pi^2|\xi|^{2s}}.
		\end{equation*}
		Thus, since $\widehat{\Psi}$ is integrable, $\Psi$ is the radial bounded function in $\mathbb{R}^n$ given by
		\begin{equation*}
			\Psi(z) = 2\pi|z|^{1-n/2}\int_0^\infty e^{-4\pi^2 r^{2s}}r^{n/2+2\beta+2\gamma}J_{n/2-1}(2\pi r|z|)\dd r,
		\end{equation*}
		By \cite[Lemma 1]{PruTa} $\Psi$ decays as
		\begin{equation*}
			|\Psi(z)|=O\big(|z|^{-n-2\beta-2\gamma}\big), \hspace{0.5cm} \text{for }\; |z| \;\text{ large}.
		\end{equation*}
		Therefore
		\begin{equation*}
			|\Psi(z)|\lesssim_{\beta,\gamma} \big( 1+|z|^2 \big)^{-(n+2\beta+2\gamma)/2}.
		\end{equation*}
		So analogously to the proof of \textit{1}, we deduce the desired result:
		\begin{equation*}
			\big\rvert (-\Delta)^\gamma K_\beta(x,t) \big\rvert\lesssim_{\beta,\gamma} \frac{t^{-\frac{n+2\beta+2\gamma}{2s}}}{\big( 1+|x|^2t^{-1/s} \big)^{(n+2\beta+2\gamma)/2}} \approx \frac{1}{|\ox|_{p_s}^{n+2\beta+2\gamma}}.
		\end{equation*}
		
		Regarding estimate \textit{3}, notice that
		\begin{equation*}
			|\nabla_x K_\beta(x,t)|\simeq_{\beta} \Big\rvert \nabla_x \Big( t^{-\frac{n+2\beta}{2s}}\psi\big( xt^{-\frac{1}{2s}} \big)\Big)\Big\rvert =t^{-\frac{n+2\beta+1}{2s}}\big\rvert\nabla\psi\big( xt^{-\frac{1}{2s}} \big)\big\rvert.
		\end{equation*}
		Therefore, applying the bound obtained for $\nabla \psi$ in \eqref{eq2.6} we deduce
		\begin{align*}
			\big\rvert \nabla_x K_\beta(x,t) \big\rvert\lesssim_{\beta} t^{-\frac{n+2\beta+1}{2s}} \frac{|x|t^{-\frac{1}{2s}}}{\big( 1+|x|^2t^{-1/s} \big)^{(n+2\beta+1)/2}}&\approx \frac{|x|}{|\ox|_{p_s}^{n+2\beta+2}}.
		\end{align*}
		We move on to estimate \textit{4}, that is, the one concerning $\partial_tK_\beta(\ox)$ at points of the form $\ox\neq (x,0)$. Observe that the previous derivative is well defined if $t>0$, since the expression of $K_\beta$ can be written as 
		\begin{align*}
			K_\beta(x,t\big) &\simeq t^{-\frac{n+2\beta}{2s}}(-\Delta)^{\beta}\phi\big( |x|t^{-\frac{1}{2s}} \big) \\
			&\simeq_{\beta} |x|^{1-n/2}\bigg( \frac{1}{t^{\frac{n+4\beta+2}{4s}}} \int_0^\infty e^{-4\pi^2r^{2s}}r^{n/2+2\beta}J_{n/2-1}(2\pi r|x|t^{-\frac{1}{2s}})\dd r\bigg),
		\end{align*}
		so differentiating under the integral sign, it is clear that temporal derivatives of any order exist in $\mathbb{R}^{n+1}\setminus{\{t=0\}}$. We claim now that the operators $\partial_t$ and $(-\Delta)^{\beta}$ commute when applied to $P_s$. To prove this, let us first observe that for each $t_0>0$ fixed we have
		\begin{align*}
			\big[ (-\Delta)^{\beta}\big( \partial_t P_s \big) \big]^\wedge(\xi, t_0)=|\xi|^{2\beta}  \widehat{\partial_t P_s}(\xi, t_0)=|\xi|^{2\beta}\int_{\mathbb{R}^n}e^{-2\pi i\langle x,\xi \rangle}\partial_tP_s(x,t_0)\dd x.
		\end{align*}
		If we can bound $\partial_t P_s$ by an integrable function on $\mathbb{R}^n$ in a neighborhood of $t_0$, we will be able to locally differentiate outside the integral sign for each $t_0$. If $0<s<1$, this is a consequence of \cite[Equation 2.6]{Va} and \eqref{eq2.2}. Indeed,
		\begin{equation*}
			|\partial_tP_s(x,t_0)|\lesssim \frac{1}{t_0}|P_s(x,t_0)|\lesssim \frac{1}{t_0^{\frac{n+2s}{2s}}}\Bigg[ \frac{1}{\big( 1+|x|^2t_0^{-1/s} \big)^{(n+2s)/2}} \Bigg].
		\end{equation*}
		On the other hand, if $s=1$ by definition we have
		\begin{equation*}
			|\partial_tW(x,t_0)|\lesssim \Bigg( 1+\frac{|x|^2}{t_0} \Bigg)\frac{1}{t_0^{n/2+1}}e^{-|x|^2/(4t_0)}.
		\end{equation*}
		In both cases we obtain a bounded function of $x$ that decreases like $|x|^{-n-2}$ at infinity (for the case $s=1$, see \cite[Lemma 2.1]{MaPrTo}) and thus it is integrable on $\mathbb{R}^n$. Therefore, differentiating outside the integral sign we have
		\begin{equation*}
			\big[ (-\Delta)^{\beta}\big( \partial_t P_s \big) \big]^\wedge(\xi, t_0) = \partial_t\big[ (-\Delta)^{\beta}P_s \big]^\wedge(\xi, t_0), \qquad \forall t_0>0.
		\end{equation*}
		So we are left to check whether we can enter $\partial_t$ inside the previous Fourier transform, that is, whether the following holds
		\begin{equation*}
			\partial_t\big[ (-\Delta)^{\beta}P_s \big]^\wedge(\xi, t_0)=\big[\partial_t (-\Delta)^{\beta}P_s \big]^\wedge(\xi, t_0).
		\end{equation*}
		Again, the latter is just a matter of being able to bound $|\partial_t (-\Delta)^{\beta}P_s|=|\partial_tK_\beta|$ locally for each $t_0>0$ by an integrable function, so that we can differentiate under the integral defining the Fourier transform. We know that
		\begin{align*}
			|\partial_tK_\beta(x,t_0)| &= \Big\rvert \partial_t\Big[ t^{-\frac{n+2\beta}{2s}}\psi\big( xt^{-\frac{1}{2s}} \big) \Big]_{t=t_0}\Big\rvert \\
			&\hspace{3cm}\lesssim_{\beta} C_1(t_0)\big\rvert \psi\big( xt_0^{-\frac{1}{2s}} \big) \big\rvert + C_2(t_0)|x|\big\rvert \nabla \psi\big( xt_0^{-\frac{1}{2s}} \big) \big\rvert.
		\end{align*}
		For the first summand, using  that $|\psi|$ is bounded and decays as $|x|^{-n-2\beta}$, we deduce the desired integrability condition. For the second summand we can argue exactly in the same manner, using that $|\nabla \psi|$ is bounded and decays as $|x|^{-n-2\beta-1}$. Hence, we conclude that $\partial_t$ and $(-\Delta)^{\beta}$ commute.
        
		The previous commutativity relation and \cite[Eq.
		2.5]{MaPr} yield the following for $t>0$,
		\begin{align*}
			\partial_tK_\beta(x,t)&=\partial_t\big[(-\Delta)^{\beta}P_s\big](x,t)=(-\Delta)^{\beta}\big( \partial_tP_s \big)(x,t)\\
			&=(-\Delta)^{\beta}\Big[ -(-\Delta)^s P_s \Big](x,t)=-(-\Delta)^{s}K_\beta(x,t),
		\end{align*}
		where we have commuted the operators $(-\Delta)^{s}$ and $(-\Delta)^{\beta}$, that can be easily checked via their Fourier transform. Then, applying \textit{2} with $\gamma=s$ we are done.
        
		Finally, regarding estimate \textit{5}, we can follow the same proof to that presented for the last estimate in Theorem \ref{C-Z_thm2}, using estimates \textit{3} and \textit{4} from above.
	\end{proof}
	
	\subsection{Estimates for \mathinhead{\partial_t^{\beta}P_s}{}}
	
	In this subsection we obtain similar estimates now for the kernel $\partial_t^{\beta}P_s$, with $\beta\in(0,1)$. Recall that the $\beta$\textit{-temporal derivative} of $f:\mathbb{R}^{n+1}\to\mathbb{R}$ is defined, provided it exists, as
	\begin{equation*}
		\partial_t^{\beta}f(x,t):=\int_{\mathbb{R}}\frac{f(x,\tau)-f(x,t)}{|\tau-t|^{1+s}}\dd \tau.
	\end{equation*}
	The study below considers the cases $s<1$ and $s=1$ separately. In the following theorem, which generalizes \cite[Lemma 2.2]{MaPr}, we get dimensional restrictions that in the end will not matter for our purposes.
	
	\begin{thm}
		\label{lem2.4}
		Let $\beta,s\in(0,1)$. Then, the following hold for any $\ox=(x,t)\neq (0,t):$
		\begin{align*}
			\textit{1}\hspace{0.03cm}.&  \;\; \text{If}\hspace{0.25cm} n>1, \hspace{0.5cm} |\partial_t^{\beta}P_s(\ox)|\lesssim_{\beta} \frac{1}{|x|^{n-2s}|\ox|_{p_s}^{2s(1+\beta)}},\\
			\textit{2}\hspace{0.03cm}.&\;\; \text{If}\hspace{0.25cm} n=1 \text{ and }\, \beta>1-\frac{1}{2s}, \hspace{0.5cm} |\partial_t^{\beta}P_s(\ox)|\lesssim_{\beta,\alpha} \frac{1}{|x|^{1-2s+\alpha}|\ox|_{p_s}^{2s(1+\beta)-\alpha}}, \;\;\forall \alpha\in(2s-1,4s).
		\end{align*}
		Moreover, for every $n$,
		\begin{equation*}
			\textit{3}\hspace{0.03cm}.\;\;|\nabla_x \partial_t^{\beta}P_s(\ox)|\lesssim_{\beta} \frac{1}{|x|^{n-2s+1}|\ox|_{p_s}^{2s(1+\beta)}}, \hspace{0.5cm}  \textit{4}\hspace{0.03cm}.\;\;|\partial_t \partial_t^{\beta}P_s(\ox)|\lesssim_{\beta} \frac{1}{|x|^{n}|\ox|_{p_s}^{2s(1+\beta)}}, \hspace{0.5cm} \text{for} \hspace{0.35cm} t\neq 0.
		\end{equation*}
		Finally, if $\ox'\in\mathbb{R}^{n+1}$ is such that $|\ox-\ox'|_{p_s}\leq |x|/2$,
		\begin{align*}
			\textit{5}\hspace{0.03cm}.\;\; |\partial_t^{\beta}P_s(\ox)-\partial_t^{\beta}P_s(\ox')|\lesssim_{\beta} \frac{|\ox-\ox'|_{p_s}^{2\zeta}}{|x|^{n+2\zeta-2s}|\ox|_{p_s}^{2s(1+\beta)}}.
		\end{align*}
	\end{thm}
	\begin{proof}
		To prove \textit{1}, we use \cite[Equation 2.9]{MaPr} and deduce the existence of a function $F_s$ such that for $t>0$,
		\begin{equation}
			\label{eq2.9}
			P_s(x,t)=\frac{1}{|x|^n}F_s\bigg( \frac{t}{|x|^{2s}} \bigg),
		\end{equation}
		and such that
		\begin{equation}
			\label{eq2.10}
			F_s(u)\approx \frac{u}{\big( 1+u^{1/s} \big)^{(n+2s)/2}}.
		\end{equation}
		We extend continuously $F_s(u):=0$ for $u\leq 0$, so that \eqref{eq2.9} is verified for any value of $t$. The existence of $F_s$ is clear, since for $t>0$ the function $P_s$ can be written as
		\begin{equation*}
			P_s(x,t)=\frac{1}{|x|^n}\bigg( \frac{t}{|x|^{2s}} \bigg)^{-\frac{n}{2s}}\phi_{n,s}\Bigg[ \bigg( \frac{t}{|x|^{2s}} \bigg)^{-\frac{1}{2s}} \Bigg],
		\end{equation*}
		and defining for $u>0$, $F_s(u):=u^{-\frac{n}{2s}}\phi_{n,s}\big(u^{-\frac{1}{2s}}\big)$, we are done. Notice that $F_s$ is a bounded continuous function, null for negative values of $u$, smooth in the domain $u>0$ and vanishing at $\infty$. Moreover, using the bounds obtained for $\phi'$ and $\phi''$ we obtain the following estimates for $u>0$,
		\begin{equation}
			\label{eq2.11}
			|F_s'(u)|\lesssim \frac{1}{\big( 1+u^{1/s} \big)^{(n+2s)/2}}, \qquad |F_s''(u)|\lesssim \frac{1}{u\big( 1+u^{1/s} \big)^{(n+2s)/2}}
		\end{equation}
		Let us argue that, in fact, $|F_s''(u)|$ is also a bounded function. Notice that, by definition,
		\begin{equation*}
			\partial_\tau^2P_s(x,\tau)=\frac{1}{|x|^{n+4s}}F_s''\bigg( \frac{\tau}{|x|^{2s}} \bigg) \hspace{0.5cm} \Leftrightarrow \hspace{0.5cm} \bigg\rvert F_s''\bigg( \frac{\tau}{|x|^{2s}} \bigg) \bigg\rvert = |x|^{n+4s}\big\rvert \partial_\tau^2P_s(x,\tau)\big\rvert,
		\end{equation*}
		and using that $P_s$ is the fundamental solution to the $\Theta^s$-equation and that $\tau>0$, we have
		\begin{align*}
			\partial_\tau^2P_s(x,\tau)=\partial_\tau\big[-(-\Delta)^s P_s(x,\tau)  \big].
		\end{align*}
		By the commutativity of $\partial_\tau$ and $(-\Delta)^s$, we deduce
		\begin{equation*}
			\big\rvert \partial_\tau^2P_s(x,\tau) \big\rvert = \big\rvert(-\Delta)^s \big[ \partial_\tau P_s (x,\tau) \big]\big\rvert = \big\rvert(-\Delta)^s \big[ -(-\Delta)^s P_s (x,\tau) \big]\big\rvert\lesssim \frac{1}{|\ox|_{p_s}^{n+4s}}.
		\end{equation*}
		Therefore,
		\begin{align*}
			\bigg\rvert F_s''\bigg( \frac{\tau}{|x|^{2s}} \bigg) \bigg\rvert \lesssim \frac{|x|^{n+4s}}{|\ox|_{p_s}^{n+4s}}=\frac{1}{\max\Big\{ 1,\big(\tau/|x|^{2s}\big)^{1/{(2s)}} \Big\}^{n+4s}}\lesssim \frac{1}{\Big[ 1+ \big(\tau/|x|^{2s}\big)^{1/{s}}\Big]^{(n+4s)/2}},
		\end{align*}
		that implies the following (improved) bound for $F''_s$,
		\begin{equation}
			\label{eq2.12}
			|F''_s(u)|\lesssim \frac{1}{\big( 1+u^{1/s} \big)^{(n+4s)/2}}\leq 1, \qquad u>0.
		\end{equation}
		We continue by observing that by a change of variables the following holds,
		\begin{align}
			\partial_{t}^{\beta}P_s(x,t)&=\frac{1}{|x|^n}\bigg[ \partial_t^{\beta} F_s\bigg( \frac{\cdot}{|x|^{2s}} \bigg) \bigg](t) =\frac{1}{|x|^{n+2s\beta}}\partial^{\beta} F_s\bigg( \frac{t}{|x|^{2s}} \bigg). \label{eq2.13}
		\end{align}
		We shall prove the following inequality,
		\begin{equation}
			\big\rvert\partial^{\beta}F_s(u)\big\rvert\lesssim_{\beta} \min\bigg\{ 1,\frac{1}{|u|^{1+\beta}} \bigg\},\label{eq2.14}
		\end{equation}
		where for $u=0$ is just asking for $\big\rvert\partial^{\beta}F_s(0)\big\rvert$ to be bounded. To verify \eqref{eq2.14} we distinguish whether if $u=0$, $u<0$ or $u>0$. For $u=0$ observe that by definition and relation \eqref{eq2.10},
		\begin{align*}
			\big\rvert\partial^{\beta}F_s(0)\big\rvert &\leq \int_{\mathbb{R}}\frac{|F_s(0)-F_s(w)|}{|0-w|^{1+\beta}}\dd w=\int_0^\infty\frac{|F_s(w)|}{w^{1+\beta}}\dd w\\
			&\lesssim_{\beta} \int_0^\infty\frac{1}{w^{1+\beta}}\frac{w}{\big( 1+w^{1/s} \big)^{(n+2s)/2}}\dd w\\
			&=\int_0^1 \frac{\dd w}{w^{\beta}\big( 1+w^{1/s} \big)^{(n+2s)/2}} +\int_1^\infty \frac{\dd w}{w^{\beta}\big( 1+w^{1/s} \big)^{(n+2s)/2}}\\
			&\approx \int_0^1\frac{\dd w}{w^{\beta}}+\int_1^\infty\frac{\dd w}{w^{\frac{n}{2s}+1+\beta}}\lesssim (1-\beta)^{-1}+\bigg(\frac{n}{2s}+\beta\bigg)^{-1}\lesssim_{\beta} 1,
		\end{align*}
		so case $u=0$ is done. Let us assume $u<0$, so that
		\begin{align*}
			\big\rvert\partial ^{\beta}F_s(u)\big\rvert &\leq \int_{\mathbb{R}}\frac{|F_s(w)|}{||u|+w|^{1+\beta}}\dd w\lesssim \int_0^\infty\frac{1}{(|u|+w)^{1+\beta}}\frac{w}{\big( 1+w^{1/s} \big)^{(n+2s)/2}}\dd w.
		\end{align*}
		On the one hand notice that the since $|u|+w>w$, the previous expression is bounded by a constant depending on $n,s$ and $\beta$ (by the same arguments given for the case $u=0$). On the other hand, observe that
		\begin{align*}
			\big\rvert\partial^{\beta}F_s(u)\big\rvert &\lesssim \frac{1}{|u|^{1+\beta}}\int_0^\infty\frac{1}{(w/|u|+1)^{1+\beta}}\frac{w}{\big( 1+w^{1/s} \big)^{(n+2s)/2}}\dd w\\
			&=\frac{1}{|u|^{1+\beta}}\Bigg[ \int_0^1\frac{1}{(w/|u|+1)^{1+\beta}}\frac{w}{\big( 1+w^{1/s} \big)^{(n+2s)/2}}\dd w \\
			&\hspace{2cm}+ \int_1^\infty\frac{1}{(w/|u|+1)^{1+\beta}}\frac{w}{\big( 1+w^{1/s} \big)^{(n+2s)/2}}\dd w  \Bigg] =:\frac{1}{|u|^{1+\beta}}(I_1+I_2).
		\end{align*}
		Regarding $I_1$, since the denominators are bigger than $1$, we directly have
		\begin{equation}
			\label{eq2.15}
			I_1\lesssim \int_0^1w\,\dd w\leq 1
		\end{equation}
		Turning to $I_2$, we similarly obtain
		\begin{align}
			\label{eq2.16}
			I_2\leq \int_1^\infty \frac{w}{\big( 1+w^{1/s} \big)^{(n+2s)/2}}\dd w\leq \int_1^\infty \frac{\dd w}{w^{\frac{n}{2s}}}=w^{-\frac{n}{2s}+1}\bigg\rvert_1^\infty=1,
		\end{align}
		where notice that $-\frac{n}{2s}+1<0$ because $n>1$ and $s<1$. Therefore, we also have  $\big\rvert\partial^{\beta}F_s(u)\big\rvert\lesssim |u|^{-1-\beta}$ and we conclude that for $u\leq 0$,
		\begin{equation*}
			\big\rvert\partial^{\beta}F_s(u)\big\rvert\lesssim_{\beta}\min\bigg\{ 1,\frac{1}{|u|^{1+\beta}} \bigg\}.
		\end{equation*}
		Let us finally assume $u>0$. Begin by writing
		\begin{align*}
			\big\rvert\partial^{\beta}F_s(u)\big\rvert &\leq \int_{|w|\leq u/2} \frac{|F_s(w)-F_s(u)|}{|w-u|^{1+\beta}}\dd w+\int_{u/2\leq |w|\leq 2u} \frac{|F_s(w)-F_s(u)|}{|w-u|^{1+\beta}}\dd w\\
			&\hspace{1cm}+\int_{|w|> 2u} \frac{|F_s(w)-F_s(u)|}{|w-u|^{1+\beta}}\dd w=:I_1+I_2+I_{3}.
		\end{align*}
		We study each of the previous integrals separately. Concerning the first, notice that in its domain of integration $u/2\leq |w-u|\leq 3u/2$, i.e. $|w-u|\approx u$. We split it as follows
		\begin{equation*}
			I_1=\int_{-u/2}^0\frac{|F_s(u)|}{|w-u|^{1+\beta}}\dd w+\int_{0}^{u/2}\frac{|F_s(w)-F_s(u)|}{|w-u|^{1+\beta}}\dd w=:I_{11}+I_{12}.
		\end{equation*}
		Observe that $I_{11}$ can be estimated by
		\begin{equation*}
			I_{11}\lesssim \frac{u}{\big( 1+u^{1/s} \big)^{(n+2s)/2}}\int_{-u/2}^0\frac{\dd w}{|u|^{1+\beta}}\simeq_{\beta} \frac{u^{1-\beta}}{\big( 1+u^{1/s} \big)^{(n+2s)/2}}.
		\end{equation*}
		The expression of the right, viewed as a continuous function of $u$, tends to zero as $u\to 0$ and decays as $|u|^{-\beta-\frac{n}{2s}}$ as $u\to\infty$. Hence, it is bounded by a constant (depending on $n,s$ and $\beta$) and so $I_{11}\lesssim_{\beta} 1$. On the other hand, to prove that $I_{11}\lesssim_{\beta} |u|^{-1-\beta}$ it suffices to check that the following expression is bounded by a constant,
		\begin{equation*}
			\frac{u^{2}}{\big( 1+u^{1/s} \big)^{(n+2s)/2}}\approx u F_s(u).
		\end{equation*}
		Again, it is clear it that tends to zero as $u\to 0$, but observe that it behaves as $|u|^{-\frac{n}{2s}+1}$ as $u\to\infty$, which vanishes only if $n>2s$, that is, only if $n>1$, since $s<1$. But this is satisfied by hypothesis. Therefore we deduce $I_{11}\lesssim_{\beta} \min\{1,|u|^{-1-\beta}\}$.
		Regarding $I_{12}$ proceed in a similar manner to obtain
		\begin{align*}
			I_{12}\lesssim_{\beta} \frac{1}{u^{1+\beta}}\int_0^{u/2}\frac{w}{\big( 1+w^{1/s} \big)^{(n+2s)/2}}\dd w +\frac{u^{1-\beta}}{\big( 1+u^{1/s} \big)^{(n+2s)/2}}.
		\end{align*}
		The second summand has already been studied in $I_{11}$. Regarding the first, notice that
		\begin{align*}
			\int_0^{u/2}\frac{w}{\big( 1+w^{1/s} \big)^{(n+2s)/2}}\dd w&\leq \int_0^{1}\frac{w}{\big( 1+w^{1/s} \big)^{(n+2s)/2}}\dd w+\int_1^{\infty}\frac{w}{\big( 1+w^{1/s} \big)^{(n+2s)/2}}\dd w\\
			&\leq \int_0^1w\,\dd w+\int_1^\infty\frac{\dd w}{w^{\frac{n}{2s}}}\lesssim 1,
		\end{align*}
		where we have applied the same arguments as in \eqref{eq2.15} and \eqref{eq2.16}. On the other hand, by applying the following inequality for $w>0$,
		\begin{equation*}
			(1+w^{1/s})^{(n+2s)/2}> w^{1-\beta},
		\end{equation*}
		that can be checked by a direct computation, we deduce
		\begin{align*}
			\int_0^{u/2}\frac{w}{\big( 1+w^{1/s} \big)^{(n+2s)/2}}\dd w&<  \int_0^{u/2} w^{\beta}\dd w\lesssim_{\beta} u^{1+\beta}.
		\end{align*}
		Therefore we conclude
		\begin{equation*}
			I_{12}\lesssim_{\beta} \frac{1}{u^{1+\beta}}\min\Big\{ 1,u^{1+\beta} \Big\}+\min\bigg\{ 1,\frac{1}{u^{1+\beta}} \bigg\}= 2\min\bigg\{ 1,\frac{1}{u^{1+\beta}}\bigg\},
		\end{equation*}
		that implies the desired estimate for $I_1$.
        
		Moving on to $I_2$, we split it as follows
		\begin{equation*}
			I_2=\int_{-2u}^{-u/2}\frac{|F_s(u)|}{|w-u|^{1+\beta}}\dd w+\int_{u/2}^{2u}\frac{|F_s(w)-F_s(u)|}{|w-u|^{1+\beta}}\dd w=:I_{21}+I_{22}.
		\end{equation*}
		The study of $I_{21}$ is exactly the same as the one presented for $I_{11}$, so we focus on $I_{22}$. Apply the mean value theorem to obtain
		\begin{equation*}
			I_{22}\leq \sup_{\nu\in[u/2,2u]}|F_s'(\nu)|\int_{u/2}^{2u}\frac{\dd w}{|w-u|^{\beta}}\lesssim_{\beta} \sup_{\nu\in[u/2,2u]}|F_s'(\nu)| \, u^{1-\beta}.
		\end{equation*}
		Therefore, if we are able to bound $|F'_s|$ by $u^{\beta-1}$ and $u^{-2}$ we will be done. But recalling relation \eqref{eq2.11}, this is equivalent to proving that the following functions are bounded by a constant:
		\begin{equation}
			\label{eq2.17}
			\frac{u^{\beta-1}}{\big( 1+u^{1/s} \big)^{(n+2s)/2}}, \hspace{0.75cm} \frac{u^{2}}{\big( 1+u^{1/s} \big)^{(n+2s)/2}},
		\end{equation}
		that has already been done in $I_{11}$. Therefore, we are only left to study $I_{3}$,
		\begin{align*}
			I_{3}=\int_{-\infty}^{-2u}\frac{|F_s(u)|}{|w-u|^{1+\beta}}\dd w+\int_{2u}^{\infty}\frac{|F_s(w)-F_s(u)|}{|w-u|^{1+\beta}}\dd w=:I_{31}+I_{32}.
		\end{align*}
		To deal with $I_{31}$ we first notice that in the domain of integration $|w-u|\approx |w|$, implying
		\begin{align*}
			I_{31}\approx \frac{u}{\big( 1+u^{1/s} \big)^{(n+2s)/2}}\int_{-\infty}^{-2u}\frac{\dd w}{|w|^{1+\beta}}\lesssim_{\beta} \frac{u^{1-\beta}}{\big( 1+u^{1/s} \big)^{(n+2s)/2}}\lesssim_{\beta} \min\bigg\{ 1,\frac{1}{u^{1+\beta}}\bigg\}
		\end{align*}
		We study $I_{32}$ by splitting it as
		\begin{equation*}
			I_{32}\leq \int_{2u}^{\infty}\frac{|F_s(w)|}{|w-u|^{1+\beta}}\dd w+\int_{2u}^{\infty}\frac{|F_s(u)|}{|w-u|^{1+\beta}}\dd w.
		\end{equation*}
		The second summand is tackled in exactly the same way as $I_{31}$, so we focus on the first one. Using that $|w-u|\approx |w|\gtrsim u$, we have
		\begin{align*}
			\int_{2u}^{\infty}\frac{|F_s(w)|}{|w-u|^{1+\beta}}\dd w&\lesssim \frac{1}{u^{1+\beta}}\int_{2u}^{\infty}\frac{w}{\big( 1+w^{1/s} \big)^{(n+2s)/2}}\dd w \\
			&\leq \frac{1}{u^{1+\beta}}\Bigg[\int_0^1w\,\dd w+\int_1^\infty\frac{\dd w}{w^{\frac{n}{2s}}} \Bigg] \lesssim \frac{1}{u^{1+\beta}},
		\end{align*}
		by the same arguments used in \eqref{eq2.15} and \eqref{eq2.16}. On the other hand, we also have
		\begin{align*}
			\int_{2u}^{\infty}\frac{|F_s(w)|}{|w-u|^{1+\beta}}\dd w&\lesssim \int_{2u}^\infty \frac{1}{w^{1+\beta}}\frac{w}{\big( 1+w^{1/s} \big)^{(n+2s)/2}}\dd w\leq \int_0^1\frac{\dd w}{w^{\beta}}+\int_{1}^\infty\frac{\dd w}{w^{\frac{n}{2s}}}.
		\end{align*}
		We already know that the second integral is bounded by a constant for $n>1$, while the first one is also bounded, since $0<\beta<1$. So we conclude that $I_{32}\lesssim_{\beta} \min\{1,|u|^{-1-\beta}\}$ and we obtain the desired bound for $I_{3}$ and thus for $|\partial^{\beta} F_s(u)|$ if $u>0$.
        
		All in all, returning to \eqref{eq2.13}, we finally have
		\begin{align*}
			|\partial_t^{\beta}P_s(x,t)|&=\frac{1}{|x|^{n+2s\beta}}\bigg\rvert \partial^{\beta} F_s\bigg( \frac{t}{|x|^{2s}} \bigg)\bigg\rvert \lesssim_{\beta} \frac{1}{|x|^{n+2s\beta}} \min\bigg\{ 1, \frac{|x|^{2s(1+\beta)}}{|t|^{1+\beta}} \bigg\}\\
			&=\frac{1}{|x|^{n-2s}} \min\bigg\{ \frac{1}{|x|^{2s(1+\beta)}}, \frac{1}{|t|^{\frac{2s(1+\beta)}{2s}}} \bigg\}=\frac{1}{|x|^{n-2s}|\ox|_{p_s}^{2s(1+\beta)}},
		\end{align*}
		that is estimate \textit{1} in the statement of the lemma.
        
		In order to prove \textit{2}, we follow the same scheme. Indeed, the desired estimate follows once we prove 
		\begin{equation*}
			\big\rvert\partial^{\beta}F_s(u)\big\rvert\lesssim_{\beta,\alpha} \min\bigg\{ 1,\frac{1}{|u|^{1+\beta-\frac{\alpha}{2s}}} \bigg\}, \qquad \text{for }\, 2s-1<\alpha<4s.
		\end{equation*}
		If one followed the same arguments used to prove \textit{1}, in the regime $u<0$ one already encounters a first bound for which dimension $n=1$ is troublesome, namely when trying to obtain $|\partial^{\beta}F_s(u)|\lesssim |u|^{-1-\beta+\frac{\alpha}{2s}}$. However, in our current setting we observe that
		\begin{align*}
			\int_0^\infty&\frac{1}{(w+|u|)^{1+\beta}}\frac{w}{\big( 1+w^{1/s} \big)^{(n+2s)/2}}\dd w \\
			&\hspace{-0.5cm}\quad\lesssim \frac{1}{|u|^{1+\beta-\frac{\alpha}{2s}}}\int_0^\infty \frac{1}{(w/|u|+1)^{1+\beta-\frac{\alpha}{2s}}}\frac{1}{(w+|u|)^{\frac{\alpha}{2s}}}\frac{w}{\big( 1+w^{1/s} \big)^{(n+2s)/2}}\dd w \\
			&\hspace{-0.5cm}\quad\lesssim \frac{1}{|u|^{1+\beta-\frac{\alpha}{2s}}}\bigg(\int_0^1 w^{1-\frac{\alpha}{2s}}\dd w+ \int_1^\infty \frac{\dd w }{w^{\frac{n+\alpha}{2s}}}\bigg)\lesssim_{\beta,\alpha} \frac{1}{|u|^{1+\beta-\frac{\alpha}{2s}}}, \quad \text{since $2s-1<\alpha<4s$},
		\end{align*}
		so the desired bound for $|\partial^{\beta}F_s(u)|$ follows. For the case $u>0$ we also proceed analogously. Let us comment those steps where the hypotheses on $\alpha$ and $\beta$ come into play. In $I_1$, using the same notation as for the case $n>1$, we obtain the estimates
		\begin{equation*}
			I_{11}\lesssim_{\beta} \frac{u^{1-\beta}}{\big( 1+u^{1/s} \big)^{(n+2s)/2}} \qquad \text{and} \qquad I_{12}\lesssim_{\beta} \frac{1}{u^{1+\beta}}\int_0^{u/2}\frac{w}{\big( 1+w^{1/s} \big)^{(n+2s)/2}}\dd w,
		\end{equation*}
		expression that we already know to be bounded by a constant. To prove that $I_{11}\lesssim_{\beta,\alpha} |u|^{-1-\beta+\frac{\alpha}{2s}}$ observe that the function
		\begin{equation*}
			\frac{u^{2-\frac{\alpha}{2s}}}{\big( 1+u^{1/s} \big)^{(n+2s)/2}}\approx u^{1-\frac{\alpha}{2s}} F_s(u)
		\end{equation*}
		tends to zero as $u\to 0$, since $\alpha<4s$. Moreover, it behaves as $|u|^{-\frac{n+\alpha}{2s}+1}$ as $u\to\infty$, which also tends to $0$ because $\alpha<2s-1$. Thus, $I_{11}\lesssim_{\beta,\alpha} \min\{1,|u|^{-1-\beta+\frac{\alpha}{2s}}\}$. On the other hand, since the following holds
		\begin{equation*}
			\big( 1+w^{1/s} \big)^{(n+2s)/2}>w^{2-\frac{\alpha}{2s}},
		\end{equation*}
		we obtain
		\begin{align*}
			\frac{1}{u^{1+\beta}}\int_0^{u/2}\frac{w}{\big( 1+w^{1/s} \big)^{(n+2s)/2}}\dd w < \frac{1}{u^{1+\beta}}\int_0^{u/2}\frac{\dd w}{w^{1-\frac{\alpha}{2s}}}\lesssim_{\beta,\alpha} \frac{1}{u^{1+\beta-\frac{\alpha}{2s}}}.
		\end{align*}
		Therefore, $I_{12}\lesssim_{\beta,\alpha} \min\{1,|u|^{-1-\beta+\frac{\alpha}{2s}}\}$, hence $I_1$ satisfies the same estimate. The study of $I_2$ is completely analogous to that of $n>1$. Therefore we are only left to study $I_{3}$. The arguments can be carried out analogously up to the point of estimating
		\begin{align*}
			\int_{2u}^{\infty}\frac{|F_s(w)|}{|w-u|^{1+\beta}}\dd w.
		\end{align*}
		Using that $|w-u|\approx |w|\gtrsim u$, we have
		\begin{align*}
			\int_{2u}^{\infty}\frac{|F_s(w)|}{|w-u|^{1+\beta}}\dd w&\lesssim \frac{1}{u^{1+\beta-\frac{\alpha}{2s}}}\int_{2u}^{\infty}\frac{w^{1-\frac{\alpha}{2s}}}{\big( 1+w^{1/s} \big)^{(n+2s)/2}}\dd w \\
			&\leq \frac{1}{u^{1+\beta-\frac{\alpha}{2s}}}\Bigg[\int_0^1w^{1-\frac{\alpha}{2s}}\,\dd w+\int_1^\infty\frac{\dd w}{w^{\frac{n+\alpha}{2s}}} \Bigg] \lesssim_{\beta,\alpha} \frac{1}{u^{1+\beta-\frac{\alpha}{2s}}},
		\end{align*}
		since $2s-1<\alpha<4s$. Therefore, $I_{32}\lesssim_{\beta,\alpha} \min\{1,|u|^{-1-\beta+\frac{\alpha}{2s}}\}$, and with this we get the desired bound for $I_{3}$ and the completion of the proof for the case $n=1$.
        
		Moving on to estimate \textit{3}, we begin by defining for $u>0$ the real variable function
		\begin{equation*}
			G_s(u):=u^{-\frac{n+1}{2s}}\phi'_n\big(u^{-\frac{1}{2s}}\big),
		\end{equation*}
		so that in light of relation \eqref{eq2.1} we have
		\begin{equation*}
			\nabla_x P_{s}(x,t)\simeq \frac{x}{|x|^{n+2}}G_s\bigg( \frac{t}{|x|^{2s}} \bigg), \qquad \text{for}\quad t>0, x\neq 0. 
		\end{equation*}
		
		By \eqref{eq2.5} it is clear that
		\begin{equation}
			\label{eq2.18}
			|G_s(u)|\approx \frac{u}{\big( 1+u^{1/s}\big)^{(n+2s+2)/2}}.
		\end{equation}
		Hence, as done for $F_s$, we can extend continuously the definition of $G_s$ by zero for negative values of $u$. Notice also that the previous estimate implies that $G_s$ is bounded on $\mathbb{R}$.
        
		Our next claim is that the operators $\nabla_x$ and $\partial_t^{\beta}$ commute when applied to $P_s$. To prove this, it suffices to check that the following integral is locally well-defined for every $x$ and $t$,
		\begin{equation*}
			\int_{\mathbb{R}}\frac{|\nabla_x P_s(x,t)-\nabla_x P_{s}(x,w)|}{|t-w|^{1+\beta}}\dd w.
		\end{equation*}
		Split the domain of integration as
		\begin{equation*}
			\int_{|t-w|<1}\frac{|\nabla_x P_s(x,t)-\nabla_x P_{s}(x,w)|}{|t-w|^{1+\beta}}\dd w+\int_{|t-w|\geq 1}\frac{|\nabla_x P_s(x,t)-\nabla_x P_{s}(x,w)|}{|t-w|^{1+\beta}}\dd w.
		\end{equation*}
		The second integral is clearly well-defined, since $\nabla_x P_s(x,t)\simeq x/|x|^{n+2}G_s(t/|x|^{2s})$ and we know that $G_s$ is bounded. Thus, directly applying the triangle inequality in the numerator and using that $\beta>0$, we deduce that, indeed, the second integral is finite. For the first one, we need some more work. We shall distinguish four possibilities:
        
		\textit{Case 1: $t\leq -1$}. For such values of $t$ the integral becomes null, since $\nabla_xP_s(x,t)$ and $\nabla_x P(x,w)$ are zero.
        
		\textit{Case 2: $t\in(-1,0]$}. Observe that in this setting the integral can be rewritten as
		\begin{align*}
			\int_{0}^{1-|t|}\frac{|\nabla_x P_{s}(x,w)|}{|w-t|^{1+\beta}}\dd w&=\int_{0}^{1-|t|}\frac{|\nabla_x P_{s}(x,w)-\nabla_x P_{s}(x,0)|}{|w-t|^{1+\beta}}\dd w\\
			&\hspace{4cm}\lesssim \frac{1}{|x|^{n+2s+1}}\int_0^{1-|t|}\frac{|G'_s(\tau/|x|^{2s})|}{|w|^{\beta}}\dd w,
		\end{align*}
		for some $\tau\in(0,w)$. By definition, there are constants $C_1,C_2$ so that for $u>0$
		\begin{equation*}
			G_s'(u)=C_1\,u^{-(n+2s+1)/(2s)}\phi_{n,s}'\big( u^{-\frac{1}{2s}} \big)+C_2\,u^{-(n+2s+2)/(2s)}\phi_{n,s}''\big( u^{-\frac{1}{2s}}\big),
		\end{equation*}
		so using the estimates for $\phi'_n$ and $\phi_{n,s}''$ in \eqref{eq2.5} we deduce
		\begin{align}
			\label{eq2.19}
			|G'_s(u)|\approx \frac{1}{\big( 1+u^{1/s} \big)^{(n+2s+2)/2}},
		\end{align}
		which is a bounded function. Therefore
		\begin{equation*}
			\frac{1}{|x|^{n+2s+1}}\int_0^{1-|t|}\frac{|G'_s(\tau/|x|^{2s})|}{|w|^{\beta}}\dd w\lesssim_{\beta}\frac{1}{|x|^{n+2s+1}}<\infty,
		\end{equation*}
		for every $x\neq 0$.
        
		\textit{Case 3: $t\in(0,1]$}. The integral we were initially studying can be written as
		\begin{align*}
			\int_{t-1}^{0}\frac{|\nabla_x P_{s}(x,t)|}{|t-w|^{1+\beta}}\dd w+\int_{0}^{t+1}\frac{|\nabla_x P_{s}(x,t)-\nabla_x P_{s}(x,w)|}{|t-w|^{1+\beta}}\dd w.
		\end{align*}
		The second integral can be tackled in exactly the same way as the integral in \textit{Case 2}. Regarding the first one, estimate it as follows
		\begin{align*}
			\int_{t-1}^{0}\frac{|\nabla_x P_{s}(x,t)-\nabla_x P_{s}(x,0)|}{|t-w|^{1+\beta}}\dd w&\leq \frac{1}{|x|^{n+2s+1}}\int_{t-1}^0\frac{|G_s'(\tau/|x|^{2s})||t|}{|t-w|^{1+\beta}}\dd w\\
			&\lesssim \frac{1}{|x|^{n+2s+1}} \int_{t-1}^0\frac{|G_s'(\tau/|x|^{2s})|}{|w|^{\beta}}\dd w <\infty,
		\end{align*}
		where we have used $|t|\leq |t-w|+|w|$ and also that $|t-w|=(t+|w|)\geq |w|$. The last inequality follows by the same arguments used in \textit{Case 2}.
        
		\textit{Case 4: $t> 1$}. For this final case, the integral can be estimated as
		\begin{align*}
			\int_{t-1}^{t+1}\frac{|\nabla_x P_{s}(x,t)-\nabla_x P_{s}(x,w)|}{|t-w|^{1+\beta}}\dd w\lesssim \frac{1}{|x|^{n+2s+1}}  \int_{t-1}^{t+1}\frac{|G_s'(\tau/|x|^{2s})|}{|t-w|^{\beta}}\dd w\lesssim_{\beta} \frac{1}{|x|^{n+2s+1}} <\infty.
		\end{align*}
		Thus, we have obtained the desired commutativity between $\partial_t^{\beta}$ and $\nabla_x$, which yields
		\begin{align*}
			\nabla_x \partial_t^{\beta}P_s(x,t)=\partial_t^{\beta}\big[\nabla_x P_s\big](x,t)&=\frac{x}{|x|^{n+2}}\bigg[\partial_t^{\beta} G_s\bigg( \frac{\cdot}{|x|^{2s}} \bigg) \bigg](t) =\frac{x}{|x|^{n+2s\beta+2}}\partial^{\beta} G_s\bigg( \frac{t}{|x|^{2s}} \bigg).
		\end{align*}
		Now it is a matter of showing that the following inequality holds
		\begin{equation}
			\big\rvert\partial^{\beta}G_s(u)\big\rvert\lesssim_{\beta} \min\bigg\{ 1,\frac{1}{|u|^{1+\beta}} \bigg\},\label{eq2.20}
		\end{equation}
		The proof of \eqref{eq2.20} is essentially identical to the one given for \eqref{eq2.14}, using the bounds for $G_s$ and $G_s'$ (\eqref{eq2.18} and \eqref{eq2.19} respectively) instead of those for $F_s$ and $F_s'$. The faster decay of $G_s$ and its derivative implies that one does not find any obstacles in \eqref{eq2.16}. In fact, the integral that appears in the current analysis is $\int_1^\infty w^{-\frac{n+1}{2s}} \dd w$, which also converges for $n=1$. So using the previous estimate we deduce, for any $n>0$,
		\begin{align*}
			|\nabla_x \partial_t^{\beta}P_s(x,t)|&=\frac{1}{|x|^{n+2s\beta+1}}\bigg\rvert \partial_t^{\beta} G_s\bigg( \frac{t}{|x|^{2s}} \bigg)\bigg\rvert \lesssim_\beta \frac{1}{|x|^{n+2s\beta+1}} \min\bigg\{ 1, \frac{|x|^{2s(1+\beta)}}{|t|^{1+\beta}} \bigg\}\\
			&=\frac{1}{|x|^{n-2s+1}} \min\bigg\{ \frac{1}{|x|^{2s(1+\beta)}}, \frac{1}{|t|^{\frac{2s(1+\beta)}{2s}}} \bigg\}=\frac{1}{|x|^{n-2s+1}|\ox|_{p_s}^{2s(1+\beta)}},
		\end{align*}
		which proves the statement \textit{3} in our lemma.
        
		We continue by estimating $\partial_t\partial_t^{\beta}P_s(x,t)$ for $x\neq 0$ and $t\neq 0$. Using \eqref{eq2.13} we rewrite it as
		\begin{equation*}
			\partial_t\partial_t^{\beta}P_s(\ox)=\frac{1}{|x|^{n+2s(1+\beta)}}\partial^{\beta}F_s'\bigg( \frac{t}{|x|^{2s}} \bigg),
		\end{equation*}
		and we claim that the following inequality holds for $u\neq 0$,
		\begin{equation*}
			\big\rvert\partial^{\beta}F_s'(u)\big\rvert\lesssim_{\beta} \min\bigg\{ 1,\frac{1}{|u|^{1+\beta}} \bigg\}.
		\end{equation*}
		Let us also recall that we had the following estimates for $u>0$,
		\begin{align*}
			|F_s'(u)|&\lesssim \frac{1}{\big( 1+u^{1/s} \big)^{(n+2s)/2}}, \hspace{1cm} |F_s''(u)|\lesssim \frac{1}{\big( 1+u^{1/s} \big)^{(n+4s)/2}} \leq \frac{1}{u\big( 1+u^{1/s} \big)^{(n+2s)/2}}.
		\end{align*}
		Observe that, on the one hand,
		\begin{align*}
			\big\rvert\partial^{\beta}F_s'(u)\big\rvert &\leq \int_{\mathbb{R}}\frac{|F_s'(u)-F_s'(w)|}{|u-w|^{1+\beta}}\dd w\\
			&\leq \sup_{\nu\in \mathbb{R}}|F_s''(\nu)|  \int_{|u-w|<1}\frac{\dd w}{|u-w|^{\beta}} + 2\sup_{\nu\in \mathbb{R}}|F_s'(\nu)|\int_{|u-w|\geq 1}\frac{\dd w}{|u-w|^{1+\beta}} \lesssim_\beta 1, 
		\end{align*}
		by the boundedness of $F_s'$ and $F_s''$, and the fact that $\beta\in(0,1)$. Therefore we are left to verify $\big\rvert\partial^{\beta}F_s'(u)\big\rvert \lesssim_\beta |u|^{-1-\beta}$. If $u<0$, since $F_s'$ is supported on $(0,\infty)$ and $|u-w|>|u|$ for $w\geq 0$, we have
		\begin{align*}
			\big\rvert\partial^{\beta}F_s'(u)\big\rvert &\leq \int_0^\infty \frac{|F_s'(w)|}{|u-w|^{1+\beta}}\dd w \lesssim \int_0^{1} \frac{\dd w}{|u-w|^{1+\beta}} +\int_{1}^\infty \frac{|F_s'(w)| }{|u-w|^{1+\beta}}\dd w\\
			&\lesssim \frac{1}{|u|^{1+\beta}}\bigg( 1 + \int_{1}^\infty \frac{\dd w}{\big( 1+w^{1/s} \big)^{(n+2s)/2}} \bigg)\leq \frac{1}{|u|^{1+\beta}}\bigg( 1+\int_1^\infty \frac{\dd w}{w^{\frac{n}{2s}+1}} \bigg) \lesssim \frac{1}{|u|^{1+\beta}},
		\end{align*}
		and we are done. If on the other hand $u>0$, we estimate $\rvert\partial^{\beta}F_s'\rvert$ in a similar way as $\rvert\partial^{\beta}F_s\rvert$ in the proof of point \textit{1} of this lemma. Namely, we write
		\begin{align*}
			\big\rvert\partial^{\beta}F_s'(u)\big\rvert &\leq \int_{|w|\leq u/2} \frac{|F_s'(w)-F_s'(u)|}{|w-u|^{1+\beta}}\dd w+\int_{u/2\leq |w|\leq 2u} \frac{|F_s'(w)-F_s'(u)|}{|w-u|^{1+\beta}}\dd w\\
			&\hspace{1cm}+\int_{|w|> 2u} \frac{|F_s'(w)-F_s'(u)|}{|w-u|^{1+\beta}}\dd w=:I_1+I_2+I_{3}.
		\end{align*}
		Regarding $I_1$, notice that in the domain of integration we have $|w-u|\approx u$, so
		\begin{align*}
			I_1&\lesssim \int_{-u/2}^{0} \frac{|F_s'(u)|}{|w-u|^{1+\beta}}\dd w + \int_0^{u/2} \frac{|F_s'(u)|}{|w-u|^{1+\beta}}\dd w+ \int_0^{u/2} \frac{|F_s'(w)|}{|w-u|^{1+\beta}}\dd w.
		\end{align*}
		The first two integrals can be directly bounded by
		\begin{equation*}
			\frac{1}{|u|^{1+\beta}}|F_s'(u)|\int_{0}^{u/2}\dd w \leq  \frac{1}{|u|^{1+\beta}}\bigg(\frac{u}{(1+u^{1/s})^{(n+2s)/2}}\bigg)\leq \frac{1}{|u|^{1+\beta}}.
		\end{equation*}
		For the third,
		\begin{align*}
			\int_0^{u/2} \frac{|F_s'(w)|}{|w-u|^{1+\beta}}\dd w &\lesssim \frac{1}{|u|^{1+\beta}}\bigg( \int_0^1|F_s'(w)|\dd w + \int_1^\infty |F_s'(w)|\dd w \bigg) \\
			&\lesssim \frac{1}{|u|^{1+\beta}}\bigg( 1 + \int_1^\infty \frac{\dd w}{w^{\frac{n}{2s}+1}} \bigg) \lesssim \frac{1}{|u|^{1+\beta}},
		\end{align*}
		and we are done with $I_1$. Moving on to $I_2$, we split it as follows
		\begin{equation*}
			I_2=\int_{-2u}^{-u/2}\frac{|F_s'(u)|}{|w-u|^{1+\beta}}\dd w+\int_{u/2}^{2u}\frac{|F_s'(w)-F_s'(u)|}{|w-u|^{1+\beta}}\dd w=:I_{21}+I_{22}.
		\end{equation*}
		The study of $I_{21}$ can be carried analogously to that of $I_1$, since in that domain of integration one has $|w-u|\geq 3|u|/2$, so we focus on $I_{22}$. Applying the mean value theorem and the bound for $|F_s''|$ of \eqref{eq2.11} as well as relation \eqref{eq2.10} we get 
		\begin{align*}
			I_{22}\leq \sup_{\nu\in[u/2,2u]}|F_s''(\nu)|\int_{u/2}^{2u}\frac{\dd w}{|w-u|^{\beta}}&\lesssim_\beta \sup_{\nu\in[u/2,2u]}|F_s''(\nu)| \, u^{1-\beta} \\
			&\lesssim \frac{u^{1-\beta}}{u(1+u^{1/s})^{(n+2s)/2}} \approx F_s(u) u^{-1-\beta}\leq u^{-1-\beta}.
		\end{align*}
		So we are left to study $I_{3}$. Since in its domain of integration we have $|w-u|\gtrsim w$, we get
		\begin{align*}
			I_{3} \lesssim \int_{-\infty}^{-2u} \frac{|F_s'(u)|}{|w|^{1+\beta}}\dd w &+ \int_{2u}^{\infty} \frac{|F_s'(w)|+|F_s'(u)|}{|w|^{1+\beta}}\dd w \\
			&\lesssim \frac{3}{(1+u^{1/s})^{(n+2s)/2}}\int_{2u}^{\infty}\frac{\dd w}{|w|^{1+\beta}} \lesssim_\beta \frac{u^{-\beta}}{(1+u^{1/s})^{(n+2s)/2}}\leq u^{-1-\beta},
		\end{align*}
		that allows us to finally conclude
		\begin{equation*}
			\big\rvert\partial^{\beta}F_s'(u)\big\rvert\lesssim_\beta \min\bigg\{ 1,\frac{1}{|u|^{1+\beta}} \bigg\}, \qquad u\neq 0.
		\end{equation*}
		So using the previous estimate get, for $x\neq 0$ and $t\neq 0$,
		\begin{align*}
			|\partial_t \partial_t^{\beta}P_s(x,t)|&=\frac{1}{|x|^{n+2s(1+\beta)}}\bigg\rvert \partial_t^{\beta} F_s'\bigg( \frac{t}{|x|^{2s}} \bigg)\bigg\rvert \lesssim_\beta \frac{1}{|x|^{n+2s(1+\beta)}} \min\bigg\{ 1, \frac{|x|^{2s(1+\beta)}}{|t|^{1+\beta}} \bigg\}\\
			&=\frac{1}{|x|^{n}} \min\bigg\{ \frac{1}{|x|^{2s(1+\beta)}}, \frac{1}{|t|^{\frac{2s(1+\beta)}{2s}}} \bigg\}=\frac{1}{|x|^{n}|\ox|_{p_s}^{2s(1+\beta)}}.
		\end{align*}
		Finally, the proof of estimate \textit{5} is analogous to that of \textit{5} in Theorem \ref{lem2.3}. Indeed, let $\ox'=(x',t')\in\mathbb{R}^{n+1}$ such that $|\ox-\ox'|_{p_s}\leq |x|/2$, which is a stronger assumption than that of Theorem \ref{lem2.3}. In fact, it can be checked by a direct computation that this already implies $|\ox|_{p_s}\leq 2 |\ox'|_{p_s}$ and $|x|\leq 2|x'|$. Write again $\widehat{x}=(x',t)$ and consider
		\begin{equation*}
			|\partial_t^{\beta}P_s(\ox)-\partial_t^{\beta}P_s(\ox')|\leq |\partial_t^{\beta}P_s(\ox)-\partial_t^{\beta}P_s(\widehat{x})|+|\partial_t^{\beta}P_s(\widehat{x})-\partial_t^{\beta}P_s(\ox')|.
		\end{equation*}
		By estimate \textit{2}, the first term in the above inequality now satisfies
		\begin{align*}
			|x-x'|\sup_{\xi\in[x,x']}|\nabla_x\partial_t^{\beta}P_s(\xi,t)|\lesssim_\beta \frac{|x-x'|}{|x|^{n-2s+1}|\ox|_{p_s}^{2s(1+\beta)}}&\leq \frac{|\ox-\ox'|_{p_s}}{|x|^{n-2s+1}|\ox|_{p_s}^{2s(1+\beta)}}\\
			&\lesssim \frac{|\ox-\ox'|_{p_s}^{2\zeta}}{|x|^{n+2\zeta-2s}|\ox|_{p_s}^{2s(1+\beta)}},
		\end{align*}
		where we have used that $1-2\zeta\geq 0$ and that condition $|\ox-\ox'|_{p_s}\leq |x|/2$ implies that the line segment joining $\ox$ with $\ox'$ is at a distance of the time axis comparable to $|x|$. Regarding the second term, assume $t>t'$. If $t$ and $t'$ share sign we apply estimate \textit{3} to directly deduce
		\begin{align*}
			|t-t'|\sup_{\tau\in[t,t']}&|\partial_t\partial_t^{\beta}P_s(x',\tau)|\lesssim_\beta \frac{|t-t'|}{|x|^{n}|\ox|_{p_s}^{2s(1+\beta)}}\leq \frac{|\ox-\ox'|_{p_s}^{2s}}{|x|^{n}|\ox|_{p_s}^{2s(1+\beta)}} \lesssim \frac{|\ox-\ox'|_{p_s}^{2\zeta}}{|x|^{n+2\zeta-2s}|\ox|_{p_s}^{2s(1+\beta)}}
		\end{align*}
		If on the other hand $t>0$ and $t'<0$, we use relation \eqref{eq2.8}, valid also in this case, together with $|x'|\geq |x|/2$ to finally obtain
		\begin{align*}
			|\partial_t^{\beta}P_s(\widehat{x})&-\partial_t^{\beta}P_s(\ox')|\\
			&\leq|\partial_t^{\beta}P_s(x',t)-\partial_t^{\beta}P_s(x',0)|+|\partial_t^{\beta}P_s(x',0)-\partial_t^{\beta}P_s(x',t')|\\
			&\lesssim t\sup_{\tau\in(0,t)}|\partial_t\partial_t^{\beta}P_s(x',\tau)|+|t'|\sup_{\tau\in(t',0)}|\partial_t\partial_t^{\beta}P_s(x',\tau)|\\
			&\lesssim_\beta \frac{t+|t'|}{|x'|^n|x'|^{2s(1+\beta)}}\lesssim \frac{|t-t'|}{|x|^{n}|\ox|_{p_s}^{2s(1+\beta)}}\lesssim  \frac{|\ox-\ox'|_{p_s}^{2\zeta}}{|x|^{n+2\zeta-2s}|\ox|_{p_s}^{2s(1+\beta)}}.
		\end{align*}
	\end{proof}
	
	We will now carry out the same study for the case $s=1$. First, we prove the following auxiliary lemma:
	\begin{lem}
		\label{lem3.6}
		Let $f_1,f_2,f_3:\mathbb{R}\to\mathbb{R}$ be defined as
		\begin{equation*}
			f_1(t):=\frac{e^{-1/t}}{t^{n/2}}\, \chi_{t>0},\hspace{0.75cm} f_2(t):= \frac{e^{-1/t}}{t^{n/2+1}}\, \chi_{t>0}, \hspace{0.75cm} f_3(t):= \frac{e^{-1/t}}{t^{n/2+2}}\, \chi_{t>0}.
		\end{equation*}
		Then, if $\beta\in(0,1)$, the following estimates hold
		\begin{align*}
			\text{if}\hspace{0.25cm} n>2,& \hspace{0.75cm} |\partial_t^{\beta}f_1(t)|\lesssim_\beta \min\big\{ 1, |t|^{-1-\beta}\big\},\\
			\text{if}\hspace{0.25cm} n=2 \text{ and }\, \beta>\frac{1}{2},& \hspace{0.75cm} |\partial_t^{\beta}f_1(t)|\lesssim_{\beta,\alpha} \min\big\{ 1, |t|^{-1-\beta+\alpha/2}\big\}, \;\;\forall \alpha\in (0,2+2\beta],\\
			\text{if}\hspace{0.25cm} n=1,& \hspace{0.75cm} |\partial_t^{\beta}f_1(t)|\lesssim_{\beta} 1.
		\end{align*}
		In addition, for every $n$,
		\begin{equation*}
			|\partial_t^{\beta}f_2(t)|\lesssim_{\beta} \min\big\{ 1,|t|^{-1-\beta} \big\},\hspace{0.75cm} |\partial_t^{\beta}f_3(t)|\lesssim_{\beta} \min\big\{ 1,|t|^{-1-\beta} \big\}.
		\end{equation*}
		For $t=0$ the previous estimates have to be understood simply as a bound by a constant depending on $n$ and $\beta$.
	\end{lem}
	The above result will imply the following estimates for $\partial_{t}^{\beta}W$:
	\begin{thm}
		\label{lem3.5}
		For any $\ox=(x,t)\neq (0,t)$ and $\beta\in(0,1)$, the following hold:
		\begin{align*}
			\textit{1}\hspace{0.03cm}.&  \;\; \text{For}\hspace{0.25cm} n>2, \hspace{0.75cm} |\partial_{t}^{\beta}W(\ox)|\lesssim_{\beta} \frac{1}{|x|^{n-2}|\ox|_{p_1}^{2+2\beta}},\\
			\textit{2}\hspace{0.03cm}.&\;\; \text{For}\hspace{0.25cm} n=2, \hspace{0.75cm} |\partial_{t}^{\beta}W(\ox)|\lesssim_{\beta,\alpha} \frac{1}{|x|^{\alpha}|\ox|_{p_1}^{2+2\beta-\alpha}}, \;\;\forall \alpha\in (0,2+2\beta],\\
			\textit{3}\hspace{0.03cm}.&\;\; \text{For}\hspace{0.25cm} n=1, \hspace{0.75cm} |\partial_{t}^{\beta}W(\ox)|\lesssim_{\beta} \frac{1}{|\ox|_{p_1}^{1+2\beta}}.
		\end{align*}
		Moreover, for every $n$,
		\begin{equation*}
			\textit{4}\hspace{0.03cm}.\;\;|\nabla_x\partial_{t}^{\beta}W(\ox)|\lesssim_{\beta} \frac{1}{|x|^{n-1}|\ox|_{p_1}^{2+2\beta}}, \hspace{0.75cm}  \textit{5}\hspace{0.03cm}.\;\;|\partial_t\partial_{t}^{\beta}W(\ox)|\lesssim_{\beta} \frac{1}{|x|^{n}|\ox|_{p_1}^{2+2\beta}}.
		\end{equation*}
		Finally, if $\ox'\in\mathbb{R}^{n+1}$ is such that $|\ox-\ox'|_{p_1}\leq |x|/2$, then
		\begin{equation*}
			\textit{6}\hspace{0.03cm}.\;\; |\partial_{t}^{\beta}W(\ox)-\partial_{t}^{\beta}W(\ox')|\lesssim_{\beta} \frac{|\ox-\ox'|_{p_1}}{|x|^{n-1}|\ox|_{p_1}^{2+2\beta}}.
		\end{equation*}
	\end{thm}
	\begin{proof}[Proof of Lemma \ref{lem3.6}]
		We deal first with the estimate concerning $\partial_t^{\beta}f_1$ for $n>2$. We distinguish whether if $t=0$, $t<0$ or $t>0$. If $t=0$ we are done because,
		\begin{equation*}
			|\partial_t^{\beta}f_1(0)|\leq\int_{\mathbb{R}}\frac{|f_1(u)-f_1(0)|}{|u-0|^{1+\beta}}\dd u=\int_0^\infty \frac{e^{-1/u}}{u^{(n+2+2\beta)/2}}\dd u=\Gamma\bigg(\frac{n+2\beta}{2}\bigg)\lesssim_{\beta} 1,
		\end{equation*}
		where $\Gamma$ denotes the usual gamma function.
        
		Let us continue by assuming $t<0$. By definition,
		\begin{equation*}
			|\partial_t^{\beta}f_1(t)|\leq \int_{\mathbb{R}}\frac{|f_1(u)|}{|u+|t||^{1+\beta}}\dd u=\int_{0}^{\infty}\frac{e^{-1/u}}{u^{n/2}(u+|t|)^{1+\beta}}\dd u.
		\end{equation*}
		Observe that on the one hand, since $|u+|t||\geq u$,
		\begin{equation*}
			|\partial_t^{\beta}f_1(t)|\leq \int_0^\infty\frac{e^{-1/u}}{u^{(n+2+2\beta)/2}}\dd u\lesssim_{\beta} 1.
		\end{equation*}
		On the other hand, since $n>2$,
		\begin{equation*}
			|\partial_t^{\beta}f_1(t)|\leq \frac{1}{|t|^{1+\beta}}\int_{0}^\infty\frac{e^{-1/u}}{u^{n/2}(u/|t|+1)^{1+\beta}}\dd u\leq \frac{1}{|t|^{1+\beta}}\int_{0}^\infty\frac{e^{-1/u}}{u^{n/2}}\dd u\lesssim \frac{1}{|t|^{1+\beta}},
		\end{equation*}
		Therefore, $ |\partial_t^{\beta}f_1(t)|\lesssim_{\beta} \min\big\{ 1,|t|^{-1-\beta} \big\}$ and we are done.
        
		If $t>0$, we split the integral as follows
		\begin{align*}
			|\partial_t^{\beta}f_1(t)|\leq&\int_{|u|\leq t/2}\frac{|f_1(u)-f_1(t)|}{|u-t|^{1+\beta}}\dd u+\int_{t/2\leq |u|\leq 2t}\frac{|f_1(u)-f_1(t)|}{|u-t|^{1+\beta}}\dd u\\
			&+\int_{|u|\geq 2t}\frac{|f_1(u)-f_1(t)|}{|u-t|^{1+\beta}}\dd u =: I_1+I_2+I_{3}.
		\end{align*}
		In $I_1$ we have $t/2\leq |u-t|\leq 3t/2$. Therefore,
		\begin{align*}
			I_1:=\int_{-t/2}^{0}\frac{|f_1(t)|}{|u-t|^{1+\beta}}\dd u+\int_{0}^{t/2}\frac{|f_1(u)-f_1(t)|}{|u-t|^{1+\beta}}\dd u\lesssim \frac{e^{-1/t}}{t^{(n+2\beta)/2}}+\int_{0}^{t/2}\frac{|f_1(u)-f_1(t)|}{t^{1+\beta}}\dd u.
		\end{align*}
		By the definition of $f_1$, the last term can be bound by
		\begin{align}
			\frac{1}{t^{1+\beta}}\int_{0}^{t/2}\frac{e^{-1/u}}{u^{n/2}}\dd u+\frac{1}{t^{1+\beta}}\int_{0}^{t/2}&\frac{e^{-1/t}}{t^{n/2}}\dd u\nonumber\\
			&\simeq  \frac{1}{t^{1+\beta}}\int_{0}^{t/2}\frac{e^{-1/u}}{u^{n/2}}\dd u+\frac{e^{-1/t}}{t^{(n+2\beta)/2}}.\label{eq2.1.21}
		\end{align}
		We split the remaining integral as follows
		\begin{align*}
			\int_{0}^{t/2}\frac{e^{-1/u}}{u^{n/2}}\dd u=\int_{0}^{1}\frac{e^{-1/u}}{u^{n/2}}\dd u&+\int_{1}^{t/2}\frac{e^{-1/u}}{u^{n/2}}\dd u\\
			&\leq e^{-\frac{1}{2t}}\int_{0}^{1}\frac{e^{-\frac{1}{2u}}}{u^{n/2}}\dd u+ e^{-2/t}\int_{1}^{t/2}\frac{1}{u^{n/2}}\dd u\\
			&\lesssim e^{-\frac{1}{2t}}+\frac{e^{-\frac{1}{2t}}}{t^{n/2-1}},
		\end{align*}
		where in the first inequality we have used $e^{-1/u}\leq e^{-\frac{1}{2u}}\,e^{-\frac{1}{2t}}$, which is true for $0\leq u\leq t/2$; and in the second the general inequality $e^{-2/t}\leq e^{-\frac{1}{2t}}$. In addition, observe that in the last step we have used that $n\neq 2$ in order to compute the corresponding integral. Thus, returning to \eqref{eq2.1.21}, we obtain
		\begin{align*}
			I_1\lesssim \frac{e^{-\frac{1}{2t}}}{t^{1+\beta}}+ \frac{e^{-\frac{1}{2t}}}{t^{(n+2\beta)/2}}.
		\end{align*}
		Notice that for $t>0$
		\begin{equation}
			\label{eq2.1.22}
			e^{-\frac{1}{2t}}\leq 3\min\big\{1,t^{1+\beta}\big\}, \hspace{0.75cm} e^{-\frac{1}{2t}}\leq C \min\big\{t^{(n+2\beta)/2}, t^{(n-2)/2}\big\},
		\end{equation}
		where $C$ depends only on $n$ and $\beta$, and the second estimate only holds for $n>1$ (if $n=1$, $e^{-\frac{1}{2t}}\leq C t^{(n+2\beta)/2}$ still holds). Therefore, we finally get
		\begin{equation*}
			I_1\lesssim_{\beta} \frac{\min\big\{1,t^{1+\beta}\big\}}{t^{1+\beta}}+\frac{\min\big\{t^{(n+2\beta)/2}, t^{(n-2)/2}\big\}}{t^{(n+2\beta)/2}}\simeq \min\bigg\{ 1,\frac{1}{t^{1+\beta}} \bigg\}.
		\end{equation*}
		Let us turn to $I_2$. Write
		\begin{equation}
			\label{eq2.1.23}
			I_2:=\int_{-2t}^{-t/2}\frac{|f_1(t)|}{|u-t|^{1+\beta}}\dd u+\int_{t/2}^{2t}\frac{|f_1(u)-f_1(t)|}{|u-t|^{1+\beta}}\dd u\lesssim \frac{e^{-1/t}}{t^{(n+2\beta)/2}}+\int_{t/2}^{2t}\frac{|f_1(u)-f_1(t)|}{|u-t|^{1+\beta}}\dd u,
		\end{equation}
		where in the first integral we have used that $3t/2\leq |u-t|\leq 3t$. For the second integral observe that
		\begin{equation*}
			|f_1(u)-f_1(t)|\leq \sup_{\xi\in [s,t]}|f_1'(\xi)||u-t|, \hspace{0.5cm} \text{where} \hspace{0.5cm} f_1'(\xi)=\bigg(1- \frac{n}{2}\xi  \bigg)\frac{e^{-1/\xi}}{\xi^{n/2+2}}\,\chi_{\xi>0}.
		\end{equation*}
		Since $t/2\leq \xi\leq 2t$, we have
		\begin{align*}
			|f_1'(\xi)|&\lesssim \big( 1+t \big)\frac{e^{-\frac{1}{2t}}}{t^{n/2+2}}\,\chi_{t>0}=\big( 1+t \big)\frac{e^{-\frac{1}{2t}}}{t^{n/2+2}}\,\chi_{t>1}+\big( 1+t \big)\frac{e^{-\frac{1}{2t}}}{t^{n/2+2}}\,\chi_{0<t\leq 1}\\
			&\lesssim \frac{e^{-\frac{1}{2t}}}{t^{n/2+1}}\,\chi_{t>1}+\frac{e^{-\frac{1}{2t}}}{t^{n/2+2}}\,\chi_{0<t\leq 1}.
		\end{align*}
		Combining the last two estimates we can bound the remaining integral of \eqref{eq2.1.23} by
		\begin{align*}
			\bigg( \frac{e^{-\frac{1}{2t}}}{t^{n/2+1}}\,\chi_{t>1}+\frac{2e^{-\frac{1}{2t}}}{t^{n/2+2}}&\,\chi_{0<t\leq 1} \bigg)\int_{t/2}^{2t}\frac{\dd u}{|u-t|^{\beta}}\lesssim_{\beta} \frac{e^{-\frac{1}{2t}}}{t^{(n+2\beta)/2}}\,\chi_{t>1}+\frac{e^{-\frac{1}{2t}}}{t^{(n+2+2\beta)/2}}\,\chi_{0<t\leq 1}.
		\end{align*}
		Thus,
		\begin{equation*}
			I_2\lesssim_{\beta} \frac{2e^{-\frac{1}{2t}}}{t^{(n+2\beta)/2}} +\frac{e^{-\frac{1}{2t}}}{t^{(n+2+2\beta)/2}}.
		\end{equation*}
		If we now apply estimates 
		\begin{equation*}
			e^{-\frac{1}{2t}}\leq C_1 \min\big\{t^{(n+2\beta)/2}, t^{(n-2)/2}\big\}, \hspace{0.75cm} e^{-\frac{1}{2t}}\leq C_2 \min\big\{t^{(n+2+2\beta)/2}, t^{n/2}\big\},
		\end{equation*}
		for some constants $C_1,C_2$ depending on $n$ and $\beta$, we conclude
		\begin{equation*}
			I_2\lesssim_{\beta} \frac{\min\big\{t^{(n+2\beta)/2}, t^{(n-2)/2}\big\}}{t^{(n+2\beta)/2}}+\frac{\min\big\{t^{(n+2+2\beta)/2}, t^{n/2}\big\}}{t^{(n+2+2\beta)/2}}\simeq\min\bigg\{ 1,\frac{1}{t^{1+\beta}} \bigg\}.
		\end{equation*}
		Finally, for $I_{3}$, since $|u|/2 \leq |u-t|\leq 3|u|/2$, we have
		\begin{align*}
			I_{3}&:=\int_{-\infty}^{-2t}\frac{|f_1(t)|}{|u-t|^{1+\beta}}\dd u+\int_{2t}^{\infty}\frac{|f_1(u)-f_1(t)|}{|u-t|^{1+\beta}}\dd u \simeq \frac{e^{-1/t}}{t^{(n+2\beta)/2}}+\int_{2t}^{\infty}\frac{|f_1(u)-f_1(t)|}{|u-t|^{1+\beta}}\dd u\\
			&\leq \frac{e^{-1/t}}{t^{(n+2\beta)/2}} +  \int_{2t}^{\infty}\frac{e^{-1/u}}{u^{(n+2+2\beta)/2}}\dd u+\int_{2t}^{\infty}\frac{e^{-1/t}}{t^{n/2}u^{1+\beta}}\dd u\simeq_{\beta} \frac{e^{-1/t}}{t^{(n+2\beta)/2}} +  \int_{2t}^{\infty}\frac{e^{-1/u}}{u^{(n+2+2\beta)/2}}\dd u\\
			&\lesssim_{\beta} \min\bigg\{ 1,\frac{1}{t^{1+\beta}} \bigg\}+\int_{2t}^{\infty}\frac{e^{-1/u}}{u^{(n+2+2\beta)/2}}\dd u.
		\end{align*}
		For the remaining integral observe that on the one hand
		\begin{equation*}
			\int_{2t}^{\infty}\frac{e^{-1/u}}{u^{(n+2+2\beta)/2}}\dd u \leq \Gamma\bigg( \frac{n+2\beta}{2} \bigg)\lesssim_{\beta} 1,
		\end{equation*}
		while on the other hand, since $u>2t$,
		\begin{equation*}
			\int_{2t}^{\infty}\frac{e^{-1/u}}{u^{(n+2+2\beta)/2}}\dd u \lesssim \frac{1}{t^{1+\beta}}\int_{2t}^\infty\frac{e^{-1/u}}{u^{n/2}}\dd u\leq \frac{1}{t^{1+\beta}}\int_{0}^\infty\frac{e^{-1/u}}{u^{n/2}}\dd u \lesssim \frac{1}{t^{1+\beta}},
		\end{equation*}
		where the last inequality holds since $n>2$. Therefore, combining the previous estimates we conclude that for $n>2$, $|\partial_t^{\beta}f_1(t)|\lesssim_{\beta} \min\big\{1,t^{-1-\beta}\big\}$.
        
		Before approaching the case $n=2$, let us comment that the case $n=1$ also follows from the above arguments. We also notice that the bounds for $|\partial_t^{\beta}f_2|$ and $|\partial_t^{\beta}f_3|$ are obtained by exactly the same computations. 
		So we are left to verify the following estimate
		\begin{equation*}
			|\partial_t^{\beta}f_1(t)|\lesssim_{\beta,\alpha} \min\big\{ 1, |t|^{-1-\beta+\alpha/2}\big\}, \;\;\forall \alpha\in (0,3], \;\; n=2,
		\end{equation*}
		that can be also obtained following the same scheme of proof.
	\end{proof}
	\begin{proof}[\textit{Proof} of Theorem \ref{lem3.5}]
		We write $K_\beta(\ox):=\partial_t^{\beta}W(\ox)$. Regarding estimate \textit{1}, by the same reasoning presented at the beginning of the proof of \cite[Lemma 2.1]{MaPrTo} we get
		\begin{equation*}
			K_\beta(\ox) \simeq \frac{1}{|x|^{n+2\beta}}\partial_t^{\beta}f_1\bigg( \frac{4t}{|x|^2} \bigg).
		\end{equation*}
		Hence, if $n>2$, by Lemma \ref{lem3.6} we get
		\begin{equation*}
			|K_\beta(\ox)|\lesssim_{\beta} \frac{1}{|x|^{n+2\beta}}\min\bigg\{ 1, \frac{|x|^{2+2\beta}}{|t|^{1+\beta}} \bigg\}=\frac{1}{|x|^{n-2}}\min\bigg\{ \frac{1}{|x|^{2+2\beta}}, \frac{1}{|t|^{1+\beta}} \bigg\}=\frac{1}{|x|^{n-2}|\ox|_{p_1}^{2+2\beta}}.
		\end{equation*}
		For estimates \textit{2} and \textit{3} we follow the same procedure.
        
		We move on to estimate \textit{4}. First, observe that the expression $\nabla_x K$ is well-defined and that the operators $\nabla_x$ and $\partial_t^\beta$ commute when applied to $W$.
		We also observe that there is a constant $C$ such that
		\begin{equation*}
			\nabla_x W (x,t)=C\frac{x}{(4t)^{n/2+1}}\,e^{-|x|^2/(4t)}\, \chi_{t>0}=C\frac{x}{|x|^{n+2}}\bigg( \frac{|x|^2}{4t} \bigg)^{n/2+1}e^{-|x|^2/(4t)}\, \chi_{t>0},
		\end{equation*}
		so we can write
		\begin{equation*}
			\nabla_x W(x,t)=C\frac{x}{|x|^{n+2}}\,f_2\bigg(\frac{4t}{|x|^2}\bigg),
		\end{equation*}
		with $f_2$ defined in Lemma \ref{lem3.6}. Since $\nabla_x$ and $\partial_t^{\beta}$ commute,
		\begin{equation*}
			\nabla_x K (x,t) = C\frac{x}{|x|^{n+2}}\,\partial_t^{\beta}\bigg[f_2\bigg(\frac{4\,\cdot}{|x|^2}\bigg)\bigg](t).
		\end{equation*}
		The previous fractional derivative can be written as follows
		\begin{align*}
			\partial_t^{\beta}\bigg[f_2\bigg(\frac{4\,\cdot}{|x|^2}\bigg)\bigg](t)&=\int_{\mathbb{R}}\frac{f_2(4u/|x|^2)-f_2(4t/|x|^2)}{|u-t|^{1+\beta}}\dd u\simeq \frac{1}{|x|^{2\beta}}\partial_t^{\beta}f_2\bigg( \frac{4t}{|x|^2} \bigg),
		\end{align*}
		yielding the final equality
		\begin{equation*}
			\nabla_xK (\ox)= C\frac{x}{|x|^{n+2+2\beta}}\,\partial_t^{\beta}f_2\bigg(\frac{4t}{|x|^2}\bigg).
		\end{equation*}
		Applying Lemma \ref{lem3.6} we finally deduce \textit{3}:
		\begin{equation*}
			|\nabla_x K (\ox)|\lesssim_{\beta} \frac{1}{|x|^{n+1+2\beta}}\min\bigg\{ 1, \frac{|x|^{2+2\beta}}{|t|^{1+\beta}} \bigg\}=\frac{1}{|x|^{n-1}|\ox|_{p_1}^{2+2\beta}}.
		\end{equation*}
		Concerning inequality \textit{4}, since the operators $\partial_t^{\beta}$ and $\partial_t$ commute, we directly have
		\begin{equation*}
			\partial_t K(\ox)= \int_{\mathbb{R}}\frac{\partial_{t} W(x,u)-\partial_{t}W(x,t)}{|u-t|^{1+\beta}}\dd u,
		\end{equation*}
		and this integral makes sense. 
		As done for $\nabla_x W$, we can also rewrite $\partial_t W$ as follows,
		\begin{align*}
			\partial_t W(x,t) &= C_1\,\frac{e^{-|x|^2/(4t)}}{t^{n/2+1}}+C_2|x|^2\,\frac{e^{-|x|^2/(4t)}}{t^{n/2+2}}\\
			&=\bigg[\frac{C_1'}{|x|^{n+2}}\bigg( \frac{|x|^2}{4t} \bigg)^{n/2+1}+\frac{C_2'}{|x|^{n+2}}\bigg( \frac{|x|^2}{4t} \bigg)^{n/2+2}\bigg]e^{-|x|^2/(4t)}\\
			&=\frac{C_1'}{|x|^{n+2}}f_2\bigg( \frac{4t}{|x|^2} \bigg)+\frac{C_2'}{|x|^{n+2}}f_3\bigg( \frac{4t}{|x|^2} \bigg),
		\end{align*}
		where $f_3$ is defined in Lemma \ref{lem3.6}. By exactly the same change of variables as the one performed when studying $\nabla_x K$, we reach the identity
		\begin{equation*}
			\partial_t K(\ox)=\frac{C_1'}{|x|^{n+2+2\beta}}\partial_t^{\beta}f_2\bigg( \frac{4t}{|x|^2} \bigg)+\frac{C_2'}{|x|^{n+2+2\beta}}\partial_t^{\beta}f_3\bigg( \frac{4t}{|x|^2} \bigg).
		\end{equation*}
		By Lemma \ref{lem3.6}, we get inequality \textit{4}:
		\begin{equation*}
			|\partial_t K (\ox)|\lesssim_{\beta} \frac{1}{|x|^{n+2+2\beta}}\min\bigg\{ 1, \frac{|x|^{2+2\beta}}{|t|^{1+\beta}} \bigg\}=\frac{1}{|x|^{n}|\ox|_{p_1}^{2+2\beta}}.
		\end{equation*}
		Finally, regarding \textit{5}, we follow exactly the same proof as that of estimate \textit{4} in Theorem \ref{lem2.4}.
	\end{proof}

	\section{Growth estimates for admissible functions}
	\label{sec2.2}
	
	We will say that a positive Borel measure $\mu$ in $\mathbb{R}^{n+1}$ has upper $s$-parabolic \textit{growth of degree $\rho$ \text{\normalfont{(}}with constant $C$\text{\normalfont{)}}} or simply $s$\textit{-parabolic $\rho$-growth} if there is some constant $C(n,s)>0$ such that for any $s$-parabolic ball $B(\ox,r)$,
	\begin{equation*}
		\mu\big( B(\overline{x},r) \big) \leq Cr^\rho.
	\end{equation*}
	It is clear that this property is invariant if formulated using cubes instead of balls. We will be interested in a generalized version of such growth that can be defined not only for measures, but also for general distributions. To introduce such notion we present the concept of admissible function:
	\begin{defn}
		Let $s\in(0,1)$. Given $\phi\in \pazocal{C}^\infty(\mathbb{R}^{n+1})$, we will say that it is an \textit{admissible} function for an $s$-parabolic cube $Q$ if $\text{supp}(\phi)\subset Q$ and
		\begin{equation*}
			\|\phi\|_\infty \leq 1, \qquad \|\nabla_x\phi\|_\infty \leq \ell(Q)^{-1}, \qquad \|\partial_t\phi\|_\infty \leq \ell(Q)^{-2s}, \qquad \|\Delta \phi\|_\infty \leq \ell(Q)^{-2}.
		\end{equation*}
	\end{defn}
	
	\begin{rem}
		\label{lem2.2.1}
		If $\phi$ is a $\pazocal{C}^{2}$ function supported on $Q$ $s$-parabolic cube with $\|\phi\|_\infty \leq 1$, $\|\nabla_x\phi\|_\infty \leq \ell(Q)^{-1}$ and $\|\Delta \phi\|_\infty \leq \ell(Q)^{-2}$, then it also satisfies
		\begin{equation*}
			\|(-\Delta)^s\phi\|_\infty \lesssim \ell(Q)^{-2s}.
		\end{equation*}
		Indeed, begin by observing that translations in $\mathbb{R}^n$ commute with $\nabla_x$ and $(-\Delta)^s$. From it, it is clear that we may assume $Q$ to be centered at the origin. Assuming this, let us fix $t\in \mathbb{R}$ and compute 
		\begin{align*}
			(-\Delta)^s\phi(x,t)&:=c_{n,s}\int_{\mathbb{R}^n}\frac{\phi(x+y,t)-2\phi(x,t)+\phi(x-y,t)}{|y|^{n+2s}}\dd y\\
			&=c_{n,s}\int_{2Q}\frac{\phi(x+y,t)-2\phi(x,t)+\phi(x-y,t)}{|y|^{n+2s}}\dd y\\
			&\hspace{2cm}c_{n,s}\int_{\mathbb{R}^n\setminus{2Q}}\frac{\phi(x+y,t)-2\phi(x,t)+\phi(x-y,t)}{|y|^{n+2s}}\dd y=:I_1+I_2.
		\end{align*}
		Regarding $I_2$, integration in polar coordinates yields
		\begin{equation*}
			|I_2|\leq 4c_{n,s}\int_{\mathbb{R}^n\setminus{2Q}}\frac{\dd y}{|y|^{n+2s}}\lesssim \ell(Q)^{-2s}.
		\end{equation*}
		For $I_1$, we apply twice the mean value theorem so that
		\begin{align*}
			|I_1|\leq c_{n,s}\int_{2Q}\frac{|\langle \nabla_x \phi(x+\eta_1y, t),y\rangle + \langle \nabla_x \phi(x-\eta_2y, t),y\rangle|}{|y|^{n+2s-1}}\lesssim \int_{2Q}\frac{\|\Delta \phi\|_{\infty}}{|y|^{n+2s-2}}\dd y\lesssim \ell(Q)^{-2s}.
		\end{align*}
	\end{rem}
	\begin{defn}
		We will say that a distribution $T$ has $s$-parabolic $n$-\textit{growth}if there exists some constant $C=C(n,s)>0$ such that, given any $s$-parabolic cube $Q$ and any function $\phi$ admissible for $Q$, we have
		\begin{equation*}
			|\langle T, \phi \rangle |\leq C \ell(Q)^n.
		\end{equation*}
	\end{defn}
	In the end, the results below will help us estimate the growth of distributions of the form $\varphi T$, for some particular choices of $T$ and a fixed admissible function $\varphi$, associated with a fixed $s$-parabolic cube.
    
	In any case, let us clarify that in the following Theorems \ref{Growth_thm1}, \ref{Growth_thm2}, \ref{thm0.3} and \ref{Growth_thm4}, we will fix $s\in(0,1]$ and $Q$ and $R$ will be $s$-parabolic cubes in $\mathbb{R}^{n+1}$ with $Q\cap R \neq \varnothing$. We will write $Q:=Q_1\times I_Q \subset \mathbb{R}^n\times \mathbb{R}$ and analogously for $R$. Moreover, $\varphi$ and $\phi$ will denote $\pazocal{C}^1$ functions with $\text{supp}(\varphi)\subset Q$, $\text{supp}(\phi)\subset R$ and such that $\|\varphi\|_\infty\leq 1$ and $\|\phi\|_\infty\leq 1$.
	
	\begin{thm}
		\label{Growth_thm1}
		Let $\beta\in(0,1)$, $\alpha\in(0,1)$ and $f:\mathbb{R}^{n+1}\to \mathbb{R}$. Assume  $\|\partial_t\varphi\|_\infty\leq \ell(Q)^{-2s}$ and $\|\partial_t\phi\|_\infty\leq \ell(R)^{-2s}$. Then, if $\ell(R)\leq \ell(Q)$,
		\begin{enumerate}
			\item[1.] If $f\in \text{\normalfont{BMO}}_{p_s}$,
			\begin{equation*}
				|\langle f, \partial_t(\varphi\phi)\ast_t |t|^{-\beta}\rangle|\lesssim_{\beta} \|f\|_{\ast,p_s} \ell(R)^{n+2s(1-\beta)}.
			\end{equation*}
			\item[2.] If $f\in \text{\normalfont{Lip}}_{\alpha,p_s}$ and $\alpha<2s\beta$,
			\begin{equation*}
				|\langle f, \partial_t(\varphi\phi)\ast_t |t|^{-\beta}\rangle|\lesssim_{\beta,\alpha} \|f\|_{\text{\normalfont{Lip}}_{\alpha,p_s}} \ell(R)^{n+2s(1-\beta)+\alpha}.
			\end{equation*}
		\end{enumerate}
	\end{thm}
	\begin{proof} 
		Set $g:=\partial_t(\varphi \phi)\ast_t |t|^{-\beta}$ and begin by proving that $g$ is integrable. Firstly, observe that if $c_{Q\cap R}$ is the center of $I_Q\cap I_R$, then for each $t\notin 2(I_Q\cap I_R)$ we get
		\begin{align}
			|g(x,t)|&=\bigg\rvert \int_{I_Q\cap I_R} \frac{\partial_t(\varphi\phi)(x,u)}{|t-u|^{\beta}}\dd u \bigg\rvert \leq \int_{I_Q\cap I_R}|\partial_t(\varphi\phi)(x,u)|\bigg\rvert \frac{1}{|t-u|^{\beta}}-\frac{1}{|t-c_{Q\cap R}|^{\beta}} \bigg\rvert\dd u \nonumber\\
			&\lesssim \frac{\ell(I_Q\cap I_R)}{|t-c_{Q\cap R}|^{1+\beta}} \int_{I_Q\cap I_R} |\partial_t(\varphi\phi)(x,u)|\dd u \nonumber\\
			&\lesssim_{\beta} \frac{\ell(I_Q\cap I_R)}{|t-c_{Q\cap R}|^{1+\beta}}\bigg( \frac{1}{\ell(Q)^{2s}} + \frac{1}{\ell(R)^{2s}}\bigg)\ell(I_Q\cap I_R)\lesssim \frac{\ell(I_Q\cap I_R)}{|t-c_{Q\cap R}|^{1+\beta}}. \label{new_eq3.1}
		\end{align} 
		
		That is, $|g|$ decays as $|t|^{-1-\beta}$ for large values of $t$. Hence, since $\text{supp}(g)\subset (Q_1\cap R_1)\times \mathbb{R}$, this implies $g\in L^1(\mathbb{R}^{n+1})$. Then, for any constant $c\in \mathbb{R}$ we have
		\begin{align*}
			|\langle f, g \rangle | = \bigg\rvert \int (f-c)g \bigg\rvert \leq \int_{2R}|f-c||g|+\int_{\mathbb{R}^{n+1}\setminus{2R}}|f-c||g|=: I_1+I_2,
		\end{align*}
		where we have used that $g$ has null integral (it can be easily checked taking the Fourier transform, for example). To study $I_1$, observe that for $t\in 4I_R$ we get
		\begin{align*}
			|g(x,t)|&\leq \|\partial_t(\varphi \phi)\|_\infty \int_{-5\ell(I_R)}^{5\ell(I_R)}\frac{\dd u}{|u|^{\beta}} \lesssim_{\beta} \bigg( \frac{1}{\ell(R)^{2s}}+\frac{1}{\ell(Q)^{2s}} \bigg)\ell(I_R)^{1-\beta}\lesssim \ell(R)^{-2s\beta},
		\end{align*}
		since $\ell(R)\leq \ell(Q)$. Therefore,
		\begin{equation*}
			I_1\lesssim_{\beta} \frac{1}{\ell(R)^{2s\beta}}\int_{2R}|f-c|.
		\end{equation*}
		If $f\in \text{\normalfont{BMO}}_{p_s}$, pick $c:= f_{2R}$, the average of $f$ over $2R$, so that
		\begin{equation*}
			I_1\lesssim_{\beta} \ell(R)^{n+2s(1-\beta)}\|f\|_{\ast,p_s}.
		\end{equation*}
		If $f\in \text{\normalfont{Lip}}_{\alpha,p_s}$, pick $c:= f(\ox_R)$, where $\ox_R$ is the center of $2R$, so that
		\begin{align*}
			I_1\lesssim_{\beta,\alpha} \ell(R)^{n+2s(1-\beta)+\alpha}\|f\|_{\text{\normalfont{Lip}}_{\alpha,p_s}},
		\end{align*}
		and we are done with $I_1$. To study $I_2$, define the $s$-parabolic annuli $A_j:=2^jR\setminus{2^{j-1}R}$ for $j\geq 2$. Then, since $\text{supp}(g)\subset (Q_1\cap R_1)\times \mathbb{R}$ applying \ref{new_eq3.1} we have
		\begin{align}
			\label{eq1.2.2}
			I_2&=\sum_{j=2}^{\infty} \int_{A_j\cap \text{supp}(g)}|f(\ox)-c||g(\ox)|\dd \ox \lesssim_{\beta} \frac{1}{\ell(R)^{2s\beta}} \sum_{j=2}^\infty \frac{1}{2^{2s(1+\beta)j}} \int_{A_j\cap \text{supp}(g)}|f(\ox)-c|\dd \ox
		\end{align}
		If $f\in \text{\normalfont{BMO}}_{p_s}$, pick again $c:=f_{2R}$ and observe
		\begin{align*}
			I_2\lesssim_{\beta} \frac{1}{\ell(R)^{2s\beta}} \sum_{j=2}^\infty \frac{1}{2^{2s(1+\beta)j}}\bigg( \int_{A_i\cap\text{supp}(g)} |f(\ox)-f_{2^jR}|\dd\ox + \int_{A_j\cap\text{supp}(g)}|f_{2^jR}-f_{2R}|\dd \ox \bigg),
		\end{align*}
		Regarding the first integral, apply Hölder's inequality (with exponent $q$, to be fixed later) and John-Nirenberg's, so that
		\begin{align*}
			\int_{A_i\cap\text{supp}(g)} |f(\ox)-&f_{2^jR}|\dd\ox  \leq \bigg(\int_{A_j\cap\text{supp}(g)} |f(\ox)-f_{2^jR}|^q\dd\ox \bigg)^{\frac{1}{q}}|\text{supp}(g)\cap 2^jR|^{\frac{1}{q'}}\\
			&\leq \|f\|_{\ast,p_s} (2^j\ell(R))^{\frac{n+2s}{q}}\big[ 2^{2sj}\ell(R)^{n+2s} \big]^{\frac{1}{q'}} = \|f\|_{\ast,p_s} 2^{j(\frac{n}{q}+2s)}\ell(R)^{n+2s}.
		\end{align*}
		For the second integral we apply \cite[Ch.VI, Lemma 1.1]{Ga} to deduce $|f_{2^jR}-f_{2R}|\lesssim j\|f\|_{\ast,p_s}\leq j$, so
		\begin{align*}
			\int_{A_j\cap\text{supp}(g)}|f_{2^jR}-f_{2R}|\dd \ox \lesssim j\,|\text{supp}(g)\cap 2^jR|=j \|f\|_{\ast,p_s} \,2^{2sj}\ell(R)^{n+2s}.
		\end{align*}
		Therefore, choosing $q>\frac{n}{2s\beta}$,
		\begin{equation*}
			I_2\lesssim_{\beta} \frac{\|f\|_{\ast,p_s}}{\ell(R)^{2s\beta}}\sum_{j=2}^{\infty}\frac{1}{2^{2s(1+\beta)j}}\big( 2^{j(\frac{n}{q}+2s)}+j2^{2sj} \big)\ell(R)^{n+2s}\lesssim \|f\|_{\ast,p_s} \ell(R)^{n+2s(1-\beta)}.
		\end{equation*}
		If on the other hand $f\in \text{\normalfont{Lip}}_{\alpha,p_s}$, pick $c:=f(\ox_R)$ so that Hölder's inequality in \eqref{eq1.2.2} yields
		\begin{align*}
			I_2&\lesssim_{\beta} \frac{\|f\|_{\text{\normalfont{Lip}}_{\alpha,p_s}}}{\ell(R)^{2s\beta}} \sum_{j=2}^\infty \frac{\big(2^{j}\ell(R)\big)^\alpha}{2^{2s(1+\beta)j}}|\text{supp}(g)\cap 2^jR|\lesssim_{\beta,\alpha} \|f\|_{\text{\normalfont{Lip}}_{\alpha,p_s}}  \sum_{j=2}^\infty \frac{1\ell(R)^{n+2s(1-\beta)+\alpha}}{2^{(2s\beta-\alpha)j}},
		\end{align*}
		being this last sum convergent because $\alpha<2s\beta$, so we are done.
	\end{proof}
	
	\begin{thm}
		\label{Growth_thm2}
		Let $\alpha\in(0,1)$ and $f:\mathbb{R}^{n+1}\to\mathbb{R}$. Assume $\|\nabla_x\varphi\|_\infty\leq \ell(Q)^{-1}$ and $\|\nabla_x\phi\|_\infty\leq \ell(R)^{-1}$. Then, if $\ell(R)\leq \ell(Q)$, for each $i=1,\ldots, n$ we have
		\begin{enumerate}
			\item[1.] If $f\in \text{\normalfont{BMO}}_{p_s}$,
			\begin{equation*}
				|\langle f, \partial_{x_i}(\varphi \phi)\rangle|\lesssim_{\beta} \|f\|_{\ast,p_s} \ell(R)^{n+2s-1}.
			\end{equation*}
			\item[2.] If $f\in \text{\normalfont{Lip}}_{\alpha,p_s}$,
			\begin{equation*}
				|\langle f, \partial_{x_i}(\varphi \phi)\rangle|\lesssim_{\beta,\alpha} \|f\|_{\text{\normalfont{Lip}}_{\alpha,p_s}} \ell(R)^{n+2s-1+\alpha}.
			\end{equation*}
		\end{enumerate}
	\end{thm}
	\begin{proof}
		First, observe that for any real constant $c$, we have the identity
		\begin{equation*}
			\langle f, \partial_{x_i}(\varphi\phi)\rangle = \langle f-c, \partial_{x_i}(\varphi\phi)\rangle,
		\end{equation*}
		Therefore,
		\begin{align*}
			\langle f, \partial_{x_i}(\varphi\phi)\rangle &= \bigg\rvert \int_{Q\cap R}f(\ox) \partial_{x_i}(\varphi\phi)(\ox)\dd \ox\bigg\rvert \leq \int_{Q\cap R}\big\rvert f(\ox)-c\big\rvert \big\rvert\partial_{x_i}(\varphi\phi)(\ox)\big\rvert \dd \ox\\
			&\leq \bigg( \int_R \big\rvert f(\ox)-c\big\rvert^2\dd \ox \bigg)^{1/2}\bigg( \int_{Q\cap R} \big\rvert\partial_{x_i}(\varphi\phi)(\ox)\big\rvert^2 \dd \ox\bigg)^{1/2}\\
			&\lesssim \bigg( \int_R \big\rvert f(\ox)-c\big\rvert^2\dd \ox \bigg)^{1/2}\bigg( \int_{Q\cap R} \Big[ \|\nabla_x \varphi\|_\infty^2\|\phi\|_\infty^2 + \| \varphi\|_\infty^2\|\nabla_x\phi\|_\infty^2 \Big] \dd \ox\bigg)^{1/2}\\
			&\leq \bigg( \int_R \big\rvert f(\ox)-c\big\rvert^2\dd \ox \bigg)^{1/2}|Q\cap R|^{1/2}\Big(|Q|^{-\frac{1}{n+2s}} + |R|^{-\frac{1}{n+2s}}\Big)\\
			&\leq \bigg( \int_R \big\rvert f(\ox)-c\big\rvert^2\dd \ox \bigg)^{1/2}\bigg(\frac{|Q\cap R|^{1/2}}{\ell(Q)}+\ell(R)^{\frac{n+2s}{2}-1}\bigg)\\
			&=\bigg( \int_R \big\rvert f(\ox)-c\big\rvert^2\dd \ox \bigg)^{1/2}\bigg(|Q\cap R|^{\frac{n+2s-2}{2(n+2s)}}\frac{|Q\cap R|^{\frac{1}{n+2s}}}{\ell(Q)} + \ell(R)^{\frac{n+2s}{2}-1}\bigg)\\
			&\leq \bigg( \int_R \big\rvert f(\ox)-c\big\rvert^2\dd \ox \bigg)^{1/2} \ell(R)^{\frac{n+2s}{2}-1}.
		\end{align*}
		Now, if $f\in \text{\normalfont{BMO}}_{p_s}$, choose $c:=f_{R}$ and apply an $s$-parabolic version of John-Nirenberg's inequality (that admits an analogous proof) to deduce estimate \textit{1}. On the other hand, if $f\in \text{Lip}_{\alpha,p_s}$, choose $c:=f(\ox_R)$ to obtain estimate \textit{2}.
	\end{proof}
	
	\begin{thm}
		\label{thm0.3}
		Let $\beta\in(0,1)$, $\alpha\in(0,1)$ and $f:\mathbb{R}^{n+1}\to\mathbb{R}$. Assume that $\varphi$ and $\phi$ are $\pazocal{C}^{2}$ with $\|\nabla_x\varphi\|_\infty \leq \ell(Q)^{-1}$, $\|\Delta \varphi\|_\infty \leq \ell(Q)^{-2}$ and $\|\nabla_x\phi\|_\infty \leq \ell(R)^{-1}$, $\|\Delta \phi\|_\infty \leq \ell(R)^{-2}$. Then, if $\ell(R)\leq \ell(Q)$,
		\begin{enumerate}
			\item[1.] If $f \in\text{\normalfont{BMO}}_{p_s}$,
			\begin{equation*}
				|\langle f, (-\Delta)^{\beta}(\varphi \phi)\rangle|\lesssim_{\beta} \|f\|_{\ast,p_s} \ell(R)^{n+2(s-\beta)}.
			\end{equation*}
			\item[2.] If $f\in \text{\normalfont{Lip}}_{\alpha,p_s}$ and $\alpha<2\beta$,
			\begin{equation*}
				|\langle f, (-\Delta)^{\beta}(\varphi \phi)\rangle|\lesssim_{\beta,\alpha} \|f\|_{\text{\normalfont{Lip}}_{\alpha,p_s}} \ell(R)^{n+2(s-\beta)+\alpha}.
			\end{equation*}
		\end{enumerate}
	\end{thm}
	\begin{proof}
		Observe that for any real constant $c$,
		\begin{align*}
			|\langle f, (-\Delta)^{\beta}&(\varphi\phi)\rangle| = |\langle f-c, (-\Delta)^{\beta}(\varphi\phi)\rangle| \\
			&\leq \int_{2R_1\times (I_Q\cap I_R)}|f(\ox)-c| \big\rvert (-\Delta)^{\beta}(\varphi\phi)(\ox) \big\rvert \dd \ox\\
			&\hspace{1cm}+\int_{(\mathbb{R}^n\setminus{2R_1})\times (I_Q\cap I_R)}| f(\ox)-c| \big\rvert (-\Delta)^{\beta}(\varphi\phi)(\ox) \big\rvert \dd \ox=:I_1+I_2.
		\end{align*}
		Regarding $I_1$, observe that for any $\ox\in \mathbb{R}^{n+1}$ by Remark \ref{lem2.2.1} we have $|(-\Delta)^{\beta}(\varphi\phi)(\ox)|\lesssim_{\beta} \ell(R)^{-2\beta}.$ Therefore,
		\begin{align*}
			I_1\lesssim_{\beta} \frac{1}{\ell(R)^{2\beta}}\int_{2R_1\times (I_Q\cap I_R)}|f(\ox)-c| \dd \ox.
		\end{align*}
		Let $\ox_0$ be the center of $2R_1\times(I_Q\cap I_R)$. Choosing $c:=f_{2R}$ or $c:=f(\ox_0)$ for $f\in \text{\normalfont{BMO}}_{p_s}$ or $f\in\text{Lip}_{\alpha,p_s}$ respectively, we obtain the desired estimates.
        
		Let us turn to $I_2$. We first notice that, taking the Fourier transform, the operator $(-\Delta)^\beta$ can be rewritten as
		\begin{equation*}
			(-\Delta)^\beta(\cdot) \simeq_{\beta} \sum_{j=1}^n \partial_{x_j}\bigg( \frac{1}{|x|^{n+2\beta-2}} \bigg) \ast_n \partial_{x_j}(\cdot),
		\end{equation*}
		where the notation $\ast_n$ is used to stress that the convolution is taken with respect the first $n$ spatial variables. With this, if $x_0\in \mathbb{R}^n$ denotes the center of $Q_1\cap R_1$, for any $\ox\in (\mathbb{R}^n\setminus{2R_1})\times (I_Q\cap I_R)$ we get
		\begin{align}
			\big\rvert (-\Delta)^{\beta}(\varphi\phi)(\ox) \big\rvert &\lesssim_{\beta} \sum_{j=1}^{n}\bigg\rvert \int_{Q_1\cap R_1} \partial_j(\varphi\phi)(z,t)\frac{z_j-x_j}{|z-x|^{n+2\beta}} \dd z\bigg\rvert\nonumber \\
			&=\sum_{j=1}^{n}\bigg\rvert \int_{Q_1\cap R_1} \partial_j(\varphi\phi)(z,t)\bigg(\frac{z_j-x_j}{|z-x|^{n+2\beta}}-\frac{x_{0,j}-x_j}{|x_{0}-x|^{n+2\beta}}\bigg) \dd z \bigg\rvert\nonumber\\
			&\lesssim_{\beta} \sum_{j=1}^{n} \frac{\ell(R)}{|x_{0}-x|^{n+2\beta}}\|\nabla_x(\varphi\phi)\|_{\infty}\ell(R)^n \lesssim \frac{\ell(R)^{n}}{|x_{0}-x|^{n+2\beta}}, \label{eq2.2.2}
		\end{align}
		by the mean value theorem. So, defining the cylinders $C_j:=2^jR_1\times (I_Q\cap I_R)$ for $j\geq 1$, relation \eqref{eq2.2.2} implies
		\begin{align*}
			I_2\lesssim_{\beta} \frac{1}{\ell(R)^{2\beta}} \sum_{j=1}^\infty \frac{1}{2^{j(n+2\beta)}} \int_{C_{j+1}\setminus{C_j}} \big\rvert f(\ox)-c \big\rvert \dd \ox,
		\end{align*}
		If $f\in \text{\normalfont{BMO}}_{p_s}$, we choose $c:=f_{2R}$ and proceed as in Theorem \ref{Growth_thm1},
		\begin{align*}
			I_2&\lesssim_{\beta} \frac{1}{\ell(R)^{2\beta}} \sum_{j=1}^\infty \frac{1}{2^{j(n+2\beta)}}\bigg( \int_{C_{j+1}\setminus{C_j}} \big\rvert f(\ox)-f_{2^jR} \big\rvert \dd \ox+\int_{C_{j+1}\setminus{C_j}} \big\rvert f_{2R}-f_{2^jR} \big\rvert \dd \ox\bigg)\\
			&\lesssim \frac{\|f\|_{\ast,p_s}}{\ell(R)^{2\beta}} \sum_{j=1}^\infty \frac{1}{2^{j(n+2\beta)}}\bigg[ \big( 2^j\ell(R) \big)^{\frac{n+2s}{q}}|C_{j+1}\setminus{C_j}|^{\frac{1}{q'}}+ j |C_{j+1}\setminus{C_j}| \bigg]\\
			&\lesssim \frac{\|f\|_{\ast,p_s}}{\ell(R)^{2\beta}} \sum_{j=1}^\infty \frac{1}{2^{j(n+2\beta)}} \Big[ \ell(R)^{n+2s}2^{j(n+\frac{2s}{q})+\frac{2s}{q'}}+j\ell(R)^{n+2s}\,2^{jn+2s} \Big]\\
			&\lesssim \|f\|_{\ast,p_s}\ell(R)^{n+2(s-\beta)}\bigg(1+\sum_{j=1}^\infty \frac{2^{\frac{2}{q'}}}{2^{j(2\beta-\frac{2s}{q})}}\bigg).
		\end{align*}
		Fixing $q>s/\beta$ so that this last sum is convergent, proves the result.
        
		On the other hand, if $f\in \text{Lip}_{\alpha,p_s}$ let $c:=f(\ox_0)$ and also proceed as in Theorem \ref{Growth_thm1} to deduce
		\begin{align*}
			I_2\lesssim_{\beta,\alpha} \frac{\|f\|_{\text{Lip}_{\alpha,p_s}}}{\ell(R)^{2\beta}} \sum_{j=1}^\infty \frac{(2^j\ell(R))^\alpha}{2^{j(n+2\beta)}}  |C_{j+1}\setminus{C_j}|\lesssim \|f\|_{\text{Lip}_{\alpha,p_s}}\ell(R)^{n+2(s-\beta)+\alpha}\sum_{j=1}^\infty \frac{1}{2^{(2\beta-\alpha)j}},
		\end{align*}
		that is a convergent sum since $\alpha<2\beta$ by hypothesis.
	\end{proof}
	
	Recall that given $f: \mathbb{R}^{n+1}\to \mathbb{R}$ and $\beta\in(0,n)$, we define its $n$-dimensional $\beta$-Riesz transform (whenever it makes sense) as
	\begin{align*}
		\pazocal{I}_\beta^n f(\cdot,t):= \frac{1}{|x|^{n-\beta}}\ast f (\cdot,t),
	\end{align*}
	for each $t$, where the convolution is thought in a principal value sense. Let us observe that for a test function $f$, for example, the operators $\pazocal{I}_\beta^n$ and $\partial_{x_i}$ commute.
	\begin{thm}
		\label{Growth_thm4}
		Let $\beta\in(0,1)$, $\alpha\in(0,1-\beta)$ and $f:\mathbb{R}^{n+1}\to\mathbb{R}$. Assume $\|\nabla_x\varphi\|_\infty \leq \ell(Q)^{-1}$ and $\|\nabla_x\phi\|_\infty \leq \ell(R)^{-1}$. Then, if $\ell(R)\leq \ell(Q)$, for each $i=1,\ldots, n$ we have
		\begin{enumerate}
			\item[1.] If $f \in\text{\normalfont{BMO}}_{p_s}$,
			\begin{equation*}
				|\langle f, \partial_{x_i}[\pazocal{I}^n_\beta (\varphi \phi)]\rangle|\lesssim_{\beta} \|f\|_{\ast,p_s} \ell(R)^{n+2s+\beta-1}.
			\end{equation*}
			\item[2.] If $f\in \text{\normalfont{Lip}}_{\alpha,p_s}$,
			\begin{equation*}
				|\langle f, \partial_{x_i}[\pazocal{I}^n_\beta (\varphi \phi)]\rangle|\lesssim_{\beta,\alpha} \|f\|_{\text{\normalfont{Lip}}_{\alpha,p_s}} \ell(R)^{n+2s+\beta+\alpha-1}.
			\end{equation*}
		\end{enumerate}
	\end{thm}
	\begin{proof}
		Notice that for any $c\in\mathbb{R}$,
		\begin{align*}
			|\langle f, \partial_{x_i}[\pazocal{I}^n_\beta (\varphi \phi)]\rangle| &= |\langle f-c, \partial_{x_i}[\pazocal{I}^n_\beta (\varphi \phi)]\rangle| \\
			&\leq \int_{2R_1\times (I_Q\cap I_R)}|f(\ox)-c| \big\rvert \partial_{x_i}[\pazocal{I}^n_\beta (\varphi \phi)](\ox) \big\rvert \dd \ox\\
			&\hspace{1cm}+\int_{(\mathbb{R}^n\setminus{2R_1})\times (I_Q\cap I_R)}| f(\ox)-c| \big\rvert \partial_{x_i}[\pazocal{I}^n_\beta (\varphi \phi)](\ox) \big\rvert \dd \ox=:I_1+I_2.
		\end{align*}
		Regarding $I_1$, we have for some conjugate exponents $q,q'$ to be fixed later on,
		\begin{align*}
			I_1&\lesssim \bigg( \int_{2R}|f(\ox)-c|^{q'}\dd \ox\bigg)^{\frac{1}{q'}}\bigg( \int_{I_Q\cap I_R} \int_{2R_1} \big\rvert \pazocal{I}_\beta^n[\partial_{x_i}(\varphi\phi)](x,t)\big\rvert^q \dd x \dd t \bigg)^{\frac{1}{q}}\\
			&\lesssim \bigg( \int_{2R}|f(\ox)-c|^{q'}\dd \ox\bigg)^{\frac{1}{q'}} \bigg( \int_{I_Q\cap I_R} \|\pazocal{I}^n_\beta[\partial_{x_i}(\varphi\phi)(\cdot,t)]\|_q^q \dd t  \bigg)^{\frac{1}{q}}.
		\end{align*}
		Choosing $q>\frac{n}{n-\beta}$, we shall apply \cite[Theorem 6.1.3]{Gr} and obtain
		\begin{align*}
			I_1&\lesssim_{\beta} \bigg( \int_{2R}|f(\ox)-c|^{q'}\dd \ox\bigg)^{\frac{1}{q'}} \bigg( \int_{I_Q\cap I_R} \|\partial_{x_i}(\varphi\phi)(\cdot,t)\|_{\frac{qn}{n+q\beta}}^q \dd t  \bigg)^{\frac{1}{q}}\\
			&\lesssim \bigg( \int_{2R}|f(\ox)-c|^{q'}\dd \ox\bigg)^{\frac{1}{q'}}\ell(R)^{\frac{n+q\beta+2s}{q}-1}.
		\end{align*}
		If we assume $f\in \text{\normalfont{BMO}}_{p_s}$, we choose $c:=f_{2R}$ and apply a $s$-parabolic version of John-Nirenberg's inequality to deduce
		\begin{equation*}
			I_1\lesssim_{\beta} \|f\|_{\ast,p_s}\ell(R)^{\frac{n+2s}{q'}}\ell(R)^{\frac{n+q\beta+2s}{q}-1}=\|f\|_{\ast,p_s}\ell(R)^{n+2s+\beta-1}.
		\end{equation*}
		If we assume $f\in \text{Lip}_{\alpha,p_s}$, we choose $c:=f(\ox_R)$, being $\ox_R$ the center of $R$, and obtain
		\begin{equation*}
			I_1\lesssim_{\beta,\alpha} \|f\|_{\text{Lip}_\alpha,p_s}\ell(R)^{\frac{n+2s}{q'}+\alpha}\ell(R)^{\frac{n+q\beta+2s}{q}-1}=\|f\|_{\ast,p_s}\ell(R)^{n+2s+\beta+\alpha-1}.
		\end{equation*}
		To study $I_2$, we proceed as in Theorem \ref{thm0.3}. For any $\ox\in (\mathbb{R}^n\setminus{2R_1})\times (I_Q\cap I_R)$, if $x_0\in \mathbb{R}^n$ denotes the center of $Q_1\cap R_1$, by the mean value theorem we get
		\begin{align*}
			\big\rvert \pazocal{I}_\beta^n[\partial_{x_i}(\varphi\phi)](\ox) \big\rvert &=\bigg\rvert \int_{Q_1\cap R_1} \partial_{x_i}(\varphi\phi)(z,t)\frac{1}{|z-x|^{n-\beta}} \dd z\bigg\rvert \\
			&=\bigg\rvert \int_{Q_1\cap R_1} \partial_{x_i}(\varphi\phi)(z,t)\bigg(\frac{1}{|z-x|^{n-\beta}}-\frac{1}{|x_{0}-x|^{n-\beta}}\bigg) \dd z \bigg\rvert\\
			&\lesssim_{\beta} \sum_{j=1}^{n} \frac{\ell(R)}{|x_{0}-x|^{n-\beta+1}}\|\nabla_x(\varphi\phi)\|_{\infty}\,\ell(R)^n \lesssim \frac{\ell(R)^{n}}{|x_{0}-x|^{n-\beta+1}} . 
		\end{align*}
		This way, putting $C_j:=2^jR_1\times (I_Q\cap I_R)$ for $j\geq 1$, as in Theorem \ref{thm0.3},
		\begin{align*}
			I_2\lesssim_{\beta} \frac{1}{\ell(R)^{-\beta+1}} \sum_{j=1}^\infty \frac{1}{2^{j(n-\beta+1)}} \int_{C_{j+1}\setminus{C_j}} \big\rvert f(\ox)-c \big\rvert \dd \ox.
		\end{align*}
		The case $f\in \text{\normalfont{BMO}}_{p_s}$ is dealt with analogously as in Theorem \ref{thm0.3}, obtaining
		\begin{align*}
			I_2\lesssim_{\beta} \|f\|_{\ast,p_s}\ell(R)^{n+2s+\beta-1}\sum_{j=1}^\infty \frac{1}{2^{j(n-\beta+1)}} \Big[ 2^{j(n+\frac{2s}{q})}+j\,2^{jn} \Big],
		\end{align*}
		so choosing $q>\frac{2s}{1-\beta}$ we are done. Observe that we also need $\beta<1$ in order for the above sum to converge. The case $f\in \text{Lip}_{\alpha,p_s}$ can be dealt with as follows
		\begin{align*}
			I_2&\lesssim_{\beta,\alpha} \frac{\|f\|_{\text{Lip}_{\alpha,p_s}}}{\ell(R)^{-\beta+1}} \sum_{j=1}^\infty \frac{(2^j\ell(R))^\alpha }{2^{j(n-\beta+1)}} |C_{j+1}\setminus{C_j}|\lesssim \|f\|_{\text{Lip}_{\alpha,p_s}}\sum_{j=1}^\infty \frac{\ell(R)^{n+2s+\beta+\alpha-1}}{2^{(1-\beta-\alpha)j}},
		\end{align*}
		and this sum is convergent by the hypothesis $\alpha<1-\beta$.
	\end{proof}
	
	\section{Potentials of positive measures with growth restrictions}
	\label{sec2.3}
	
	The main goal of this section is to deduce some important $\text{BMO}_{p_s}$ and $\text{Lip}_{\alpha,p_s}$ estimates of potentials of the form $\partial_t^\beta P_s\ast \mu$, where $\mu$ is a finite positive Borel measure with some upper $s$-parabolic growth. We begin by proving a generalization of \cite[Lemma 4.2]{MaPrTo} and \cite[Lemma 7.2]{MaPr}.
	\begin{lem}
		\label{Pos_lem1}
		Let $s\in(0,1]$, $\eta \in (0,1)$ and $\mu$ be a positive measure in $\mathbb{R}^{n+1}$ which has upper $s$-parabolic growth of degree $n+2s\eta$. Then
		\begin{equation*}
			\|P_s\ast \mu \|_{\text{\normalfont{Lip}}_{\eta,t}}\lesssim_{\eta} 1.
		\end{equation*}
	\end{lem}
	\begin{proof}
		Let $\ox:=(x,t), \widehat{x}:=(x,\tau)$ be fixed points in $\mathbb{R}^{n+1}$ with $t\neq \tau$, and set $\ox_0:=(\ox+\widehat{x})/2$. Writing $\oy:=(y,u)$ and  $B_0:=B(\ox_0,|\ox-\widehat{x}|_{p_s})=B(\ox_0,|t-\tau|^{\frac{1}{2s}})$, we split
		\begin{align*}
			|P_s\ast \mu(\ox)&-P_s\ast \mu(\widehat{x})|\\
			&\leq \int_{\mathbb{R}^{n+1}\setminus{2B_0}}|P_s(x-y,t-u)-P_s(x-y,\tau-u)|\text{d}\mu(\oy)\\
			&\hspace{1cm}+\int_{2B_0}|P_s(x-y,t-u)-P_s(x-y,\tau-u)|\text{d}\mu(\oy)=:I_1+I_2.
		\end{align*}
		Defining the $s$-parabolic annuli $A_j:=2^{j+1}B_0\setminus{2^{j}B_0}$ for $j\geq 1$ and arguing as in the last estimate of Theorem \ref{C-Z_thm2} we get
		\begin{align*}
			I_1\lesssim \sum_{j\geq 1} \int_{A_j} \frac{|t-\tau|}{|\ox-\oy|_{p_s}^{n+2s}}\text{d}\mu(\oy) &\lesssim |t-\tau|\sum_{j\geq 1}\frac{\mu\big(2^{j+1}B_0\big)}{\big(2^{j}|t-\tau|^{\frac{1}{2s}}\big)^{n+2s}}\\
			& \lesssim |t-\tau|^{\eta}\sum_{j\geq 1}\frac{1}{2^{2s(1-\eta)}} \simeq_{\eta} |t-\tau|^{\eta},
		\end{align*}
		that is the desired estimate. Regarding $I_2$, observe that
		\begin{equation*}
			I_2\leq P_s\ast (\chi_{2B_0}\mu)(\ox)+P_s\ast(\chi_{2B_0}\mu)(\widehat{x}).
		\end{equation*}
		Notice now that
		\begin{align*}
			P_s\ast (\chi_{2B_0}\mu)(\ox)\lesssim \int_{2B_0}\frac{\text{d}\mu(\oy)}{|\ox-\oy|_{p_s}^n}&\leq \int_{|\ox-\oy|_{p_s}\leq 5|t-\tau|^{\frac{1}{2s}}} \frac{\text{d}\mu(\oy)}{|\ox-\oy|_{p_s}^n} \lesssim_\eta |t-\tau|^{\eta},
		\end{align*}
		where we have split the latter domain of integration into (decreasing) $s$-parabolic annuli. Since this also holds replacing $\ox$ by $\widehat{x}$, we also have $I_2\lesssim |t-\tau|^{\eta}$ and we are done.
	\end{proof}
	
	The above result allows us to prove that, given a positive measure as in the above statement, we can ensure that the potential $\partial_t^{\beta}P_s\ast \mu$ already belongs to $\text{BMO}_{p_s}$.
	
	\begin{lem}
		\label{Pos_lem2}
		Let $s\in(0,1]$, $\beta \in (0,1)$. Let $\mu$ be a finite positive Borel measure in $\mathbb{R}^{n+1}$ with upper $s$-parabolic growth of degree $n+2s\beta$. Then,
		\begin{equation*}
			\big\| \partial_t^{\beta}P_s\ast \mu \big\|_{\ast,p_s}\lesssim_{\beta} 1.
		\end{equation*}
	\end{lem}
	\begin{proof}
		Fix $\ox_0\in\mathbb{R}^{n+1}$ and $r>0$. Consider the $s$-parabolic ball $B:=B(\ox_0,r)=B_0\times I_0\subset \mathbb{R}^n\times \mathbb{R}$ and a constant $c_B$ to be determined later. We want to show that $c_B$ can be chosen so that
		\begin{align*}
			\frac{1}{|B|}\int_B |\partial_t^{\beta}P_s\ast \mu(\oy)-c_B|\dd \oy \lesssim_{\beta}1.
		\end{align*}
		To that end, begin by considering the following sets, which define a partition of $\mathbb{R}^{n+1}$:
		\begin{align*}
			R_1:=5B, \hspace{0.75cm} R_2:=\mathbb{R}^{n+1}\setminus{(5B_0\times\mathbb{R})}, \hspace{0.75cm} R_3:=(5B_0\times \mathbb{R})\setminus{5B},
		\end{align*}
		as well as their corresponding characteristic functions $\chi_1,\chi_2$ and $\chi_3$. Bearing in mind the estimates proved in Theorems \ref{lem2.4} and \ref{lem3.5} for $\partial_t^{\beta}P_s$ and the fact that $\mu$ is finite, it is clear that the quantity $|\partial_t^{\beta}P_s\ast(\chi_2\mu)(\ox_0) |$ is also finite. Moreover, notice that $|\partial_t^{\beta}P_s|$ is bounded by $s$-parabolically homogeneous functions of degree $-n-2s\beta$ for any dimension. In fact, we deduce the following estimates: given any $\varepsilon,\alpha>0$, we obtain if $n>2$,
		\begin{equation*}
			|\partial_t^{\beta}P_s(\ox)|\lesssim_{\beta} \frac{1}{|x|^{n-2s}|\ox|_{p_s}^{2s(1+\beta)}}\leq \frac{1}{|x|^{n-\varepsilon}|t|^{\frac{\varepsilon+2s\beta}{2s}}}, \qquad \text{if }\;\varepsilon<2s(1-\beta).
		\end{equation*}
		For $n=2$,
		\begin{align*}
			\text{if $s<1$,}& \quad |\partial_t^{\beta}P_s(\ox)|\lesssim_{\beta} \frac{1}{|x|^{2-2s}|\ox|_{p_s}^{2s(1+\beta)}}\leq \frac{1}{|x|^{2-\varepsilon}|t|^{\frac{\varepsilon+2s\beta}{2s}}}, \qquad \text{if }\;\varepsilon<2s(1-\beta),\\
			\text{if $s=1$,}& \quad |\partial_t^{\beta}W(\ox)|\lesssim_{\beta,\alpha} \frac{1}{|x|^{\alpha}|\ox|_{p_s}^{2+2\beta-\alpha}}\leq \frac{1}{|x|^{\alpha}|t|^{1+\beta-\frac{\alpha}{2}}}, \qquad \hspace{0.6cm} \text{if }\; 2\beta<\alpha<2.
		\end{align*}
		And for $n=1$,
		\begin{align*}
			\text{if $s<1$,}& \quad |\partial_t^{\beta}P_s(\ox)|\lesssim_{\beta,\alpha}\frac{1}{|x|^{1-2s+\alpha}|\ox|_{p_s}^{2s(1+\beta)-\alpha}}\leq \frac{1}{|x|^{1-2s+\alpha}|t|^{1+\beta-\frac{\alpha}{2s}}}, \quad \text{if }\; 2s\beta<\alpha<2s,\\
			\text{if $s=1$,}& \quad |\partial_t^{\beta}W(\ox)|\lesssim_{\beta} \frac{1}{|\ox|_{p_s}^{1+2\beta}}\leq \frac{1}{|x|^{\varepsilon}|t|^{\frac{1+2\beta-\varepsilon}{2}}}, \qquad \hspace{3.025cm} \text{if }\; 2\beta-1<\varepsilon<1.
		\end{align*}
		In light of the above inequalities, and using that $\beta<1$, it is clear that $\partial_t^{\beta}P_s$ defines a $\pazocal{L}^{n+1}$-locally integrable function in $\mathbb{R}^{n+1}$ once endowed with the $s$-parabolic distance. Hence, there exists some $\overline{\xi}_0\in B$ (that we may think as close as we need to $\ox_0$) such that $|\partial_t^{\beta}P_s\ast(\chi_3\mu)(\overline{\xi}_0)|$ is finite. Bearing all these observations in mind, we choose $c_B$ to be
		\begin{equation*}
			c_B:=\partial_t^{\beta}P_s\ast(\chi_2\mu)(\ox_0)+\partial_t^{\beta}P_s\ast(\chi_3\mu)(\overline{\xi}_0).
		\end{equation*}
		Therefore, we are interested in bounding by a constant the following quantity:
		\begin{align*}
			\frac{1}{|B|}\int_B |&\partial_t^{\beta}P_s\ast \mu(\oy)-c_B|\dd \oy \leq\frac{1}{|B|}\int_B |\partial_t^{\beta}P_s\ast (\chi_1\mu)(\oy)|\dd \oy\\
			&\hspace{0.5cm}+\frac{1}{|B|}\int_B |\partial_t^{\beta}P_s\ast (\chi_2\mu)(\oy)-\partial_t^{\beta}P_s\ast(\chi_2\mu)(\ox_0)|\dd \oy\\
			&\hspace{0.5cm}+\frac{1}{|B|}\int_B |\partial_t^{\beta}P_s\ast (\chi_3\mu)(\oy)-\partial_t^{\beta}P_s\ast(\chi_3\mu)(\overline{\xi}_0)|\dd \oy=:I_1+I_2+I_{3}.
		\end{align*}
		For $I_1$, simply notice that
		\begin{equation*}
			I_1\leq \frac{1}{|B|}\int_{5B}\bigg( \int_{B}|\partial_t^\beta P_s(\oy-\oz)|\dd \oy \bigg)\text{d}\mu(\oz).
		\end{equation*}
		Using any of the bounds above for $\partial_t^\beta P_s$, depending on $n$ and $s$, integration in polar coordinates yields
		\begin{equation*}
			I_1\lesssim_{\beta} \frac{1}{|B|}r^{2s(1-\beta)}\mu(5B)\lesssim_{\beta} 1.
		\end{equation*}
		Regarding $I_2$, write
		\begin{align*}
			I_2\leq \frac{1}{|B|}\int_{B}\bigg(\int_{R_2}|\partial_t^{\beta}P_s(\oy-\oz)-\partial_t^{\beta}P_s(\ox_0-\oz)|\text{d}\mu(\oz)\bigg)\dd \oy.
		\end{align*}
		If we name $\ox:=\ox_0-\oz$ and $\ox':=\oy-\oz$, we have in particular
		\begin{equation*}
			|\ox-\ox'|_{p_s} = |\ox_0-\oy|_{p_s}\leq r < \frac{|x_0-z|}{2}=\frac{|x|}{2},
		\end{equation*}
		where the second inequality holds because $\oz\in R_2$. Therefore, by the last estimate of Theorems  \ref{lem2.4} and \ref{lem3.5}, writing $2\zeta:=\min\{1,2s\}$ we get
		\begin{align*}
			I_2\lesssim_{\beta} \frac{1}{|B|}\int_{B}\bigg(\int_{R_2} \frac{|\oy-\ox_0|_{p_s}^{2\zeta}}{|x_0-z|^{n+2\zeta-2s}|\ox_0-\oz|_{p_s}^{2s(1+\beta)}} &\text{d}\mu(\oz)\bigg)\dd \oy\\
			&\lesssim r^{2\zeta}\int_{R_2} \frac{\text{d}\mu(\oz)}{|x_0-z|^{n+2\zeta-2s}|\ox_0-\oz|_{p_s}^{2s(1+\beta)}}.
		\end{align*}
		Let us split $R_2$ into proper disjoint pieces. Take the cylinders given by $C_j:= 5^j B_0\times \mathbb{R},\, j\in \mathbb{Z},\;j\geq 1$, as well as the annular cylinders $\widehat{C}_j:=C_{j+1}\setminus{C_j}, \,j\geq 1$. The partition of $R_2$ we are interested in is given by the disjoint union of all the sets $\widehat{C}_j, \,j\geq 1$, which clearly cover $R_2$. Therefore
		\begin{equation}
			\label{Pos_eq1}
			I_2 \lesssim_{\beta} r^{2\zeta}\sum_{j=1}^\infty \int_{\widehat{C}_j} \frac{\text{d}\mu(\oz)}{|x_0-z|^{n+2\zeta-2s}|\ox_0-\oz|_{p_s}^{2s(1+\beta)}}.
		\end{equation}
		At the same time, for each $j\geq 1$, we shall consider a proper partition of $\widehat{C}_j$. Denote $A_k=5^{k+1}B\setminus{5^kB}$ for every positive integer $k$ and define $\widehat{C}_{j,k}:=\widehat{C}_j\cap A_k, \,k\geq 1$. Let us make some observations about the sets $\widehat{C}_{j,k}$. First, notice that by definition, for each $j\geq 1$,
		\begin{equation*}
			\widehat{C}_{j,k} = \big[ (5^{j+1}B_0\setminus 5^{j}B_0)\times \mathbb{R}\big] \cap \big(5^{k+1}B\setminus{5^kB}\big).
		\end{equation*}
		Hence, using that
		\begin{equation*}
			\big[ (5^{j+1}B_0\setminus 5^{j}B_0)\times \mathbb{R}\big] \cap \big(5^{k+1}B\setminus{5^kB}\big)=\varnothing, \hspace{0.5cm} \text{for} \hspace{0.5cm} k<j,
		\end{equation*}
		we have that, in fact, $\widehat{C}_j$ can be covered by $\widehat{C}_{j,k}$ for $k\geq j$, that is
		\begin{equation*}
			\widehat{C}_j =\bigcup_{k=1}^\infty\widehat{C}_{j,k}=\bigcup_{k=j}^\infty\widehat{C}_{j,k}.
		\end{equation*}
		Secondly, in order to estimate $\mu(\widehat{C}_{j,k})$, observe that for any $k\geq j$, by definition, the set $\widehat{C}_{j,k}$ can be written explicitly as follows:
		\begin{align*}
			\widehat{C}_{j,k}&= \big[\big(5^{j+1}B_0\setminus{5^jB_0}\big)\times\mathbb{R}\big]\cap \big(5^{k+1}B\setminus 5^kB\big)\\
			&=\big[\big(5^{j+1}B_0\setminus{5^jB_0}\big)\times\mathbb{R}\big]\\
			&\hspace{2.5cm}\cap\Big\{\big[ \big(5^{k+1}B_0\setminus{5^kB_0}\big)\times 5^{2s(k+1)}I_0 \big]\cup \big[ 5^kB_0\times \big(5^{2s(k+1)}I_0\setminus 5^{2k}I_0\big) \big] \Big\}.
		\end{align*}
		Continue by observing that if $k=j$, the intersection with the second element of the union is empty, so
		\begin{equation*}
			\widehat{C}_{j,j}=\big( 5^{j+1}B_0\setminus{5^{j}B_0}\big)\times 5^{2s(j+1)}I_0;
		\end{equation*}
		while if $k>j$ one has the contrary, that is, the intersection with the first element is empty, and therefore, since $5^{j+1}B_0\setminus{5^jB_0}\subset 5^kB_0$,
		\begin{equation*}
			\widehat{C}_{j,k}=\big( 5^{j+1}B_0\setminus{5^{j}B_0}\big)\times\big[5^{2s(k+1)}I_0\setminus 5^{2sk}I_0\big].
		\end{equation*}
		Observe that $\widehat{C}_{j,j}\subset 5^{j+1}B$, which implies $\mu(\widehat{C}_{j,j})\leq \mu(5^{j+1}B)\lesssim (5^{j+1}r)^{n+2s\beta}$. On the other hand, for $k>j$, notice that the set $\widehat{C}_{j,k}$ can be covered by disjoint temporal translates of $\widehat{C}_{j,j}$, and the number needed to do it is proportional to the ratio between their respective time lengths, that is
		\begin{equation*}
			\frac{2\big( 5^{2s(k+1)}-5^{2sk}\big)}{5^{2s(j+1)}}\simeq \frac{5^{2sk}}{5^{2sj}}.
		\end{equation*}
		Therefore, since this last ratio is also valid for the case $k=j$, for every $k\geq j$ we have
		\begin{equation*}
			\mu(\widehat{C}_{j,k})\simeq \frac{5^{2sk}}{5^{2sj}}\mu(\widehat{C}_{j,j})\lesssim_{\beta} \frac{5^{2sk}}{5^{2sj}}\big(5^{j+1}r\big)^{n+2s\beta}.
		\end{equation*}
		All in all, we finally obtain
		\begin{align*}
			I_2 &\lesssim_{\beta} r^{2\zeta} \sum_{j=1}^\infty\sum_{k\geq j} \int_{\widehat{C}_{j,k}}\frac{\text{d}\mu(\oz)}{|x_0-z|^{n+2\zeta-2s}|\ox_0-\oz|_{p_s}^{2s(1+\beta)}} \lesssim r^{2\zeta} \sum_{j=1}^\infty\sum_{k\geq j} \frac{\mu(\widehat{C}_{j,k})}{(5^jr)^{n+2\zeta-2s}(5^kr)^{2s(1+\beta)}}\\
			&\lesssim_{\beta} \sum_{j=1}^\infty\sum_{k\geq j} \frac{1}{5^{j(2\zeta-2s\beta)}5^{2s\beta k}} = \sum_{k=1}^\infty \frac{1}{5^{2s\beta k}} \sum_{j=1}^k \frac{1}{5^{j(2\zeta-2s\beta)}}\lesssim \sum_{k=1}^\infty \frac{1}{5^{2s\beta k}}\bigg( 1+\frac{1}{5^{(2\zeta-2s\beta)k}} \bigg) \lesssim_{\beta} 1.
		\end{align*}
		Finally, let us study $I_{3}$. Notice that the estimate we want to check is deduced if we prove
		\begin{equation*}
			|\partial_t^{\beta} P_s\ast (\chi_{3}\mu)(\oy)-\partial_t^{\beta} P_s\ast (\chi_{3}\mu)(\overline{\xi}_0)|\lesssim_{\beta} 1,
		\end{equation*}
		that at the same time, can be obtained if we show that for any $\ox,\oy\in B$ we have
		\begin{equation}
			\label{eq2.3.2}
			|\partial_t^{\beta} P_s\ast (\chi_{3}\mu)(\ox)-\partial_t^{\beta} P_s\ast (\chi_{3}\mu)(\oy)|\lesssim_{\beta} 1.
		\end{equation}
		It is clear that it suffices to check the latter estimate in two particular cases: when $\ox$ and $\oy$ share their time coordinate, and when they share their spatial coordinate.
        
		\textit{Case 1: $\ox=(x,t)$ and $\oy=(y,t)$ points of $B$}.\; Let us begin by observing that
		\begin{align*}
			|&\partial_t^{\beta} P_s\ast (\chi_{3}\mu)(\ox)-\partial_t^{\beta} P_s\ast (\chi_{3}\mu)(\oy)|\\
			&\hspace{-0.1cm}=\bigg\rvert\int \frac{P_s\ast(\chi_3\mu)(x,\tau)-P_s\ast(\chi_3\mu)(x,t)}{|\tau-t|^{1+\beta}}\dd \tau-\int\frac{P_s\ast(\chi_3\mu)(y,\tau)-P_s\ast(\chi_3\mu)(y,t)}{|\tau-t|^{1+\beta}}\dd \tau \bigg\rvert\\
			&\hspace{-0.1cm}\leq \int_{|\tau-t|\leq (2r)^{2s}} \frac{|P_s\ast(\chi_3\mu)(x,\tau)-P_s\ast(\chi_3\mu)(x,t)|}{|\tau-t|^{1+\beta}}\dd \tau\\
			&\hspace{0.01cm}+\int_{|\tau-t|\leq (2r)^{2s}} \frac{|P_s\ast(\chi_3\mu)(y,\tau)-P_s\ast(\chi_3\mu)(y,t)|}{|\tau-t|^{1+\beta}}\dd \tau\\
			&\hspace{0.01cm}+\int_{|\tau-t|> (2r)^{2s}} \frac{|P_s\ast(\chi_3\mu)(x,\tau)-P_s\ast(\chi_3\mu)(x,t)-P_s\ast(\chi_3\mu)(y,\tau)+P_s\ast(\chi_3\mu)(y,t)|}{|\tau-t|^{1+\beta}}\dd \tau\\
			&\hspace{0.2cm}=:I_1+I_2+I_{3}.
		\end{align*}
		First, we estimate $I_1$. Argue as in the proof of the last estimate of Theorem \ref{C-Z_thm2} to obtain
		\begin{align*}
			|P_s\ast(\chi_3\mu)(x,\tau)-P_s\ast(\chi_3\mu)(x,t)|\leq |\tau-t|\int_{R_3} \frac{\dd \mu(\oz)}{|\ox-\oz|_{p_s}^{n+2s}}\lesssim_{\beta} \frac{|t-\tau|}{r^{2s(1-\beta)}}
		\end{align*}
		where the last inequality can be obtained by splitting the domain of integration into $s$-parabolic annuli and using the $s$-parabolic growth condition of degree $n+2s\beta$ of $\mu$. Thus,
		\begin{align*}
			I_1\lesssim_{\beta} \frac{1}{r^{2s(1-\beta)}}\int_{|\tau-t|\leq (2r)^{2s}}\frac{\dd \tau}{|\tau-t|^{\beta}}\lesssim_{\beta} \frac{(r^{2s})^{(1-\beta)}}{r^{2s(1-\beta)}}=1.
		\end{align*}
		The arguments to obtain $I_2\lesssim_{\beta} 1$ are exactly the same (just write $y$ instead of $x$ in the lines above). Concerning the term $I_{3}$, we split it as follows
		\begin{align*}
			I_{3}&\leq \int_{|\tau-t|> (2r)^{2s}} \frac{|P_s\ast(\chi_3\mu)(x,\tau)-P_s\ast(\chi_3\mu)(y,\tau)|}{|\tau-t|^{1+\beta}}\dd \tau\\
			&\hspace{2cm}+\int_{|\tau-t|> (2r)^{2s}} \frac{|P_s\ast(\chi_3\mu)(x,t)-P_s\ast(\chi_3\mu)(y,t)|}{|\tau-t|^{1+\beta}}\dd \tau=:I_{31}+I_{32}.
		\end{align*}
		First, let us deal with integral $I_{32}$. Since $(x,t), (y,t)\in B$, 
		\begin{equation*}
			|P_s\ast(\chi_3\mu)(x,t)-P_s\ast(\chi_3\mu)(y,t)|\leq |x-y|\,\|\nabla_xP_s\ast (\chi_3\mu)\|_{\infty,B} 
		\end{equation*}
		Notice that for any $\oz\in B$, by Theorem \ref{C-Z_thm2} and the fact that $s\beta < 1$, we have
		\begin{align*}
			|\nabla_xP_s\ast (\chi_3\mu)(\oz)|\lesssim  \int_{R_3}\frac{|z-w|}{|\oz-\overline{w}|_{p_s}^{n+2}}\text{d}\mu(\overline{w}) \lesssim r \int_{\mathbb{R}^{n+1}\setminus{5B}}\frac{\text{d}\mu(\overline{w})}{|\oz-\overline{w}|_{p_s}^{n+2}}\lesssim_{\beta} r^{2s\beta-1}.
		\end{align*}
		Therefore, since $|x-y|\leq r$,
		\begin{equation}
			\label{eq2.3.3}
			I_{32}\lesssim_{\beta} r^{2s\beta}\int_{|\tau-t|>(2r)^{2s}}\frac{\dd \tau}{|\tau-t|^{1+\beta}}\lesssim_{\beta} r^{2s\beta}\frac{1}{(r^{2s})^{\beta}}=1.
		\end{equation}
		Regarding $I_{31}$, observe that for each $\tau$ the points $(x,\tau)$ and $(y,\tau)$ belong to a temporal translate of $B$ that does not intersect $B$, since $|\tau-t|>(2r)^{2s}$ and $t\in I_0$. We call it $B_\tau$.  Hence, bearing in mind the first estimate of \cite[Lemma 2.1]{MaPrTo} we deduce 
		\begin{align}
			|P_s\ast(\chi_3&\mu)(x,\tau)-P_s\ast(\chi_3\mu)(y,\tau)|\nonumber\\
			&\leq \int_{2B_{\tau}}|P_s((x,\tau)-\overline{w})-P_s((y,\tau)-\overline{w})|\text{d}\mu(\overline{w}) \nonumber\\
			&\hspace{2.5cm}+ \int_{[(5B_0\times \mathbb{R})\setminus{5B}]\cap(2B_{\tau})^c}|P_s((x,\tau)-\overline{w})-P_s((y,\tau)-\overline{w})|\text{d}\mu(\overline{w}) \nonumber\\
			&\lesssim \int_{2B_{\tau}}\frac{\text{d}\mu(\overline{w})}{|(x,\tau)-\overline{w}|_{p_s}^n}+\int_{2B_{\tau}}\frac{\text{d}\mu(\overline{w})}{|(y,\tau)-\overline{w}|_{p_s}^n} \label{eq2.3.4} \\
			&\hspace{2.5cm}+ |x-y|\int_{[(5B_0\times \mathbb{R})\setminus{5B}]\cap(2B_{\tau})^c}|\nabla_xP_s((\widetilde{x},\tau)-\overline{w})|\text{d}\mu(\overline{w})\nonumber\\
			&\lesssim_{\beta} r^{2s\beta} + r\int_{[(5B_0\times \mathbb{R})\setminus{5B}]\cap(2B_{\tau})^c}\frac{|\widetilde{x}-w|}{|(\widetilde{x},\tau)-\overline{w}|_{p_s}^{n+2}}\text{d}\mu(\overline{w})\nonumber \\
			&\lesssim  r^{2s\beta} + r^2\int_{[(5B_0\times \mathbb{R})\setminus{5B}]\cap(2B_{\tau})^c}\frac{\text{d}\mu(\overline{w})}{|(\widetilde{x},\tau)-\overline{w}|_{p_s}^{n+2}}\nonumber \\
			&\leq r^{2s\beta} + r^2\int_{\mathbb{R}^{n+1}\setminus{2B_{\tau}}}\frac{\text{d}\mu(\overline{w})}{|(\widetilde{x},\tau)-\overline{w}|_{p_s}^{n+2}}
			\lesssim_{\beta} r^{2s\beta}\nonumber,
		\end{align}
		where for both integrals in \eqref{eq2.3.4} we have split the domain of integration into (decreasing) $s$-parabolic annuli; while in the remaining term, $\widetilde{x}$ belongs to the segment joining $x$ and $y$ and we have split the domain of integration into $s$-parabolic annuli centered at $(x_0,t+s)$. Hence, similarly to \eqref{eq2.3.3} we get $I_{31}\lesssim_{\beta} 1$ and we are done with \textit{Case 1}.
        
		\textit{Case 2: $\ox=(x,t)$ and $\oy=(x,u)$ points of $B$}.\; Write
		\begin{align*}
			|&\partial_t^{\beta} P_s\ast (\chi_{3}\mu)(\ox)-\partial_t^{\beta} P_s\ast (\chi_{3}\mu)(\oy)|\\
			&=\bigg\rvert\int \frac{P_s\ast(\chi_3\mu)(x,\tau)-P_s\ast(\chi_3\mu)(x,t)}{|\tau-t|^{1+\beta}}\dd \tau-\int\frac{P_s\ast(\chi_3\mu)(x,\tau)-P_s\ast(\chi_3\mu)(x,u)}{|\tau-u|^{1+\beta}}\dd \tau \bigg\rvert\\
			&\leq \int_{|\tau-t|\leq (2r)^{2s}} \frac{|P_s\ast(\chi_3\mu)(x,\tau)-P_s\ast(\chi_3\mu)(x,t)|}{|\tau-t|^{1+\beta}}\dd \tau\\
			&\hspace{0.25cm}+\int_{|\tau-t|\leq (2r)^{2s}} \frac{|P_s\ast(\chi_3\mu)(x,\tau)-P_s\ast(\chi_3\mu)(x,u)|}{|\tau-u|^{1+\beta}}\dd \tau\\
			&\hspace{0.25cm}+ \int_{|\tau-t|> (2r)^{2s}}\bigg\rvert \frac{P_s\ast(\chi_3\mu)(x,\tau)-P_s\ast(\chi_3\mu)(x,t)}{|\tau-t|^{1+\beta}}\\
			&\hspace{4.5cm}- \frac{P_s\ast(\chi_3\mu)(x,\tau)-P_s\ast(\chi_3\mu)(x,u)}{|\tau-u|^{1+\beta}}\bigg\rvert\dd \tau=:I_1'+I_2'+I_{3}'.
		\end{align*}
		The expressions corresponding to $I_1',I_2'$ can be tackled in the same way as $I_1,I_2$. Hence, $I_1'\lesssim_{\beta} 1$ and $I_2'\lesssim_{\beta} 1$. Finally, for $I_{3}'$, adding and subtracting $P_s\ast(\chi_3\mu)(x,t)/|\tau-u|^{1+\beta}$,
		\begin{align*}
			I_{3}'&\leq \int_{|\tau-t|>(2r)^{2s}}\bigg\rvert \frac{1}{|\tau-t|^{1+\beta}}-\frac{1}{|\tau-u|^{1+\beta}} \bigg\rvert| P_s\ast(\chi_3\mu)(x,\tau)-P_s\ast(\chi_3\mu)(x,t) |\dd \tau\\
			&\hspace{0.5cm}+ \int_{|\tau-t|>(2r)^{2s}} \frac{1}{|\tau-u|^{1+\beta}} |P_s\ast(\chi_3\mu)(x,t)-P_s\ast(\chi_3\mu)(x,u) |\dd \tau.
		\end{align*}
		Since $|\tau-t|>(2r)^{2s}$ we can apply the mean value theorem to deduce
		\begin{align*}
			\bigg\rvert \frac{1}{|\tau-t|^{1+\beta}}-\frac{1}{|\tau-u|^{1+\beta}} \bigg\rvert\lesssim_{\beta} \frac{|t-u|}{|\tau-t|^{2+\beta}}\lesssim \frac{r^{2s}}{|\tau-t|^{2+\beta}}.
		\end{align*}
		In addition, since $\mu$ has upper $s$-parabolic growth of degree $n+2s\beta$, by Lemma \ref{Pos_lem1}, with $\eta:=\beta$, the time function $P_s\ast (\chi_3\mu)(x,\cdot)$ is Lip-$\beta$. Therefore,
		\begin{align*}
			I_{3}'\lesssim_{\beta} \int_{|\tau-t|>(2r)^{2s}}\frac{r^{2s}}{|\tau-t|^{2+\beta}}|\tau-t|^{\beta}\dd \tau+\int_{|\tau-t|>(2r)^{2s}}\frac{1}{|\tau-u|^{1+\beta}}|t-u|^{\beta}\dd \tau\lesssim_{\beta} 1.
		\end{align*}
		Therefore estimate \eqref{eq2.3.2} is satisfied and we are done with $I_{3}$ and also with the proof.
	\end{proof}
	
	In the same spirit, if we ask the positive measure for an extra $\alpha$ growth, the potential $\partial_t^{\beta}P_s\ast \mu$ will satisfy a $\text{Lip}_{\alpha,p_s}$ property. Recall that $2\zeta:=\min\{1,2s\}$.
	
	\begin{lem}
		\label{lem2.3.3}
		Let $s\in(0,1]$, $\beta \in (0,1)$ and $\alpha\in(0,2\zeta)$ such that $2s\beta+\alpha<2$. Let $\mu$ be a positive measure in $\mathbb{R}^{n+1}$ which has upper $s$-parabolic growth of degree $n+2s\beta+\alpha$. Then, 
		\begin{equation*}
			\|\partial_t^{\beta}P_s\ast \mu\|_{\text{\normalfont{Lip}}_\alpha,p_s}\lesssim_{\beta,\alpha} 1.
		\end{equation*}
	\end{lem}
	\begin{proof}
		Fix any $\ox,\oy\in\mathbb{R}^{n+1}, \ox\neq \oy$. We have to check if the following holds 
		\begin{equation*}
			|\partial_{t}^{\beta} P_s\ast \mu(\ox)-\partial_{t}^{\beta} P_s\ast \mu(\oy)|\lesssim_{\beta,\alpha} |\ox-\oy|_{p_s}^{\alpha}.
		\end{equation*}
		Begin by choosing the following partition of $\mathbb{R}^{n+1}$
		\begin{align*}
			R_1:&= \big\{\oz \;:\;|\ox-\oy|_{p_s}\leq |x-z|/5\big\}\cup\big\{\oz \;:\;|\oy-\ox|_{p_s}\leq |y-z|/5\big\},\\
			R_2:= \mathbb{R}^{n+1}\setminus{R_1}&=\big\{\oz \;:\;|\ox-\oy|_{p_s}> |x-z|/5\big\}\cap\big\{\oz \;:\;|\oy-\ox|_{p_s}>|y-z|/5\big\},
		\end{align*}
		and their corresponding characteristic functions $\chi_1,\chi_2$. From the latter we have
		\begin{align}
			&\frac{|\partial_{t}^{\beta} P_s\ast \mu(\ox)-\partial_{t}^{\beta} P_s\ast \mu(\oy)|}{|\ox-\oy|_{p_s}^{\alpha}}  \nonumber \\
			&\hspace{1cm}\leq\frac{1}{|\ox-\oy|_{p_s}^\alpha}\int_{ |\ox-\oy|_{p_s}\leq |x-z|/5}|\partial_{t}^{\beta} P_s(\ox-\oz)-\partial_{t}^{\beta} P_s(\oy-\oz)|\text{d}\mu(\oz) \nonumber \\
			& \hspace{2cm}+\frac{1}{|\ox-\oy|_{p_s}^\alpha}\int_{ |\oy-\ox|_{p_s}\leq |y-z|/5}|\partial_{t}^{\beta} P_s(\ox-\oz)-\partial_{t}^{\beta} P_s(\oy-\oz)|\text{d}\mu(\oz) \nonumber \\
			&\hspace{2cm} +\frac{1}{|\ox-\oy|_{p_s}^\alpha}\big\rvert\partial_{t}^{\beta} P_s\ast (\chi_2\mu)(\ox)-\partial_{t}^{\beta} P_s\ast (\chi_2\mu)(\oy)\big\rvert =:I_{1,\ox}+I_{1,\oy}+I_2. \label{eq2.4.2}
		\end{align}
		
		Regarding $I_{1,\ox}$, name $\overline{\xi}:=\ox-\oz$, $\overline{\xi}':=\oy-\oz$ and observe that, in particular, one has
		\begin{equation*}
			|\overline{\xi}-\overline{\xi}'|_{p_s} = |\ox-\oy|_{p_s} < \frac{|x-z|}{2}=\frac{|\xi|}{2},
		\end{equation*}
		Applying the last estimate either of Theorem \ref{lem2.4} or Theorem \ref{lem3.5}, we deduce
		\begin{equation*}
			I_{1,\ox}\lesssim_{\beta} \frac{1}{|\ox-\oy|_{p_s}^{\alpha-2\zeta}}\int_{ |\ox-\oy|_{p_s}\leq |x-z|/5} \frac{\text{d}\mu(\oz)}{|x-z|^{n+2\zeta-2s}|\ox-\oz|_{p_s}^{2s(1+\beta)}} 
		\end{equation*}
		Let us split the domain of integration into proper disjoint pieces. For $\ox=(x,t)$, we denote
		\begin{equation*}
			B_{\ox}:=B(\ox, |\ox-\oy|_{p_s})=B_1(x,|\ox-\oy|_{p_s})\times J_{\ox},
		\end{equation*}
		where $B_1 (x,|\ox-\oy|_{p_s})$ is an Euclidean ball in $\mathbb{R}^n$ and $J_{\ox}$ is a real interval centered at $t$ with length $2|\ox-\oy|_{p_s}^{2s}$. As in Lemma \ref{Pos_lem2},  take cylinders $C_{j,\ox}:= 5^j B_1(\ox, |\ox-\oy|_{p_s})\times \mathbb{R}$ for $j\geq 1$, as well as the annular cylinders $\widehat{C}_{j,\ox}:=C_{j+1,\ox}\setminus{C_{j,\ox}}$, for $j\geq 1$. We express $\{\oz\,:\, |\ox-\oy|_{p_s}\leq |x-z|/5 \}$ as the disjoint union of the sets $\widehat{C}_{j,\ox}$, so that
		\begin{equation*}
			I_{1,\ox} \lesssim_{\beta} \frac{1}{|\ox-\oy|_{p_s}^{\alpha-2\zeta}}\sum_{j=1}^\infty \int_{\widehat{C}_{j,\ox}} \frac{\text{d}\mu(\oz)}{|x-z|^{n+2\zeta-2s}|\ox-\oz|_{p_s}^{2s(1+\beta)}}.
		\end{equation*}
		The above integral can be studied as that appearing in \eqref{Pos_eq1}, in the study of the term $I_2$ of Lemma \ref{Pos_lem2} (centering now the cylinders in $\ox$ and interchanging the roles of $r$ and $|\ox-\oy|_{p_s}$). Doing so, and taking into account the $n+2s\beta+\alpha$ growth of $\mu$, one obtains
		\begin{align*}
			I_{1,\ox} &\lesssim_{\beta,\alpha} \frac{1}{|\ox-\oy|_{p_s}^{\alpha-2\zeta}} \sum_{j=1}^\infty\sum_{k\geq j} \frac{(5^{j+1}|\ox-\oy|_{p_s})^{n+2s\beta+\alpha}}{(5^j|\ox-\oy|_{p_s})^{n+2\zeta-2s}(5^k|\ox-\oy|_{p_s})^{2s(1+\beta)}}\frac{5^{2sk}}{5^{2sj}}\\
			&=\sum_{j=1}^\infty\sum_{k\geq j} \frac{5^{j(n+2s\beta+\alpha)}}{5^{j(n+2\zeta-2s)}5^{2s(1+\beta)k}}\frac{5^{2sk}}{5^{2sj}}\simeq \sum_{j=1}^\infty\sum_{k\geq j}\frac{5^{j(2s\beta + \alpha-2\zeta)}}{5^{2s\beta k}}=\sum_{k=1}^\infty \frac{1}{5^{2s\beta k}} \sum_{j=1}^k\frac{1}{5^{j(2\zeta-2s\beta-\alpha)}}\\
			&\lesssim \sum_{k=1}^\infty \frac{1}{5^{2s\beta k}} \bigg( 1+\frac{1}{5^{(2\zeta-2s\beta-\alpha)k}} \bigg) \lesssim_{\beta,\alpha} 1, \quad \text{if $\alpha < 2\zeta$}.
		\end{align*}
		The study of $I_{1,\oy}$ is analogous, interchanging the roles of $\ox$ and $\oy$. Finally we deal with $I_2$. We claim that the following estimate holds
		\begin{equation*}
			\big\rvert\partial_{t}^{\beta} P_s\ast (\chi_{2}\mu)(\ox)- \partial_{t}^{\beta}P_s\ast (\chi_{2}\mu)(\oy)\big\rvert\lesssim_{\beta,\alpha} |\ox-\oy|_{p_s}^\alpha.
		\end{equation*}
		The general case will follows from the following two cases: whether $\ox$ and $\oy$ share their time coordinate, or if they share their spatial coordinate.  Indeed, write $\ox=(x,t), \oy=(y,\tau)$ and set $\widehat{x}:=(x,\tau)$ so that
		\begin{align*}
			\big\rvert\partial_{t}^{\beta} P_s&\ast (\chi_{2}\mu)(\ox)- \partial_{t}^{\beta}P_s\ast (\chi_{2}\mu)(\oy)\big\rvert \\
			&\hspace{-0.9cm}\leq \big\rvert\partial_{t}^{\beta} P_s\ast (\chi_{2}\mu)(\ox)- \partial_{t}^{\beta}P_s\ast (\chi_{2}\mu)(\widehat{x})\big\rvert + \big\rvert\partial_{t}^{\beta} P_s\ast (\chi_{2}\mu)(\widehat{x})- \partial_{t}^{\beta}P_s\ast (\chi_{2}\mu)(\oy)\big\rvert\\
			&\hspace{-0.9cm}\lesssim_{\beta,\alpha} |\ox-\widehat{x}|_{p_s}^\alpha+|\widehat{x}-\oy|_{p_s}^\alpha=|t-\tau|^{\alpha/2}+|x-y|^\alpha\leq 2 |\ox-\oy|_{p_s}^\alpha, \;\; \text{and we are done.}
		\end{align*}
		\textit{Case 1: $\ox=(x,t)$ and $\oy=(x,u)$}. Write $\mu_{2}:=\chi_{2}\mu$ and estimate $|\partial_t^{\beta} P_s\ast \mu_{2}(\ox)-\partial_t^{\beta} P_s\ast \mu_{2}(\oy)|$ as follows
		\begin{align*}
			&\bigg\rvert\int \frac{P_s\ast\mu_{2}(x,\tau)-P_s\ast\mu_{2}(x,t)}{|\tau-t|^{1+\beta}}\dd \tau-\int\frac{P_s\ast\mu_{2}(x,\tau)-P_s\ast\mu_{2}(x,u)}{|\tau-u|^{1+\beta}}\dd \tau \bigg\rvert\\
			&\leq \int_{|\tau-t|\leq 2^{2s}|\ox-\oy|_{p_s}^{2s}} \frac{|P_s\ast\mu_{2}(x,\tau)-P_s\ast\mu_{2}(x,t)|}{|\tau-t|^{1+\beta}}\dd \tau\\
			&\hspace{0.25cm}+\int_{|\tau-t|\leq 2^{2s}|\ox-\oy|_{p_s}^{2s}} \frac{|P_s\ast\mu_{2}(x,\tau)-P_s\ast\mu_{2}(x,u)|}{|\tau-u|^{1+\beta}}\dd \tau\\
			&\hspace{0.25cm}+ \int_{|\tau-t|> 2^{2s}|\ox-\oy|_{p_s}^{2s}}\bigg\rvert \frac{P_s\ast\mu_{2}(x,\tau)-P_s\ast\mu_{2}(x,t)}{|\tau-t|^{1+\beta}}\\
			&\hspace{5cm}- \frac{P_s\ast\mu_{2}(x,\tau)-P_s\ast\mu_{2}(x,u)}{|\tau-u|^{1+\beta}}\bigg\rvert\dd \tau=:I_1+I_2+I_{3}.
		\end{align*}
		By a direct application of Lemma \ref{Pos_lem1} we are able to obtain, straightforwardly,
		\begin{align*}
			I_1&\lesssim_{\beta,\alpha} \int_{|\tau-t|\leq 2^{2s}|\ox-\oy|_{p_s}^{2s}}\frac{\dd \tau}{|\tau-t|^{1-\frac{\alpha}{2s}}}\lesssim_{\alpha} |\ox-\oy|_{p_s}^\alpha \quad \text{and}\\
			I_2&\lesssim_{\beta,\alpha} \int_{|\tau-t|\leq 2^{2s}|\ox-\oy|_{p_s}^{2s}}\frac{\dd \tau}{|\tau-u|^{1-\frac{\alpha}{2s}}} \lesssim_{\alpha} |\ox-\oy|_{p_s}^\alpha.
		\end{align*}
		For $I_{3}$, adding and subtracting the term $P_s\ast\mu_{2}(x,t)/|\tau-u|^{1+\beta}$ we get
		\begin{align*}
			I_{3}&\leq \int_{|\tau-t|> 2^{2s}|\ox-\oy|_{p_s}^{2s}}\bigg\rvert \frac{1}{|\tau-t|^{1+\beta}}-\frac{1}{|\tau-u|^{1+\beta}} \bigg\rvert| P_s\ast\mu_{2}(x,\tau)-P_s\ast\mu_{2}(x,t) |\dd \tau\\
			&\hspace{0.5cm}+ \int_{|\tau-t|> 2^{2s}|\ox-\oy|_{p_s}^{2s}} \frac{1}{|\tau-u|^{1+\beta}} |P_s\ast\mu_{2}(x,t)-P_s\ast\mu_{2}(x,u) |\dd \tau.
		\end{align*}
		Since $|\tau-t|> 2^{2s}|\ox-\oy|_{p_s}^{2s}$ we can apply the mean value theorem to deduce
		\begin{align*}
			\bigg\rvert \frac{1}{|\tau-t|^{1+\beta}}-\frac{1}{|\tau-u|^{1+\beta}} \bigg\rvert\lesssim_{\beta} \frac{|t-u|}{|\tau-t|^{2+\beta}}\lesssim \frac{|\ox-\oy|_{p_s}^{2s}}{|\tau-t|^{2+\beta}}.
		\end{align*}
		Therefore, by Lemma \ref{Pos_lem1} with $\eta:=\beta+\frac{\alpha}{2s}$, we finally have
		\begin{align*}
			I_{3}\lesssim_{\beta,\alpha} \int_{|\tau-t|> 2^{2s}|\ox-\oy|_{p_s}^{2s}}&\frac{|\ox-\oy|_{p_s}^{2s}}{|\tau-t|^{2+\beta}} |\tau-t|^{\beta+\frac{\alpha}{2s}}\dd \tau\\
			&+\int_{|\tau-t|> 2^{2s}|\ox-\oy|_{p_s}^{2s}}\frac{1}{|\tau-u|^{1+\beta}}|t-u|^{\beta+\frac{\alpha}{2s}}\dd \tau\lesssim_{\beta,\alpha} |\ox-\oy|_{p_s}^\alpha.
		\end{align*}
		Therefore $|\partial_t^{\beta} P_s\ast \mu_{2}(\ox)-\partial_t^{\beta} P_s\ast \mu_{2}(\oy)|\leq I_1+I_2+I_{3} \lesssim_{\beta,\alpha} |\ox-\oy|_{p_s}^\alpha$, and this ends the study of \textit{Case 1}.
        
		\textit{Case 2: $\ox=(x,t)$ and $\oy=(y,t)$}. To tackle this case, let us first rewrite the set $R_2$ as
		\begin{align*}
			R_2&=\Big[5B_1\big( x, |\ox-\oy|_{p_s} \big) \times \mathbb{R}\Big]\cap \Big[5B_1\big( y, |\oy-\ox|_{p_s} \big) \times \mathbb{R}\Big] =\big( 5B_{1,x}\times \mathbb{R} \big)\cap \big( 5B_{1,y}\times \mathbb{R} \big),
		\end{align*}
		Continue rewriting $R_2$ as follows
		\begin{align*}
			R_2&=\Big\{ 5B_{\ox}\cup \big[(5B_{1,x}\times \mathbb{R})\setminus{5B_{\ox}}\big] \Big\}\cap\Big\{ 5B_{\oy}\cup \big[(5B_{1,y}\times \mathbb{R})\setminus{5B_{\oy}}\big] \Big\}\\
			&=\big( 5B_{\ox}\cap 5B_{\oy} \big)\cup \Big\{ 5B_{\ox}\cap \big[(5B_{1,y}\times \mathbb{R})\setminus{5B_{\oy}}\big] \Big\} \\
			&\hspace{3.09cm}\cup \Big\{ 5B_{\oy}\cap \big[(5B_{1,x}\times \mathbb{R})\setminus{5B_{\ox}}\big] \Big\}\\
			&\hspace{3.09cm}\cup \Big\{ \big[(5B_{1,x}\times \mathbb{R})\setminus{5B_{\ox}}\big]\cap \big[(5B_{1,y}\times \mathbb{R})\setminus{5B_{\oy}}\big] \Big\}\\
			&=:R_{21}\cup R_{22}\cup R_{23}\cup R_{24}.
		\end{align*}
		Observe that in \textit{Case 2} the real intervals $J_{\ox}$ and $J_{\oy}$ coincide. We name them $J$. Therefore,
		\begin{align*}
			R_{22}&:= 5B_{\ox}\cap \big[ (5B_{1,y}\times \mathbb{R})\setminus{5B_{\oy}} \big]=(5B_{1,x}\times J)\cap \big[5B_{1,y}\times (\mathbb{R}\setminus{J})\big]=\varnothing,\\
			R_{23}&:= 5B_{\oy}\cap \big[ (5B_{1,y}\times \mathbb{R})\setminus{5B_{\oy}} \big]=(5B_{1,y}\times J)\cap \big[5B_{1,x}\times (\mathbb{R}\setminus{J})\big]=\varnothing,
		\end{align*}
		meaning that, in fact, $R_2=R_{21}\cup R_{24}$. Observe also that $R_{24}$ can be rewritten as
		\begin{align*}
			R_{24}:&=\big[(5B_{1,x}\times \mathbb{R})\setminus{5B_{\ox}}\big]\cap \big[(5B_{1,y}\times \mathbb{R})\setminus{5B_{\oy}}\big]\\
			&=(5B_{1,x}\cap 5B_{1,y})\times(\mathbb{R}\setminus{J}).
		\end{align*}
		Therefore, if $\chi_{21}$ and $\chi_{24}$ are the characteristic functions of $R_{21}$ and $R_{24}$, we have, naming $\mu_{21}:=\chi_{21}\mu$ and $\mu_{24}:=\chi_{24}\mu$,
		\begin{align*}
			I_2&\leq \frac{1}{|\ox-\oy|_{p_s}^\alpha}\big\rvert\partial_{t}^{\beta} P_s\ast \mu_{21}(\ox)-\partial_{t}^{\beta} P_s\ast \mu_{21}(\oy)\big\rvert\nonumber\\
			&\hspace{2cm}+\frac{1}{|\ox-   \oy|_{p_s}^\alpha}\big\rvert\partial_{t}^{\beta} P_s\ast \mu_{24}(\ox)-\partial_{t}^{\beta} P_s\ast \mu_{24}(\oy)\big\rvert=:I_{21}+I_{24}.
		\end{align*}
		Hence, fixing $j\in\{1,4\}$, begin by establishing the following estimate
		\begin{align*}
			|\partial&_t^{\beta} P_s\ast \mu_{2j}(\ox)-\partial_t^{\beta} P_s\ast \mu_{2j}(\oy)|\\
			&=\bigg\rvert\int \frac{P_s\ast\mu_{2j}(x,\tau)-P_s\ast\mu_{2j}(x,t)}{|\tau-t|^{1+\beta}}\dd \tau-\int\frac{P_s\ast\mu_{2j}(y,\tau)-P_s\ast\mu_{2j}(y,t)}{|\tau-t|^{1+\beta}}\dd \tau \bigg\rvert\\
			&\leq \int_{|\tau-t|\leq 2^{2s}|\ox-\oy|_{p_s}^{2s}} \frac{|P_s\ast\mu_{2j}(x,\tau)-P_s\ast\mu_{2j}(x,t)|}{|\tau-t|^{1+\beta}}\dd \tau\\
			&\hspace{0.2cm}+\int_{|\tau-t|\leq 2^{2s}|\ox-\oy|_{p_s}^{2s}} \frac{|P_s\ast\mu_{2j}(y,\tau)-P_s\ast\mu_{2j}(y,t)|}{|\tau-t|^{1+\beta}}\dd \tau\\
			&\hspace{0.2cm}+\int_{|\tau-t|> 2^{2s}|\ox-\oy|_{p_s}^{2s}} \frac{|P_s\ast\mu_{2j}(x,\tau)-P_s\ast\mu_{2j}(x,t)-P_s\ast\mu_{2j}(y,\tau)+P_s\ast\mu_{2j}(y,t)|}{|\tau-t|^{1+\beta}}\dd \tau\\
			&\hspace{0.2cm}=:C_1+C_2+C_3.
		\end{align*}
		Lemma \ref{Pos_lem1} with $\eta=\beta$ yields $C_1\lesssim_{\beta,\alpha} |\ox-\oy|_{p_s}^\alpha$ and $C_2\lesssim_{\beta,\alpha} |\ox-\oy|_{p_s}^\alpha$, so we focus on $C_3$. Split it as follows
		\begin{align*}
			C_3&\leq \int_{|\tau-t|> 2^{2s}|\ox-\oy|_{p_s}^{2s}} \frac{|P_s\ast\mu_{2j}(x,\tau)-P_s\ast\mu_{2j}(y,\tau)|}{|\tau-t|^{1+\beta}}\dd \tau\\
			&\hspace{2cm}+\int_{|\tau-t|> 2^{2s}|\ox-\oy|_{p_s}^{2s}} \frac{|P_s\ast\mu_{2j}(x,t)-P_s\ast\mu_{2j}(y,t)|}{|\tau-t|^{1+\beta}}\dd \tau=:C_{31}+C_{32}.
		\end{align*}
		First, let us deal with integral $C_{32}$. On the one hand, if $j=1$, observe that for any $\oz\in 2B_{\ox}$, since $2B_{\ox}\subset R_{21}\subset 5B_{\ox}$, we can contain $R_{21}$ into $s$-parabolic annuli centered at $\oz$ and (exponentially decreasing) radii proportional to $|\ox-\oy|_{p_s}$. Hence, by \cite[Lemma 2.2]{MaPr} and the upper $s$-parabolic growth of degree $n+2s\beta+\alpha$ of $\mu$, we deduce 
		\begin{align*}
			|P_s\ast \mu_{21}(\oz)|&\lesssim  \int_{5B_{\ox}\cap 5B_{\oy}}\frac{\text{d}\mu(\overline{w})}{|\oz-\overline{w}|_{p_s}^{n}}\lesssim_{\beta,\alpha} |\ox-\oy|_{p_s}^{2s\beta+\alpha}.
		\end{align*}
		If $j=4$, observe that $|P_s\ast\mu_{24}(x,t)-P_s\ast\mu_{2j}(y,t)|\leq |x-y|\,\|\nabla_xP_s\ast \mu_{24}\|_{\infty,2B_{\ox}}$. 
		So for any $\oz\in 2B_{\ox}$, by Theorem \ref{C-Z_thm2} we obtain
		\begin{align*}
			|\nabla_xP_s&\ast \mu_{24}(\oz)|\lesssim  \int_{(5B_{\ox}\cap 5B_{\oy})\times(\mathbb{R}\setminus{J})}\frac{|z-w|}{|\oz-\overline{w}|_{p_s}^{n+2}}\text{d}\mu(\overline{w}) \\
			&\lesssim |\ox-\oy|_{p_s} \int_{\mathbb{R}^{n+1}\setminus{(5B_{p,\ox}\cap 5B_{p,\oy})}}\frac{\text{d}\mu(\overline{w})}{|\oz-\overline{w}|_{p_s}^{n+2}}\lesssim_{\beta,\alpha} |\ox-\oy|_{p_s}^{2s\beta+\alpha-1}, \quad \text{since }\, 2s\beta+\alpha<2.
		\end{align*}
		For the last inequality we can split, for example, the domain of integration into $s$-parabolic annuli centered at $\oz$ with (exponentially increasing) radii proportional to $2|\ox-\oy|_{p_s}$. Then,
		\begin{equation}
			\label{eq2.4.3}
			C_{32}\lesssim_{\beta,\alpha} |\ox-\oy|_{p_s}^{2s\beta+\alpha}\int_{|\tau-t|> 2^{2s}|\ox-\oy|_{p_s}^{2s}}\frac{\dd \tau}{|\tau-t|^{1+\beta}}\lesssim_{\beta} |\ox-\oy|_{p_s}^{\alpha}.
		\end{equation}
		Regarding $C_{31}$, the points $(x,\tau)$ and $(y,\tau)$ belong to a temporal translate of $2B_{\ox}\cap 2B_{\oy}$ that does not intersect $2B_{\ox}\cap 2B_{\oy}$, since $|\tau-t|> 2^{2s}|\ox-\oy|_{p_s}^{2s}$. We call it $2B_{\ox}^{\tau}\cap 2B_{\oy}^{\tau}$. For each $j\in\{1,4\}$ and $\tau$ (and bearing in mind Theorem \ref{C-Z_thm2}) we deduce
		\begin{align}
			|P_s\ast&\mu_{2j}(x,\tau)-P_s\ast\mu_{2j}(y,\tau)|\nonumber\\
			&\leq \int_{2B_{\ox}^{\tau}\cap 2B_{\oy}^{\tau}}|P_s((x,\tau)-\overline{w})-P_s((y,\tau)-\overline{w})|\text{d}\mu(\overline{w}) \nonumber\\
			&\hspace{2.5cm}+ \int_{R_{2j}\setminus{(2B_{\ox}^{\tau}\cap 2B_{\oy}^{\tau})}}|P_s((x,\tau)-\overline{w})-P_s((y,\tau)-\overline{w})|\text{d}\mu(\overline{w}) \nonumber\\
			&\lesssim \int_{2B_{\ox}^{\tau}}\frac{\text{d}\mu(\overline{w})}{|(x,\tau)-\overline{w}|_{p_s}^n}+\int_{2B_{\oy}^{\tau}}\frac{\text{d}\mu(\overline{w})}{|(y,\tau)-\overline{w}|_{p_s}^n} \label{eq2.4.4} \\
			&\hspace{2.5cm}+ |x-y|\int_{R_{2j}\setminus{(2B_{\ox}^{\tau}\cap 2B_{\oy}^{\tau})}}|\nabla_xP_s((\widetilde{x},\tau)-\overline{w})|\text{d}\mu(\overline{w})\nonumber\\
			&\lesssim_{\beta,\alpha} |\ox-\oy|_{p_s}^{2s\beta+\alpha} + |\ox-\oy|_{p_s}^2\int_{\mathbb{R}^{n+1}\setminus{(2B_{\ox}^{\tau}\cap 2B_{\oy}^{\tau})}}\frac{\text{d}\mu(\overline{w})}{|(\widetilde{x},\tau)-\overline{w}|_{p_s}^{n+2}}, \label{eq2.4.5}
		\end{align}
		where for both integrals of \eqref{eq2.4.4} we have split the domain of integration into (exponentially decreasing) $s$-parabolic annuli; while in the remaining term $\widetilde{x}$ belongs to the segment joining $x$ and $y$. Observe also that in the last inequality we have used that the spatial distance between any two points of $R_{21}\setminus{(2B_{\ox}^{\tau}\cap 2B_{\oy}^{\tau})}$ and $R_{24}\setminus{(2B_{\ox}^{\tau}\cap 2B_{\oy}^{\tau})}$ is bounded by a multiple of $|x-y|$ and thus of $|\ox-\oy|_{p_s}$. Observe now that, if $\xi:=(x+y)/2$, we have
		\begin{align*}
			2B_{\ox}^{\tau}\cap 2B_{\oy}^{\tau} &= B\big( (x,t+\tau), 2|\ox-\oy|_{p_s} \big)\cap B\big( (y,t+\tau), 2|\ox-\oy|_{p_s} \big)\\
			&\supset B\big( (\xi,t+\tau), |\ox-\oy|_{p_s} \big)=:\widehat{B}^{\tau},
		\end{align*}
		meaning that
		\begin{equation*}
			\mathbb{R}^{n+1}\setminus{(2B_{\ox}^{\tau}\cap 2B_{\oy}^{\tau})} \subset \mathbb{R}^{n+1}\setminus{\widehat{B}^\tau}.
		\end{equation*}
		Return to \eqref{eq2.4.5} and estimate the remaining integral by another one with the same integrand, but over the enlarged domain $\mathbb{R}^{n+1}\setminus{\widehat{B}^\tau}$. Afterwards, split the latter into $s$-parabolic annuli centered at $(\widetilde{x}, \tau)$ and (exponentially increasing) radii proportional to $|\ox-\oy|_{p_s}/2$ and use that $2s\beta+\alpha<2$ so that 
		\begin{align*}
			|P_s\ast\mu_{2j}(x,\tau)&-P_s\ast\mu_{2j}(y,\tau)| \lesssim_{\beta,\alpha}  |\ox-\oy|_{p_s}^{2s\beta+\alpha}+\frac{|\ox-\oy|_{p_s}^{2}}{|\ox-\oy|_{p_s}^{2-2s\beta-\alpha}}\simeq |\ox-\oy|_{p_s}^{2s\beta+\alpha}.
		\end{align*}
		Hence, similarly to \eqref{eq2.4.3} we deduce $C_{31}\lesssim_{\beta,\alpha} |\ox-\oy|_{p_s}^{\alpha}$, which means $I_2\leq I_{21}+I_{24}\lesssim_{\beta,\alpha} 1$ and we are done with \textit{Case 2}. This last estimate finally implies 
		\begin{equation*}
			\big\rvert\partial_{t}^{\beta} P_s\ast (\chi_{2}\mu)(\ox)- \partial_t^{\beta}P_s\ast (\chi_{2}\mu)(\oy)\big\rvert\lesssim_{\beta,\alpha} |\ox-\oy|_{p_s}^\alpha,
		\end{equation*}
		which means $I_2\lesssim_{\beta,\alpha} 1$. So applying it to \eqref{eq2.4.2} we conclude that
		\begin{equation*}
			\frac{|\partial_{t}^{\beta} P_s\ast \mu(\ox)-\partial_{t}^{\beta} P_s\ast \mu(\oy)|}{|\ox-\oy|_{p_s}^{\alpha}}\leq I_{1,\ox}+I_{1,\oy}+I_2\lesssim_{\beta,\alpha} 1,
		\end{equation*}
		and the desired $s$-parabolic $\text{Lip}_\alpha$ condition follows.
	\end{proof}
	
	\section{The \mathinhead{s}{}-parabolic BMO and \mathinhead{\text{Lip}_{\alpha}}{} caloric capacities}
	\label{sec2.4}
	
	We are finally ready to introduce the $s$-parabolic BMO and $\text{Lip}_\alpha$ variants of the caloric capacities presented in \cite{MaPrTo, MaPr}. This section generalizes the concept to include a broader set of variants. The principal result will be that, in any case, such capacities will turn out to be comparable to a certain $s$-parabolic Hausdorff content. Moreover, we will be able to characterize removable sets for $\text{BMO}_{p_s}$ and $\text{Lip}_{\alpha,p_s}$ solutions to the $\Theta^s$-equation in terms of the nullity of the respective capacities. In order to do so, we will need a fundamental lemma that we present before introducing the different capacities. The result below will characterize distributions supported on a compact set with finite $d$-dimensional Hausdorff measure that satisfy some growth property only for small enough $s$-parabolic cubes.
	
	\begin{lem}
		\label{Cap_thm2}
		Let $d>0$ and $E\subset \mathbb{R}^{n+1}$ be a compact set with $\pazocal{H}_{p_s}^{d}(E)<\infty$. Let $T$ be a distribution supported on $E$ with the property that there exists $0<\ell_0\leq \infty$ such that for any $R\subset \mathbb{R}^{n+1}$ $s$-parabolic cube with $\ell(R)\leq \ell_0$,
		\begin{equation*}
			|\langle T, \phi \rangle | \lesssim \ell(R)^{d}, \qquad \forall \phi \;\; \text{admissible for}\;\; R.
		\end{equation*}
		Then, $T$ is a signed measure satisfying
		\begin{equation*}
			|\langle T, \psi \rangle|\lesssim \pazocal{H}_{p_s}^{d}(E)\|\psi\|_\infty,\qquad \forall \psi\in \pazocal{C}_c^{\infty}(\mathbb{R}^{n+1}).
		\end{equation*}
	\end{lem}
	\begin{proof}
		We follow the proof of \cite[Lemma 6.2]{MaPrTo}. Let $\psi\in \pazocal{C}_c^{\infty}(\mathbb{R}^{n+1})$ and $0<\varepsilon\leq \ell_0/4$. Let $Q_i, i\in I_\varepsilon$ be a collection of $s$-parabolic cubes with $F\subset \bigcup_{i\in I_\varepsilon} Q_i$ with $\ell(Q_i)\leq \varepsilon$ and
		\begin{equation*}
			\sum_{i\in I_\varepsilon}\ell(Q_i)^{d}\leq C\pazocal{H}_{p_s}^{d}(E)+\varepsilon.
		\end{equation*}
		Now cover each $Q_i$ by a bounded number (depending on the dimension) of dyadic $s$-parabolic cubes $R_i^1,\ldots, R_i^m$ with $\ell(R_i^j)\leq \ell(Q_i)/8$ and apply an $s$-parabolic version of Harvey-Polking's lemma (that admits an analogous proof, see \cite[Lemma 3.1]{HP}) to obtain a collection of non-negative functions $\{\varphi_i\}_{i\in I_\varepsilon}$ with $\text{supp}(\varphi_i)\subset 2Q_i$, $c\varphi_i$ admissible for $2Q_i$ and satisfying $\sum_{i\in I_\varepsilon}\varphi_i\equiv 1$ on $\bigcup_{i\in I_\varepsilon}Q_i\supset E$. Now we write
		\begin{equation*}
			|\langle T, \psi \rangle| \leq \sum_{i\in I_\varepsilon} |\langle T, \varphi_i\psi \rangle|.
		\end{equation*}
		Proceeding as in \cite[Lemma 6.2]{MaPrTo} it can be shown that
		\begin{equation*}
			\eta_i:=\frac{\varphi_i\psi}{\|\psi\|_\infty + \ell(Q_i)\|\nabla_x\psi\|_\infty+\ell(Q_i)^{2s}\|\partial_t\psi\|_\infty+\ell(Q_i)^2\|\Delta \psi\|_\infty}
		\end{equation*}
		is an admissible function for $2Q_i$ (up to a dimensional constant), with $\ell(2Q_i)\leq \ell_0/2$. Therefore, by the growth assumptions on $T$,
		\begin{align*}
			|\langle T, \psi \rangle| &\lesssim  \sum_{i\in I_\varepsilon} \ell(Q_i)^{d}\big( \|\psi\|_\infty + \ell(Q_i)\|\nabla_x\psi\|+\ell(Q_i)^2\|\partial_t\psi\|_\infty+\ell(Q_i)^2\|\Delta \psi\|_\infty \big)\\
			&\lesssim (\pazocal{H}_{p_s}^{d}(E)+\varepsilon)\big( \|\psi\|_\infty + \varepsilon\|\nabla_x\psi\|+\varepsilon^2\|\partial_t\psi\|_\infty+\varepsilon^2\|\Delta \psi\|_\infty \big),
		\end{align*}
		and making $\varepsilon$ tend to 0, we deduce the result.
	\end{proof}
	
	\subsection{The capacity \mathinhead{\Gamma_{\Theta^s,\ast}}{}}
	\label{subsec2.4.1}
	The first capacity we introduce is the $\text{BMO}_{p_s}$ variant of the caloric capacity first defined in \cite{MaPrTo} for the usual heat equation. 
	\begin{defn}
		Given $s\in(1/2,1]$ and $E\subset \mathbb{R}^{n+1}$ compact set, define its \textit{$\text{\normalfont{BMO}}_{p_s}$-caloric capacity} as
		\begin{equation*}
			\Gamma_{\Theta^s,\ast}(E):=\sup  |\langle T, 1 \rangle| ,
		\end{equation*}
		where the supremum is taken among all distributions $T$ with $\text{supp}(T)\subset E$ and satisfying
		\begin{equation}
			\label{eq1.4.1}
			\|\nabla_x P_s\ast T\|_{\ast, p_s} \leq 1, \hspace{0.75cm} \|\partial^{\frac{1}{2s}}_tP_s\ast T\|_{\ast, p_s}\leq 1.
		\end{equation}
		Such distributions will be called \textit{admissible for $\Gamma_{\Theta^s,\ast}(E)$}.
	\end{defn}
	Let us also introduce what we will understand as removable sets in this context:
	
	\begin{defn}
		A compact set $E\subset \mathbb{R}^{n+1}$ is said to be \textit{removable for $s$-caloric functions with $\text{\normalfont{BMO}}_{p_s}$-$(1,\frac{1}{2s})$-derivatives} if for any open subset $\Omega\subset \mathbb{R}^{n+1}$, any function $f:\mathbb{R}^{n+1}\to \mathbb{R}$ with
		\begin{equation*}
			\|\nabla_x f\|_{\ast,p_s} < \infty, \hspace{0.75cm} \|\partial_t^{\frac{1}{2s}}f\|_{\ast, p_s}<\infty,
		\end{equation*}
		satisfying the $\Theta^s$-equation in $\Omega\setminus{E}$, also satisfies the previous equation in the whole $\Omega$.
	\end{defn}
	
	First, we shall prove that if $T$ satisfies \eqref{eq1.4.1}, then $T$ has upper $s$-parabolic growth of degree $n+1$. In fact, we shall prove a stronger result:
	\begin{thm}
		\label{Cap_thm1}
		Let $s\in(1/2,1]$ and $T$ be a distribution in $\mathbb{R}^{n+1}$ with
		\begin{equation*}
			\|\nabla_x P_s\ast T\|_{\ast, p_s} \leq 1, \hspace{0.75cm} \|\partial^{\frac{1}{2s}}_tP_s\ast T\|_{\ast, p_s}\leq 1.
		\end{equation*}
		Let $Q$ be a fixed $s$-parabolic cube and $\varphi$ an admissible function for $Q$. Then, if $R$ is any $s$-parabolic cube with $\ell(R)\leq \ell(Q)$ and $\phi$ is admissible for $R$, we have $|\langle \varphi T, \phi \rangle|\lesssim \ell(R)^{n+1}$.
	\end{thm}
	\begin{proof}
		Let $T$, $Q$ and $\varphi$ be as above. Let $R$ be an $s$-parabolic cube with $\ell(R)\leq \ell(Q)$ and $R\cap Q\neq \varnothing$ (if not, the result is trivial) and $\phi$ admissible function for $R$. Since $P_s$ is the fundamental solution to the $\Theta^s$-equation,
		\begin{equation*}
			|\langle \varphi T, \phi\rangle|=|\langle \Theta^s P_s \ast  T,  \varphi\phi \rangle|\leq |\langle (-\Delta)^sP_s\ast T, \varphi\phi\rangle|+|\langle P_s\ast T, \partial_t(\varphi\phi)\rangle|=:I_1+I_2.
		\end{equation*}
		Regarding $I_2$, observe that defining $\beta:=1-\frac{1}{2s}\in(0,1/2]$ we get
		\begin{equation*}
			\partial_t(\varphi\phi) = c\, \partial_t^{1-\beta}\Big( \partial_t(\varphi\phi)\ast_t |t|^{-\beta} \Big),
		\end{equation*}
		for some constant $c$. The latter can be checked via the Fourier transform with respect to the $t$ variable. Therefore, applying Theorem \ref{Growth_thm1} we get
		\begin{equation*}
			I_2\simeq c|\langle \partial_t^{1-\beta} P_s\ast T,  \partial_t(\varphi\phi)\ast_t |t|^{-\beta} \rangle| \lesssim \ell(R)^{n+2s(1-\beta)}=\ell(R)^{n+1}.
		\end{equation*}
		To study $I_1$ we distinguish whether if $s=1$ or $s<1$. If $s=1$, Theorem \ref{Growth_thm2} yields
		\begin{align*}
			I_1=|\langle \Delta W\ast T, \varphi\phi\rangle|=|\langle \nabla_x W\ast T, \nabla_x(\varphi\phi)\rangle| \lesssim \ell(R)^{n+1}.
		\end{align*}
		Recall that the operator $(-\Delta)^s$ can be rewritten as
		\begin{equation*}
			(-\Delta)^s(\cdot)\simeq\sum_{i=1}^n \partial_{x_i}\bigg( \frac{1}{|x|^{n+2s-2}} \bigg) \ast_n \partial_{x_i}(\cdot),
		\end{equation*}
		where $\ast_n$ indicates that the convolution is taken with respect the first $n$ spatial variables. Therefore, by Theorem \ref{Growth_thm4}, since $s\in(1/2,1)$, we have
		\begin{align*}
			I_1 \lesssim \sum_{i=1}^n \bigg\rvert \bigg\langle \partial_{x_i}P_s\ast T, \partial_{x_i}&\bigg( \frac{1}{|x|^{n+2s-2}} \bigg) \ast_n (\varphi\phi) \bigg\rangle \bigg\rvert\\
			&=\sum_{i=1}^n \big\rvert \big\langle \partial_{x_i}P_s\ast T,  \partial_{x_i}[\pazocal{I}^n_{2-2s}(\varphi\phi)] \big\rangle \big\rvert \lesssim \ell(R)^{n+1},
		\end{align*}
		and we are done.
	\end{proof}

	\begin{rem}
		Let us observe that in the particular case in which $T$ is compactly supported, we may simply convey that $Q:=\mathbb{R}^{n+1}$ and $\varphi\equiv 1$ so that we deduce
		\begin{equation*}
			|\langle T, \phi \rangle|\lesssim \ell(R)^{n+1},
		\end{equation*}
		for any $R$ $s$-parabolic cube and $\phi$ admissible function for $R$. Therefore, bearing in mind Lemma \ref{Cap_thm2}, if $E\subset \mathbb{R}^{n+1}$ is a compact set with $\pazocal{H}_{p_s}^{n+1}(E)=0$ and $T$ is a distribution supported on $E$ and satisfying the $\text{\normalfont{BMO}}_{p_s}$ estimates of Theorem \ref{Cap_thm1}, choosing $\ell_0:=\infty$ we get $T\equiv 0$.
	\end{rem}
	
	\begin{thm}
		\label{thm3.3}
		For any $s\in(1/2,1]$ and $E\subset \mathbb{R}^{n+1}$ compact set,
		\begin{equation*}
			\Gamma_{\Theta^s,\ast}(E) \approx \pazocal{H}^{n+1}_{\infty,p_s}(E).
		\end{equation*}
	\end{thm}
	\begin{proof}
		Let us first prove
		\begin{equation}
			\label{eq1.4.2}
			\Gamma_{\Theta^s,\ast}(E) \lesssim \pazocal{H}^{n+1}_{\infty,p_s}(E).
		\end{equation}
		Proceed by fixing $\varepsilon>0$ and $\{A_k\}_k$ a collection of sets in $\mathbb{R}^{n+1}$ that cover $E$ such that
		\begin{equation*}
			\sum_{k=1}^\infty \text{diam}_{p_s}(A_k)^{n+1}\leq \pazocal{H}^{n+1}_{\infty,p_s}(E)+\varepsilon.
		\end{equation*}
		Now, for each $k$ let $Q_k$ an open $s$-parabolic cube centered at some point $a_k\in A_k$ with side length $\ell(Q_k)=\text{diam}_{p_s}(A_k)$, so that $E\subset \bigcup_kQ_k$. Apply the compactness of $E$ and \cite[Lemma 3.1]{HP} to consider $\{\varphi_k\}_{k=1}^N$ a collection of smooth functions satisfying, for each $k$: $0\leq \varphi_k\leq 1$, $\text{supp}(\varphi_k)\subset 2Q_k$, $\sum_{k=1}^N\varphi_k = 1$ in $\bigcup_{k=1}^N Q_k$ and also $\|\nabla_x\varphi_k\|_\infty\leq \ell(2Q_k)^{-1}$, $\|\partial_t\varphi_k\|\leq \ell(2Q_k)^{-2s}$. Hence, by Theorem \ref{Cap_thm1}, if $T$ is any distribution admissible for $\Gamma_{\Theta^s,\ast}(E)$,
		\begin{align*}
			|\langle T, 1 \rangle |=\bigg\rvert \sum_{k=1}^N \langle T, \varphi_k \rangle \bigg\rvert  \lesssim \sum_{k=1}^N \ell(2Q_k)^{n+1} \simeq \sum_{k=1}^N\text{diam}_{p_s}(A_k)^{n+1}\leq \pazocal{H}^{n+1}_{\infty,p_s}(E)+\varepsilon.
		\end{align*}
		Since this holds for any $T$ and $\varepsilon>0$ can be arbitrarily small, \eqref{eq1.4.2} follows.
        
		For the lower bound we will apply (an $s$-parabolic version of) Frostman's lemma \cite[Theorem 8.8]{Mat}, which can be proved using an $s$-parabolic dyadic lattice, as it is presented in the proof of \cite[Lemma 5.1]{MaPrTo}. Assume then $\pazocal{H}_{\infty,p_s}^{n+1}(E)>0$ and consider a non trivial positive Borel regular measure $\mu$ supported on $E$ with $\mu(E)\geq c\pazocal{H}_{\infty,p_s}^{n+1}(E)$ and $\mu(B(\ox,r))\leq r^{n+1}$ for all $\ox\in\mathbb{R}^{n+1}$, $r>0$. If we prove that
		\begin{equation*}
			\|\nabla_x P_s\ast \mu\|_{\ast,p_s}\lesssim 1\hspace{0.5cm} \text{and}\hspace{0.5cm} \|\partial^{\frac{1}{2s}}_t P_s\ast \mu\|_{\ast,p_s}\lesssim 1,
		\end{equation*}
		we will be done, since this will imply $\Gamma_{\Theta^s,\ast}(E)\gtrsim \langle \mu, 1 \rangle = \mu(E) \gtrsim \pazocal{H}_{\infty,p_s}^{n+1}$. But by Lemma \ref{Pos_lem2} we already have $\|\partial^{\frac{1}{2s}}_t P_s\ast \mu\|_{\ast,p_s}\lesssim 1$, so we are only left with the $\text{\normalfont{BMO}}_{p_s}$ norm of $\nabla_x P_s\ast \mu$. Thus, let us fix an $s$-parabolic ball $B(\ox_0,r)$ and consider the characteristic function $\chi_{2B}$ associated to $2B$. Denote also $\chi_{2B^c}=1-\chi_{2B}$. In this setting, we pick
		\begin{equation*}
			c_B:=\nabla_xP_s\ast (\chi_{2B^c}\mu)(\ox_0).  
		\end{equation*}
		Using Theorem \ref{C-Z_thm2} it easily follows that  this last expression is well-defined. Let us now estimate $\| \nabla_xP_s\ast \mu\|_{\ast,p_s}$,
		\begin{align*}
			\frac{1}{|B|}\int_B &|\nabla_xP_s\ast \mu(\oy)-c_B|\dd \oy\\
			&\leq \frac{1}{|B|}\int_B\bigg(\int_{2B}|\nabla_x P_s(\oy-\oz)|\text{d}\mu(\oz)\bigg)\dd \oy\\
			&\hspace{1.4cm}+\frac{1}{|B|}\int_B\bigg( \int_{\mathbb{R}^{n+1}\setminus{2B}}|\nabla_x P_s(\oy-\oz)-\nabla_x P_s(\ox_0-\oz)| \text{d}\mu(\oz)\bigg) \dd \oy =:I_1+I_2.
		\end{align*}
		To deal with $I_1$ we first notice that by Theorem \ref{C-Z_thm2} and Tonelli's theorem we have
		\begin{equation*}
			I_1\lesssim \frac{1}{|B|}\int_{2B}\bigg( \int_{B}\frac{1}{|\oy-\oz|_{p_s}^{n+1}}\dd \oy \bigg)\text{d}\mu(\oz).
		\end{equation*}
		Writing $B=B_0\times I_0\subset \mathbb{R}^n\times \mathbb{R}$, $\oy=(y,t)$, $\oz=(z,u)$ and choosing $0<\varepsilon<2s-1$, integration in polar coordinates yields
		\begin{align*}
			I_1\lesssim \frac{1}{|B|}\int_{2B}\bigg( \int_{B_0}\frac{\dd y}{|y-z|^{n-\varepsilon}}\int_{I_0}\frac{dt}{|t-u|^{\frac{1+\varepsilon}{2s}}} \bigg)\text{d}\mu(\oz) \lesssim \frac{1}{|B|}\big( r^{\varepsilon}(r^{2s})^{1-\frac{1+\varepsilon}{2s}} \big)\mu(2B)\lesssim 1.
		\end{align*}
		Regarding $I_2$, we name $\ox:=\ox_0-\oz$ and $\ox':=\oy-\oz$, and observe that $|\ox-\ox'|_{p_s}\leq |\ox|_{p_s}/2$. Hence, we apply the fourth estimate in Theorem \ref{C-Z_thm2} with $2\zeta=1$ since $s>1/2$, and obtain 
		\begin{align*}
			I_2 &\lesssim \frac{1}{|B|}\int_B\bigg( \int_{\mathbb{R}^{n+1}\setminus{2B}}\frac{|\oy-\ox_0|_{p_s}}{|\oz-\ox_0|_{p_s}^{n+2}} \text{d}\mu(\oz)\bigg) \dd \oy \leq  r \int_{\mathbb{R}^{n+1}\setminus{2B}}  \frac{\text{d}\mu(\oz)}{|\oz-\ox_0|_{p_s}^{n+2}}\\
			&= r^{2\zeta}\sum_{j=1}^{\infty}\int_{2^{j+1}B\setminus{2^jB}}\frac{\text{d}\mu(\oz)}{|\oz-\ox_0|_{p_s}^{n+2}}\lesssim r\sum_{j=1}^{\infty}\frac{(2^{j+1}r)^{n+1}}{(2^jr)^{n+2}}\lesssim\sum_{j=1}^\infty\frac{1}{2^{j}} \lesssim 1,
		\end{align*}
		and we are done. 
	\end{proof}
	
	\begin{thm}
		\label{thm3.4}
		Let $s\in(1/2,1]$. A compact set $E\subset \mathbb{R}^{n+1}$ is removable for $s$-caloric functions with $\text{\normalfont{BMO}}_{p_s}$-$(1,\frac{1}{2s})$-derivatives if and only if $\Gamma_{\Theta^s,*}(E)=0$.
	\end{thm}
	
	\begin{proof}
		Fix $E\subset \mathbb{R}^{n+1}$ compact set and begin by assuming that is removable. Now pick $T$ admissible for $\Gamma_{\Theta^s,*}(E)$ and observe that defining $f:=P_s\ast T$, we have $\|\nabla_xf\|_{\ast,p_s}<\infty$, $\|\partial_t^{\frac{1}{2s}}f\|_{\ast,p_s}<\infty$ and $\Theta^s f = 0$ on $\mathbb{R}^{n+1}\setminus{E}$. So by hypothesis $\Theta^s f = 0$ in $\mathbb{R}^{n+1}$ and therefore $T\equiv 0$. Since $T$ was an arbitrary admissible distribution for $\Gamma_{\Theta^s,*}(E)$, we deduce that $\Gamma_{\Theta^s,*}(E)=0$.
        
		Let us now assume $\Gamma_{\Theta^s,*}(E)=0$ and prove the removability of $E$. Notice that by Theorem \ref{thm3.3} we get $\pazocal{H}_{\infty,p_s}^{n+1}(E)=0$ and thus, by \cite[Lemma 4.6]{Mat}, we have $\pazocal{H}_{p_s}^{n+1}(E)=0$. With this in mind, fix $\Omega\supset E$ any open set and $f:\mathbb{R}^{n+1}\to \mathbb{R}$ any function with $\|\nabla_x f\|_{\ast,p_s}<\infty$, $\|\partial_t^{\frac{1}{2s}}f\|_{\ast,p_s}<\infty$ and $\Theta^s f = 0$ on $\Omega\setminus{E}$. We will assume $\Theta^s f \neq 0$ in $\Omega$ and reach a contradiction. Define the distribution
		\begin{equation*}
			T:=\frac{\Theta^s f}{\|\nabla_x f\|_{\ast,p_s}+\|\partial_t^{\frac{1}{2s}}f\|_{\ast,p_s}},
		\end{equation*}
		which is such that $\|\nabla_x P_s \ast T\|_{\ast,p_s}\leq 1$, $\|\partial_t^{\frac{1}{2s}} P_s \ast T\|_{\ast,p_s}\leq 1$ and $\text{supp}(T)\subset E \cup \Omega^c$. Since $T\neq 0$ in $\Omega$, there exists $Q$ $s$-parabolic cube with $4Q\subset \Omega$ so that $T\neq 0$ in $Q$. Observe that $Q\cap E \neq \varnothing$. Then, by definition, there is $\varphi$ test function supported on $Q$ with $\langle T, \varphi \rangle>0$. Consider
		\begin{equation*}
			\widetilde{\varphi}:=\frac{\varphi}{\|\varphi\|_\infty + \ell(Q)\|\nabla_x\varphi\|_\infty+\ell(Q)^{2s}\|\partial_t\varphi\|_\infty+\ell(Q)^2\|\Delta \varphi\|_\infty},
		\end{equation*}
		so that $\widetilde{\varphi}$ is admissible for $Q$. Apply Theorem \ref{Cap_thm1} to deduce that $\widetilde{\varphi} T$ has upper $s$-parabolic growth of degree $n+1$ for cubes $R$ with $\ell(R)\leq \ell(Q)$. Apply Lemma \ref{Cap_thm2} to $\widetilde{\varphi}T$ with the compact set $\overline{Q}\cap E$, $\ell_0:=\ell(Q)$ and $d:=n+1$. Then,
		\begin{equation*}
			|\langle \widetilde{\varphi}T, \psi \rangle| = 0, \qquad \forall \psi\in \pazocal{C}^\infty_c(\mathbb{R}^{n+1}),
		\end{equation*}
		since $\pazocal{H}_{p_s}^{n+1}(\overline{Q}\cap E)=0$. This would imply $\widetilde{\varphi}T\equiv 0$, which is impossible, since $\langle \varphi, T \rangle >0$. Therefore $\Theta^s f = 0$ in $\Omega$, and by the arbitrariness of $\Omega$ and $f$ we are done.
	\end{proof}
	
	\subsection{The capacity \mathinhead{\Gamma_{\Theta^s,\alpha}}{}}
	\label{subsec2.4.2}
	
	We shall now present an $s$-parabolic $\text{Lip}_\alpha$ variant of the caloric capacity presented above.
	
	\begin{defn}
		Given $s\in(1/2,1]$, $\alpha\in(0,1)$ and $E\subset \mathbb{R}^{n+1}$ compact set, define its \textit{$\text{Lip}_{\alpha, p_s}$-caloric capacity} as
		\begin{equation*}
			\Gamma_{\Theta^s,\alpha}(E):=\sup  |\langle T, 1 \rangle| ,
		\end{equation*}
		where the supremum is taken among all distributions $T$ with $\text{supp}(T)\subset E$ and satisfying
		\begin{equation*}
			\|\partial_{x_i} P_s\ast T\|_{\text{\normalfont{Lip}}_\alpha, p_s} \leq 1, \;\; \forall i=1,\ldots,n, \hspace{0.75cm} \|\partial^{\frac{1}{2s}}_tP_s\ast T\|_{\text{Lip}_{\alpha, p_s}}\leq 1.
		\end{equation*}
		Such distributions will be called \textit{admissible for $\Gamma_{\Theta^s,\alpha}(E)$}.
	\end{defn}
	
	\begin{defn}
		A compact set $E\subset \mathbb{R}^{n+1}$ is said to be \textit{removable for $s$-caloric functions with $\text{Lip}_{\alpha, p_s}$-$(1,\frac{1}{2s})$-derivatives} if for any open subset $\Omega\subset \mathbb{R}^{n+1}$, any function $f:\mathbb{R}^{n+1}\to \mathbb{R}$ with
		\begin{equation*}
			\|\nabla_x f\|_{\text{Lip}_{\alpha, p_s}} < \infty, \hspace{0.75cm} \|\partial_t^{\frac{1}{2s}}f\|_{\text{Lip}_{\alpha, p_s}}<\infty,
		\end{equation*}
		satisfying the $\Theta^s$-equation in $\Omega\setminus{E}$, also satisfies the previous equation in the whole $\Omega$.
	\end{defn}
	
	As in the $s$-parabolic BMO case, if $T$ is a compactly supported distribution satisfying the required normalization conditions, $T$ will present upper $s$-parabolic growth of degree $n+1+\alpha$. In fact, the following result holds:
	\begin{thm}
		\label{thm2.4.5}
		Let $s\in(1/2,1]$, $\alpha\in(0,2s-1)$ and $T$ be a distribution in $\mathbb{R}^{n+1}$ with
		\begin{equation*}
			\|\partial_{x_i} P_s\ast T\|_{\text{\normalfont{Lip}}_\alpha, p_s} \leq 1, \;\; \forall i=1,\ldots,n, \hspace{0.75cm} \|\partial^{\frac{1}{2s}}_tP_s\ast T\|_{\text{\normalfont{Lip}}_{\alpha, p_s}}\leq 1.
		\end{equation*}
		Let $Q$ be a fixed $s$-parabolic cube and $\varphi$ admissible for $Q$. Then, if $R$ is any $s$-parabolic cube with $\ell(R)\leq \ell(Q)$ and $\phi$ is admissible for $R$, we have $|\langle \varphi T, \phi \rangle|\lesssim_\alpha \ell(R)^{n+1+\alpha}$.
	\end{thm}
	\begin{proof}
		Let $T$, $Q$ and $\varphi$ be as above. Let us also consider $R$ $s$-parabolic cube with $\ell(R)\leq \ell(Q)$ and $R\cap Q\neq \varnothing$ and $\phi$ admissible function for $R$. We proceed as in the proof of Theorem \ref{Cap_thm1} to obtain
		\begin{equation*}
			|\langle \varphi T, \phi\rangle|\leq |\langle (-\Delta)^sP_s\ast T, \varphi\phi\rangle|+|\langle P_s\ast T, \partial_t(\varphi\phi)\rangle|=:I_1+I_2.
		\end{equation*}
		Regarding $I_2$, we now define define $\beta:=1-\frac{1}{2s}$ and observe that $2s\beta=2s-1>\alpha$, so applying Theorem \ref{Growth_thm1} we get $I_2\lesssim_\alpha \ell(R)^{n+1+\alpha}$. The study of $I_1$ is also analogous to that done in Theorem \ref{Cap_thm1}. The case $s=1$ follows in exactly the same way by Theorem \ref{Growth_thm2}, and if $s\in(1/2,1)$ we also have
		\begin{align*}
			I_1 \lesssim \sum_{i=1}^n \big\rvert \big\langle \partial_{x_i}P_s\ast T,  \partial_{x_i}[\pazocal{I}^n_{2-2s}(\varphi\phi)] \big\rangle \big\rvert.
		\end{align*}
		So by Theorem \ref{Growth_thm4} and condition $\alpha<2s-1$ we deduce the desired result.
	\end{proof}
	
	\begin{thm}
		\label{thm2.4.6}
		For any $s\in(1/2,1]$, $\alpha\in(0,2s-1)$ and $E\subset \mathbb{R}^{n+1}$ compact set,
		\begin{equation*}
			\Gamma_{\Theta^s,\alpha}(E) \approx_{\alpha} \pazocal{H}^{n+1+\alpha}_{\infty,p_s}(E).
		\end{equation*}
	\end{thm}
	\begin{proof}
		For the upper bound we proceed analogously as we have done in the proof of Theorem \ref{thm3.3}, using now the growth restriction given by Theorem \ref{thm2.4.5}. So we focus on the lower bound, which will also rely on Frostman's lemma. Assume then $\pazocal{H}_{\infty,p_s}^{n+1+\alpha}(E)>0$ and consider a non trivial positive Borel measure $\mu$ supported on $E$ with $\mu(E)\geq c\pazocal{H}_{\infty,p_s}^{n+1+\alpha}(E)$ and $\mu(B(\ox,r))\leq r^{n+1+\alpha}$ for all $\ox\in\mathbb{R}^{n+1}$, $r>0$. It is enough to check
		\begin{equation*}
			\|\partial_{x_i} P_s\ast \mu\|_{\text{Lip}_\alpha, p_s} \lesssim_\alpha 1, \;\; \forall i=1,\ldots,n \hspace{0.95cm} \text{and} \hspace{0.95cm} \|\partial^{\frac{1}{2s}}_tP_s\ast \mu\|_{\text{\normalfont{Lip}}_\alpha, p_s}\lesssim_\alpha 1.
		\end{equation*}
		Notice that the right inequality follows directly from Lemma \ref{lem2.3.3} with $\beta:=\frac{1}{2s}$, so we just focus on controlling the $s$-parabolic $\text{Lip}_\alpha$ seminorm of the spatial derivatives of $P_s\ast \mu$. Fix $i=1,\ldots, n$ and choose any $\ox,\oy\in\mathbb{R}^{n+1}$ with $\ox\neq \oy$. Consider the following partition
		\begin{align*}
			R_1:&= \big\{\oz \;:\;|\ox-\oy|_{p_s}\leq |\ox-\oz|_{p_s}/2\big\}\cup\big\{z \;:\;|\oy-\ox|_{p_s}\leq |\oy-\oz|_{p_s}/2\big\},\\
			R_2:= \mathbb{R}^{n+1}\setminus{R_1}&=\big\{z \;:\;|\ox-\oy|_{p_s}> |\ox-\oz|_{p_s}/2\big\}\cap\big\{z \;:\;|\oy-\ox|_{p_s}> |\oy-\oz|_{p_s}/2\big\},
		\end{align*}
		with their corresponding characteristic functions $\chi_1,\chi_2$ respectively. This way, we have
		\begin{align*}
			\frac{|\partial_{x_i} P_s\ast \mu(\ox)-\partial_{x_i} P_s\ast \mu(\oy)|}{|\ox-\oy|_{p_s}^\alpha}\\
			&\hspace{-2.5cm}\leq \frac{1}{|\ox-\oy|_{p_s}^\alpha}\int_{ |\ox-\oy|_{p_s}\leq |\ox-\oz|_{p_s}/2}|\partial_{x_i} P_s(\ox-\oz)-\partial_{x_i} P_s(\oy-\oz)|\text{d}\mu(\oz)\\
			&\hspace{-1.5cm} +\frac{1}{|\ox-\oy|_{p_s}^\alpha}\int_{ |\oy-\ox|_{p_s}\leq |\oy-\oz|_{p_s}/2}|\partial_{x_i} P_s(\ox-\oz)-\partial_{x_i} P_s(\oy-\oz)|\text{d}\mu(\oz)\\
			&\hspace{-1.5cm} +\frac{1}{|\ox-\oy|_{p_s}^\alpha}\int_{R_2}|\partial_{x_i} P_s(\ox-\oz)-\partial_{x_i} P_s(\oy-\oz)|\text{d}\mu(\oz)=:I_1+I_2+I_{3}.
		\end{align*}
		Regarding $I_1$, apply the fourth estimate of Theorem \ref{C-Z_thm2} to obtain
		\begin{align*}
			I_1&\lesssim\frac{1}{|\ox-\oy|_{p_s}^\alpha}\int_{ |\ox-\oy|_{p_s}\leq |\ox-\oz|_{p_s}/2 } \frac{|\ox-\oy|_{p_s}}{|\ox-\oz|_{p_s}^{n+2}}\text{d}\mu(\oz).
		\end{align*}
		Split the previous domain of integration into the $s$-parabolic annuli 
		\begin{equation*}
			A_j:=2^{j+1}B\big(\ox, |\ox-\oy|_{p_s}\big)\setminus{2^{j}B\big(\ox, |\ox-\oy|_{p_s}\big)}, \hspace{0.5cm} \text{for} \hspace{0.5cm} j\geq 1,    
		\end{equation*}
		and use that $\mu$ has upper parabolic growth of degree $n+1+\alpha$ to deduce
		\begin{align*}
			I_1&\lesssim \frac{1}{|\ox-\oy|_{p_s}^{\alpha-1}}\sum_{j=1}^\infty\int_{A_j} \frac{\text{d}\mu(\oz)}{|\ox-\oz|_{p_s}^{n+2}}\lesssim  \frac{1}{|\ox-\oy|_{p_s}^{\alpha-1}}\sum_{j=1}^\infty \frac{(2^{j+1}|\ox-\oy|_{p_s})^{n+1+\alpha}}{(2^{j}|\ox-\oy|_{p_s})^{n+2}}\\
			&\hspace{10cm}\lesssim \sum_{j=1}^\infty\frac{1}{2^{(1-\alpha)j}}\lesssim_\alpha 1.
		\end{align*}
		The study of $I_2$ is analogous interchanging the roles of $\ox$ and $\oy$. Finally, for $I_{3}$, we apply the first estimate of Theorem \ref{C-Z_thm2} so that
		\begin{align*}
			I_{3}&\lesssim \frac{1}{|\ox-\oy|_{p_s}^\alpha}\int_{R_2}\frac{\text{d}\mu(\oz)}{|\ox-\oz|_{p_s}^{n+1}}+\frac{1}{|\ox-\oy|_{p_s}^\alpha}\int_{R_2}\frac{\text{d}\mu(\oz)}{|\oy-\oz|_{p_s}^{n+1}}\\
			&\leq \frac{1}{|\ox-\oy|_{p_s}^\alpha}\int_{|\ox-\oy|_{p_s}>|\ox-\oz|_{p_s}/2}\frac{\text{d}\mu(\oz)}{|\ox-\oz|_{p_s}^{n+1}}+\frac{1}{|\ox-\oy|_{p_s}^\alpha}\int_{|\oy-\ox|_{p_s}>|\oy-\oz|_{p_s}/2}\frac{\text{d}\mu(\oz)}{|\oy-\oz|_{p_s}^{n+1}}\\
			&=:I_{31}+I_{32}.
		\end{align*}
		Concerning $I_{31}$, split the domain of integration into the (decreasing) $s$-parabolic annuli
		\begin{equation*}
			\widetilde{A}_j:=2^{-j}B\big(\ox, |\ox-\oy|_{p_s}\big)\setminus{2^{-j-1}B\big(\ox, |\ox-\oy|_{p_s}\big)}, \hspace{0.5cm} \text{for} \hspace{0.5cm} j\geq -1.    
		\end{equation*}
		Thus, in this case we have
		\begin{align*}
			I_{31}&\lesssim \frac{1}{|\ox-\oy|_{p_s}^{\alpha}}\sum_{j=-1}^\infty\int_{\widetilde{A}_j} \frac{\text{d}\mu(\oz)}{|\ox-\oz|_{p_s}^{n+1}}\lesssim \frac{1}{|\ox-\oy|_{p_s}^{\alpha}}\sum_{j=-1}^\infty \frac{(2^{-j}|\ox-\oy|_{p_s})^{n+1+\alpha}}{(2^{-j-1}|\ox-\oy|_{p_s})^{n+1}}\\
			&\hspace{10cm}\lesssim \sum_{j=-1}^\infty\frac{1}{2^{\alpha j}}\lesssim_\alpha 1.
		\end{align*}
		On the other hand, for $I_{32}$ we apply the same reasoning but using the partition given by
		\begin{equation*}
			\widetilde{A}_j':=2^{-j}B_p\big(\oy, |\oy-\ox|_{p_s}\big)\setminus{2^{-j-1}B_p\big(\oy, |\oy-\ox|_{p_s}\big)}, \hspace{0.5cm} \text{for} \hspace{0.5cm} j\geq -1,    
		\end{equation*}
		yielding also $I_{32}\lesssim 1$. Combining the estimates obtained for $I_1, I_2$ and $I_{3}$ we deduce 
		\begin{equation*}
			\frac{|\partial_{x_i} P_s\ast \mu(\ox)-\partial_{x_i} P_s\ast \mu(\oy)|}{|\ox-\oy|_{p_s}^{\alpha}}\lesssim_\alpha 1,
		\end{equation*}
		and since the (different) points $\ox$ and $\oy$ were arbitrarily chosen, we deduce the desired $s$-parabolic $\text{Lip}_\alpha$ condition.
	\end{proof}
	
	\begin{thm}
		Let $s\in(1/2,1]$ and $\alpha\in(0,2s-1)$. A compact set $E\subset \mathbb{R}^{n+1}$ is removable for $s$-caloric functions with $\text{\normalfont{Lip}}_{\alpha,p_s}$-$(1,\frac{1}{2s})$-derivatives if and only if $\Gamma_{\Theta^s,\alpha}(E)=0$.
	\end{thm}
	
	\begin{proof}
		The proof is completely analogous to that of Theorem \ref{thm3.4}, now using Theorems \ref{thm2.4.5} and \ref{thm2.4.6}, as well as Lemma \ref{Cap_thm2} with $d:=n+1+\alpha$.
	\end{proof}
	
	\subsection{The capacity \mathinhead{\gamma_{\Theta^s,\ast}^{\sigma}}{}}
	\label{subsec2.4.3}
	
	Now, we shall present the $\text{BMO}_{p_s}$ variant of the capacities presented in \cite[\textsection 4 \& \textsection 7]{MaPr}. To be precise, in the aforementioned reference, Mateu and Prat work with the normalization conditions
	\begin{equation*}
		\|(-\Delta)^{s-\frac{1}{2}}P_s\ast T\|_\infty\leq 1, \qquad \|\partial_t^{1-\frac{1}{2s}}P_s\ast T\|_{\ast,p_s}\leq 1,
	\end{equation*}
	allowing $s\in[1/2,1)$. In our case we will deal with its $s$-parabolic BMO variant and we define it more generally as follows:
	
	\begin{defn}
		Given $s\in(0,1]$, $\sigma\in[0,s)$ and $E\subset \mathbb{R}^{n+1}$ compact set, define its \textit{$\Delta^{\sigma}$-$\text{\normalfont{BMO}}_{p_s}$-caloric capacity} as
		\begin{equation*}
			\gamma_{\Theta^s,\ast}^{\sigma}(E):=\sup  |\langle T, 1 \rangle| ,
		\end{equation*}
		where the supremum is taken among all distributions $T$ with $\text{supp}(T)\subset E$ and satisfying
		\begin{equation*}
			\|(-\Delta)^{\sigma} P_s\ast T\|_{\ast, p_s} \leq 1, \hspace{0.75cm} \|\partial^{\sigma/s}_tP_s\ast T\|_{\ast, p_s}\leq 1.
		\end{equation*}
		Such distributions will be called \textit{admissible for $\gamma_{\Theta^s,\ast}^{\sigma}(E)$}.
	\end{defn}
	
	\begin{defn}
		Let $s\in(0,1]$ and $\sigma\in[0,s)$. A compact set $E\subset \mathbb{R}^{n+1}$ is said to be \textit{removable for $s$-caloric functions with $\text{\normalfont{BMO}}_{p_s}$-$(\sigma,\sigma/s)$-Laplacian} if for any open subset $\Omega\subset \mathbb{R}^{n+1}$, any function $f:\mathbb{R}^{n+1}\to \mathbb{R}$ with
		\begin{equation*}
			\|(-\Delta)^{\sigma} f\|_{\ast,p_s} < \infty, \hspace{0.75cm} \|\partial_t^{\sigma/s}f\|_{\ast, p_s}<\infty,
		\end{equation*}
		satisfying the $\Theta^s$-equation in $\Omega\setminus{E}$, also satisfies the previous equation in the whole $\Omega$. If $\sigma=0$, we will also say that $E$ is \textit{removable for} $\text{\normalfont{BMO}}_{p_s}$ $s$\textit{-caloric functions}.
	\end{defn}
	
	Firstly, we shall prove that if $T$ is a compactly supported distribution satisfying the expected normalization conditions, then $T$ has upper $s$-parabolic growth of degree $n+2s-2\sigma$. In fact, we prove a stronger result:
	\begin{thm}
		\label{thm2.4.8}
		Let $s\in(0,1]$, $\sigma\in[0,s)$ and $T$ be a distribution in $\mathbb{R}^{n+1}$ with
		\begin{equation*}
			\|(-\Delta)^{\sigma} P_s\ast T\|_{\ast, p_s} \leq 1, \hspace{0.75cm} \|\partial^{\sigma/s}_tP_s\ast T\|_{\ast, p_s}\leq 1.
		\end{equation*}
		Let $Q$ be a fixed $s$-parabolic cube and $\varphi$ an admissible function for $Q$. Then, if $R$ is any $s$-parabolic cube with $\ell(R)\leq \ell(Q)$ and $\phi$ is admissible for $R$, we have $|\langle \varphi T, \phi \rangle|\lesssim_\sigma \ell(R)^{n+2\sigma}$.
	\end{thm}
	\begin{proof}
		Let $T$, $Q$ and $\varphi$ be as above, as well as $R$ $s$-parabolic cube with $\ell(R)\leq \ell(Q)$ and $R\cap Q\neq \varnothing$, and $\phi$ admissible function for $R$. We already know, in light of the proof of Theorem \ref{Cap_thm1},
		\begin{equation*}
			|\langle \varphi T, \phi\rangle|\leq |\langle (-\Delta)^sP_s\ast T, \varphi\phi\rangle|+|\langle P_s\ast T, \partial_t(\varphi\phi)\rangle|=:I_1+I_2.
		\end{equation*}
		For $I_1$, simply apply Theorem \ref{thm0.3} with $\beta:=s-\sigma$ so that
		\begin{align*}
			I_1 = |\langle (-\Delta)^{\sigma}P_s\ast T, (-\Delta)^{s-\sigma}(\varphi\phi)\rangle|\lesssim_\sigma \ell(R)^{n+2\sigma}
		\end{align*}
		Regarding $I_2$, if $\sigma>0$, observe that defining $\beta:=1-\sigma/s\in(0,1)$ we get
		\begin{equation*}
			\partial_t(\varphi\phi) \simeq_\sigma \, \partial_t^{1-\beta}\Big( \partial_t(\varphi\phi)\ast_t |t|^{-\beta} \Big),
		\end{equation*}
		so by Theorem \ref{Growth_thm1} we are done. If $\sigma = 0$, we simply have
		\begin{align*}
			I_2 &= |\langle P_s\ast T - (P_s\ast T)_R, \partial_t(\varphi\phi)\rangle| \leq \int_{Q\cap R}\big\rvert P_s\ast T(\ox)-(P\ast T)_{R}\big\rvert \big\rvert\partial_t(\varphi\phi)(\ox)\big\rvert \dd \ox\\
			&\leq \ell(R)^{-2s}\int_{R}\big\rvert P\ast T(\ox)-(P\ast T)_{R}\big\rvert \dd \ox \leq \ell(R)^{-2s} \ell(R)^{n+2s} \|P_s\ast T\|_{\ast,p_s}\leq \ell(R)^n.
		\end{align*}
	\end{proof}
	\begin{thm}
		\label{thm2.4.9}
		For any $s\in(0,1]$, $\sigma\in[0,s)$ and $E\subset \mathbb{R}^{n+1}$ compact set,
		\begin{equation*}
			\gamma_{\Theta^s,\ast}^{\sigma}(E) \approx_{\sigma} \pazocal{H}^{n+2\sigma}_{\infty,p_s}(E).
		\end{equation*}
	\end{thm}
	\begin{proof}
		Again, for the upper bound we proceed analogously as in the proof of Theorem \ref{thm3.3}, using now Theorem \ref{thm2.4.8}. For the lower bound, we apply Frostman's lemma. Assume then $\pazocal{H}_{\infty,p_s}^{n+2\sigma}(E)>0$ and consider a non trivial positive Borel measure $\mu$ supported on $E$ with $\mu(E)\geq c\pazocal{H}_{\infty,p_s}^{n+2\sigma}(E)$ and $\mu(B(\ox,r))\leq r^{n+2\sigma}$ for all $\ox\in\mathbb{R}^{n+1}$, $r>0$. We have to prove
		\begin{equation*}
			\|(-\Delta)^{\sigma} P_s\ast T\|_{\ast, p_s} \leq 1, \hspace{0.75cm} \|\partial^{\sigma/s}_tP_s\ast T\|_{\ast, p_s}\leq 1,
		\end{equation*}
		If $\sigma>0$, by Lemma \ref{Pos_lem2} with $\beta:=\sigma/s$ we already have $\|\partial^{\sigma/s}_t P_s\ast \mu\|_{\ast,p_s}\lesssim_\sigma 1$. So we are left to control the $\text{\normalfont{BMO}}_{p_s}$ norm of $(-\Delta)^{\sigma} P_s\ast \mu$ for $\sigma\in[0,s)$. Thus, let us fix an $s$-parabolic ball $B(\ox_0,r)$ and consider the characteristic function $\chi_{2B}$ associated to $2B$. Set also $\chi_{2B^c}=1-\chi_{2B}$. In this setting, we pick
		\begin{equation*}
			c_B:=(-\Delta)^{\sigma}P_s\ast (\chi_{2B^c}\mu)(\ox_0).  
		\end{equation*}
		Using Theorem \ref{lem2.3} it easily follows that  this last expression is well-defined. We estimate $\| (-\Delta)^{\sigma}P_s\ast \mu\|_{\ast,p_s}$ using the previous constant:
		\begin{align*}
			\frac{1}{|B|}&\int_B |(-\Delta)^{\sigma}P_s\ast \mu(\oy)-c_B|\dd \oy\\
			&\leq \frac{1}{|B|}\int_B\bigg(\int_{2B}|(-\Delta)^{\sigma} P_s(\oy-\oz)|\text{d}\mu(\oz)\bigg)\dd \oy\\
			&\hspace{1cm}+\frac{1}{|B|}\int_B\bigg( \int_{\mathbb{R}^{n+1}\setminus{2B}}|(-\Delta)^{\sigma} P_s(\oy-\oz)-(-\Delta)^{\sigma} P_s(\ox_0-\oz)| \text{d}\mu(\oz)\bigg) \dd \oy=:I_1+I_2.
		\end{align*}
		To deal with $I_1$, notice that by Theorem \ref{lem2.3}, choosing $0<\varepsilon<2(s-\sigma)$ and arguing as in Theorem \ref{thm3.3}  we have
		\begin{equation*}
			I_1\lesssim_\sigma \frac{1}{|B|}\int_{2B}\bigg( \int_{B}\frac{\dd \oy}{|\oy-\oz|_{p_s}^{n+2\sigma}} \bigg)\text{d}\mu(\oz)\lesssim \frac{1}{|B|}\big( r^{\varepsilon}(r^{2s})^{1-\frac{\varepsilon+2\sigma}{2s}} \big)\mu(2B)\lesssim 1
		\end{equation*}
		by the $n+2\sigma$ growth of $\mu$. Regarding $I_2$, notice that naming $\ox:=\ox_0-\oz$ and $\ox':=\oy-\oz$, we have $|\ox-\ox'|_{p_s}\leq |\ox|_{p_s}/2$ so we can apply the fifth estimate of Theorem \ref{lem2.3}, that implies 
		\begin{align*}
			I_2 &\lesssim_\sigma \frac{1}{|B|}\int_B\bigg( \int_{\mathbb{R}^{n+1}\setminus{2B}}\frac{|\oy-\ox_0|_{p_s}^{2\zeta}}{|\oz-\ox_0|_{p_s}^{n+2\sigma+2\zeta}} \text{d}\mu(\oz)\bigg) \dd \oy \leq  r^{2\zeta} \int_{\mathbb{R}^{n+1}\setminus{2B}}  \frac{\text{d}\mu(\oz)}{|\oz-\ox_0|_{p_s}^{n+2\sigma+2\zeta}}\lesssim_\sigma 1,
		\end{align*}
		again by the by the $n+2\sigma$ growth of $\mu$.
	\end{proof}
	\begin{thm}
		Let $s\in(0,1]$ and $\sigma\in[0,s)$. A compact set $E\subset \mathbb{R}^{n+1}$ is removable for $s$-caloric functions with $\text{\normalfont{BMO}}_{p_s}$-$(\sigma,\sigma/s)$-Laplacian if and only if $\gamma_{\Theta^s,\ast}^{\sigma}(E)=0$.
	\end{thm}
	\begin{proof}
		The proof is analogous to that of Theorem \ref{thm3.4}, applying Theorems \ref{thm2.4.8}, \ref{thm2.4.9} and Lemma \ref{Cap_thm2} with $d:=n+2\sigma$.
	\end{proof}
	
	\subsection{The capacity \mathinhead{\gamma_{\Theta^s,\alpha}^{\sigma}}{}}
	
	We define now a capacity with an $s$-parabolic $\text{Lip}_\alpha$ normalization condition.
	
	\begin{defn}
		Given $\alpha\in(0,1)$, $s\in(0,1]$, $\sigma\in[0,s)$ and $E\subset \mathbb{R}^{n+1}$ compact set, define its \textit{$\Delta^{\sigma}$-$\text{\normalfont{Lip}}_{\alpha,p_s}$-caloric capacity} as
		\begin{equation*}
			\gamma_{\Theta^s,\alpha}^{\sigma}(E):=\sup  |\langle T, 1 \rangle| ,
		\end{equation*}
		where the supremum is taken among all distributions $T$ with $\text{supp}(T)\subset E$ and satisfying
		\begin{equation*}
			\|(-\Delta)^{\sigma} P_s\ast T\|_{\text{Lip}_{\alpha}, p_s} \leq 1, \hspace{0.75cm} \|\partial^{\sigma/s}_tP_s\ast T\|_{\text{Lip}_{\alpha}, p_s}\leq 1.
		\end{equation*}
		Such distributions will be called \textit{admissible for $\gamma_{\Theta^s,\alpha}^{\sigma}(E)$}.
	\end{defn}
	
	\begin{defn}
		Let $\alpha\in(0,1)$, $s\in(0,1]$ and $\sigma\in[0,s)$. A compact set $E\subset \mathbb{R}^{n+1}$ is said to be \textit{removable for $s$-caloric functions with $\text{\normalfont{Lip}}_{\alpha,p_s}$-$(\sigma,\sigma/s)$-Laplacian} if for any open subset $\Omega\subset \mathbb{R}^{n+1}$, any function $f:\mathbb{R}^{n+1}\to \mathbb{R}$ with
		\begin{equation*}
			\|(-\Delta)^{\sigma} f\|_{\text{\normalfont{Lip}}_{\alpha},p_s} < \infty, \hspace{0.75cm} \|\partial_t^{\sigma/s}f\|_{\text{\normalfont{Lip}}_{\alpha}, p_s}<\infty,
		\end{equation*}
		satisfying the $\Theta^s$-equation in $\Omega\setminus{E}$, also satisfies the previous equation in the whole $\Omega$. If $\sigma=0$, we will also say that $E$ is \textit{removable for} $\text{\normalfont{Lip}}_{\alpha,p_s}$ $s$\textit{-caloric functions}.
	\end{defn}
	
	If $T$ is a compactly supported distribution satisfying the above properties, then $T$ presents upper $s$-parabolic growth of degree $n+2\sigma+\alpha$. As in \textsection\ref{subsec2.4.2}, the following result will only be valid for a certain range of values of $\alpha$, dependent on $s$ and $\sigma$.
	\begin{thm}
		\label{thm2.4.11}
		Let $s\in(0,1]$, $\sigma\in[0,s)$ and $\alpha\in(0,1)$ with $\alpha<2s-2\sigma$. Let $T$ be a distribution in $\mathbb{R}^{n+1}$ with
		\begin{equation*}
			\|(-\Delta)^{\sigma} P_s\ast T\|_{\text{\normalfont{Lip}}_\alpha, p_s} \leq 1, \hspace{0.75cm} \|\partial^{\sigma/s}_tP_s\ast T\|_{\text{\normalfont{Lip}}_\alpha, p_s}\leq 1.
		\end{equation*}
		Let $Q$ be a fixed $s$-parabolic cube and $\varphi$ an admissible function for $Q$. Then, if $R$ is any $s$-parabolic cube with $\ell(R)\leq \ell(Q)$ and $\phi$ is admissible for $R$, we have $|\langle \varphi T, \phi \rangle|\lesssim \ell(R)^{n+2\sigma+\alpha}$.
	\end{thm}
	\begin{proof}
		Let $T$, $Q$ and $\varphi$ be as above, as well as $R$ $s$-parabolic cube with $\ell(R)\leq \ell(Q)$ and $R\cap Q\neq \varnothing$, and $\phi$ admissible function for $R$. Again,
		\begin{equation*}
			|\langle \varphi T, \phi\rangle|\leq |\langle (-\Delta)^sP_s\ast T, \varphi\phi\rangle|+|\langle P_s\ast T, \partial_t(\varphi\phi)\rangle|=:I_1+I_2.
		\end{equation*}
		For $I_1$, simply apply Theorem \ref{thm0.3} with $\beta:=s-\sigma$ so that
		\begin{align*}
			I_1 = |\langle (-\Delta)^{\sigma}P_s\ast T, (-\Delta)^{s-\sigma}(\varphi\phi)\rangle|\lesssim_{\sigma,\alpha} \ell(R)^{n+2\sigma+\alpha}.
		\end{align*}
		Regarding $I_2$, if $\sigma>0$, we define $\beta:=1-\sigma/s\in(0,1)$ and apply Theorem \ref{Growth_thm1}. If $\sigma = 0$, let $\ox_R$ be the center of $R$ so that 
		\begin{align*}
			I_2 &= |\langle P_s\ast T - P_s\ast T(\ox_R), \partial_t(\varphi\phi)\rangle| \leq \ell(R)^{-2s}\int_{R}|\ox-\ox_R|^{\alpha}_{p_s} \dd \ox \lesssim \ell(R)^{n+\alpha}.
		\end{align*}
	\end{proof}
	\begin{thm}
		\label{thm2.4.12}
		Let $s\in(0,1]$, $\sigma\in[0,s)$ and $\alpha\in(0,1)$ with $\alpha<2s-2\sigma$. Then, for $E\subset \mathbb{R}^{n+1}$ compact set,
		\begin{equation*}
			\gamma_{\Theta^s,\alpha}^{\sigma}(E) \approx_{\sigma,\alpha} \pazocal{H}^{n+2\sigma+\alpha}_{\infty,p_s}(E).
		\end{equation*}
	\end{thm}
	\begin{proof}
		For the upper bound we argue again as in Theorem \ref{thm3.3}, using now Theorem \ref{thm2.4.11}. For the lower bound, assume $\pazocal{H}_{\infty,p_s}^{n+2\sigma+\alpha}(E)>0$ and apply Frostman's lemma to consider a non trivial positive Borel measure $\mu$ supported on $E$ with $\mu(E)\geq c\pazocal{H}_{\infty,p_s}^{n+2\sigma+\alpha}(E)$ and $\mu(B(\ox,r))\leq r^{n+2\sigma+\alpha}$ for all $\ox\in\mathbb{R}^{n+1}$, $r>0$. It suffices to verify
		\begin{equation*}
			\|(-\Delta)^{\sigma} P_s\ast T\|_{\text{\normalfont{Lip}}_\alpha, p_s} \leq 1, \hspace{0.75cm} \|\partial^{\sigma/s}_tP_s\ast T\|_{\text{\normalfont{Lip}}_\alpha, p_s}\leq 1.
		\end{equation*}
		If $\sigma>0$, by Lemma \ref{lem2.3.3} with $\beta:=\sigma/s$ we already have $\|\partial^{\sigma/s}_t P_s\ast \mu\|_{\text{\normalfont{Lip}}_\alpha,p_s}\lesssim_{\sigma,\alpha} 1$. So we are left to estimate $\|(-\Delta)^{\sigma} P_s\ast \mu\|_{\text{\normalfont{Lip}}_\alpha,p_s}$ for $\sigma\in[0,s)$, and we do it as in Theorem \ref{thm2.4.6}. Choose any $\ox,\oy\in\mathbb{R}^{n+1}$ with $\ox\neq \oy$ and consider the following partition of $\mathbb{R}^{n+1}$,
		\begin{align*}
			R_1:&= \big\{\oz \;:\;|\ox-\oy|_{p_s}\leq |\ox-\oz|_{p_s}/2\big\}\cup\big\{\oz \;:\;|\oy-\ox|_{p_s}\leq |\oy-\oz|_{p_s}/2\big\},\\
			R_2:= \mathbb{R}^{n+1}\setminus{R_1}&=\big\{\oz \;:\;|\ox-\oy|_{p_s}> |\ox-\oz|_{p_s}/2\big\}\cap\big\{\oz \;:\;|\oy-\ox|_{p_s}> |\oy-\oz|_{p_s}/2\big\},
		\end{align*}
		with their corresponding characteristic functions $\chi_1,\chi_2$ respectively.
		This way, we have
		\begin{align*}
			\frac{|(-\Delta)^{\sigma}P_s\ast \mu(\ox)-(-\Delta)^{\sigma}P_s\ast \mu(\oy)|}{|\ox-\oy|_{p_s}^\alpha} \\
			&\hspace{-4.75cm}\leq \frac{1}{|\ox-\oy|_{p_s}^\alpha}\int_{ |\ox-\oy|_{p_s}\leq |\ox-\oz|_{p_s}/2}|(-\Delta)^{\sigma}P_s(\ox-\oz)-(-\Delta)^{\sigma}P_s(\oy-\oz)|\text{d}\mu(\oz)\\
			&\hspace{-3.75cm}+\frac{1}{|\ox-\oy|_{p_s}^\alpha}\int_{ |\oy-\ox|_{p_s}\leq |\oy-\oz|_{p_s}/2}|(-\Delta)^{\sigma}P_s(\ox-\oz)-(-\Delta)^{\sigma}P_s(\oy-\oz)|\text{d}\mu(\oz)\\
			&\hspace{-3.75cm}+\frac{1}{|\ox-\oy|_{p_s}^\alpha}\int_{R_2}|(-\Delta)^{\sigma}P_s(\ox-\oz)-(-\Delta)^{\sigma}P_s(\oy-\oz)|\text{d}\mu(\oz) =:I_1+I_2+I_{3}.
		\end{align*}
		Regarding $I_1$, the fifth estimate of Lemma \ref{lem2.3} yields
		\begin{align*}
			I_1&\lesssim_\sigma\frac{1}{|\ox-\oy|_{p_s}^\alpha}\int_{ |\ox-\oy|_{p_s}\leq |\ox-\oz|_{p_s}/2 } \frac{|\ox-\oy|_{p_s}^{2\zeta}}{|\ox-\oz|_{p_s}^{n+2\sigma+2\zeta}}\text{d}\mu(\oz).
		\end{align*}
		Split the previous domain of integration into $s$-parabolic annuli centered at $\ox$ with exponentially increasing radii proportional to $|\ox-\oy|_{p_s}$, and deduce as in Theorem \ref{thm2.4.6} that $I_1\lesssim_{\sigma,\alpha} 1$, using now that $\mu$ has $n+2\sigma+\alpha$ growth. For $I_2$, we argue as in $I_1$ just interchanging the roles of $\ox$ and $\oy$. Finally, for $I_{3}$, the first estimate of Lemma \ref{lem2.3} yields
		\begin{align*}
			I_{3}&\leq \frac{1}{|\ox-\oy|_{p_s}^\alpha}\int_{R_2}\frac{\text{d}\mu(\oz)}{|\ox-\oz|_{p_s}^{n+2\sigma}}+\frac{1}{|\ox-\oy|_{p_s}^\alpha}\int_{R_2}\frac{\text{d}\mu(\oz)}{|\oy-\oz|_{p_s}^{n+2\sigma}}\\
			&\leq \frac{1}{|\ox-\oy|_{p_s}^\alpha}\bigg(\int_{|\ox-\oy|_{p_s}>|\ox-\oz|_{p_s}/2}\frac{\text{d}\mu(\oz)}{|\ox-\oz|_{p_s}^{n+2\sigma}}+\int_{|\oy-\ox|_{p_s}>|\oy-\oz|_{p_s}/2}\frac{\text{d}\mu(\oz)}{|\oy-\oz|_{p_s}^{n+2\sigma}}\bigg).
		\end{align*}
		Both of the above integrals can be dealt with by splitting the domain of integration into exponentially decreasing annuli, centered at $\ox$ and $\oy$ respectively, and using that $\mu$ has growth of degree strictly bigger than $n+2\sigma$. Thus, we obtain $I_{3}\lesssim_{\sigma,\alpha} 1$ and we are done
	\end{proof}
	\begin{thm}
		Let $s\in(0,1]$, $\sigma\in[0,s)$ and $\alpha\in(0,1)$ with $\alpha <2s-2\sigma$. A compact set $E\subset \mathbb{R}^{n+1}$ is removable for $s$-caloric functions with $\text{\normalfont{Lip}}_{\alpha,p_s}$-$(\sigma,\sigma/s)$-Laplacian if and only if $\gamma_{\Theta^s,\alpha}^{\sigma}(E)=0$.
	\end{thm}
	\begin{proof}
		The proof is analogous to that of Theorem \ref{thm3.4}, applying Theorems \ref{thm2.4.11}, \ref{thm2.4.12} and Lemma \ref{Cap_thm2} with $d:=n+2\sigma+\alpha$.
	\end{proof}

	\vspace{0.5cm}
	{\small
		\begin{tabular}{@{}l}
			\textsc{Joan\ Hernández,} \\ \textsc{Departament de Matem\`{a}tiques, Universitat Aut\`{o}noma de Barcelona,}\\
			\textsc{08193, Bellaterra (Barcelona), Catalonia.}\\
			{\it E-mail address}\,: \href{mailto:joan.hernandez@uab.cat}{\tt{joan.hernandez@uab.cat}}
		\end{tabular}
	}
	
	{\small
		\begin{tabular}{@{}l}
			\textsc{Joan\ Mateu,} \\ \textsc{Departament de Matem\`{a}tiques, Universitat Aut\`{o}noma de Barcelona,}\\
			\textsc{08193, Bellaterra (Barcelona), Catalonia.}\\
			{\it E-mail address}\,: \href{mailto:joan.mateu@uab.cat}{\tt{joan.mateu@uab.cat}}
		\end{tabular}
	}

	{\small
		\begin{tabular}{@{}l}
			\textsc{Laura\ Prat,} \\ \textsc{Departament de Matem\`{a}tiques, Universitat Aut\`{o}noma de Barcelona,}\\
			\textsc{08193, Bellaterra (Barcelona), Catalonia.}\\
			{\it E-mail address}\,: \href{mailto:laura.prat@uab.cat}{\tt{laura.prat@uab.cat}}
		\end{tabular}
	}  
\end{document}